\documentclass[a4paper,reqno,11pt]{amsart}
\usepackage{amsfonts,amsmath,amsthm,amssymb,stmaryrd}
\usepackage{mathrsfs}

\usepackage[left=1 in, right=1 in,top=1 in, bottom=1 in]{geometry}

\allowdisplaybreaks

\numberwithin{equation}{section}
\usepackage{color}

\newtheorem{theorem}{Theorem}[section]
\newtheorem{lemma}[theorem]{Lemma}
\newtheorem{proposition}[theorem]{Proposition}

\theoremstyle{definition}

\newcommand{\bR}{\mathbb R}

\newcommand{\Div}{\operatorname{div}}

\renewcommand{\epsilon}{\varepsilon}

\providecommand{\abs}[1]{\left\vert#1\right\vert}
\providecommand{\norm}[1]{\left\Vert#1\right\Vert}

\DeclareMathOperator{\diverge}{div}

\providecommand{\norm}[1]{\left\Vert#1\right\Vert}

\def\dt{\partial_t}

\def\na{\nabla}
\def\pa{\partial}
\def\intb{\int_{\Gamma}}
\def\intd{\int_{\Omega}}

\def\RRvert2{\right \vert\! \right\vert}
\def\Lvert3{\left \vert\!\left\vert\!\left\vert}
\def\Rvert3{\right \vert\!\right\vert\!\right\vert}

\def\dt{\partial_t}

\def\hal{\frac{1}{2}}

\def\ls{\lesssim}

\def\V{\mathcal{V}}
\def\Q{\mathcal{Q}}
\def\a{\mathcal{A}}

\def\fj1{\mathcal{J}^{-1}}

\def\X{\mathbb{X}}

\def\S{\mathbb{S}}

\def\bp{\bar\partial}

\def\ak{{\mathcal{A}^\kappa}}

\makeindex

\title[Free-boundary incompressible elastodynamics]{Local Well-posedness of Free-Boundary Incompressible Elastodynamics with Surface Tension via vanishing viscosity limit}

\author{Xumin Gu}
\address[X. Gu]{School of Mathematics\\ Shanghai University of Finance and Economics\\
Shanghai 200433, P.R.China}
\email{gu.xumin@shufe.edu.cn}

\author{Zhen Lei}
\address[Z. Lei]{School of Mathematical Sciences\\ Fudan University\\Shanghai 200433, P.R.China}
\email{zlei@fudan.edu.cn}



\begin{document}


\begin{abstract}
        In this paper, we consider the free boundary problem of incompressible elastodynamics, a coupling system of the Euler equations for the fluid motion with a transport equation for the deformation tensor. Under a natural force balance law on the free boundary with surface tension, we establish the well-posedness theory for a short time interval. Our method uses the vanishing viscosity limit by establishing a uniform a priori estimate with respect to the viscosity.  As a byproduct, the inviscid limit of the incompressible viscoelasticity (the system coupled with the Navier-Stokes equations) is also justified. We point out that the framework of establishing the well-posedness for water wave equations simply does not apply here, even for obtaining a priori estimates. Based on a crucial new observation about the inherent structure of the elastic term on the free boundary, we successfully develop a new estimate for the pressure and then manage to apply an induction method to control normal derivatives. This strategy also allows us to establish the vanishing viscosity limit in standard Sobolev spaces, rather than only in the conormal ones (cf. \cite{MasRou,Wang_15}). 

\end{abstract}

\maketitle

\section{Introduction}
The elastic effect is prevalent in complex fluids (e.g., liquid crystals), which results in many anomalous but important phenomena of fluids.  It can be described as a specific type of coupling between the transportation of the internal variable and the induced elastic stress. For incompressible neo-Hookean materials, the motion of a fluid can be described by the following elastodynamic system:
\begin{equation}
  \label{elastic}
\begin{cases}
\dt u  + u \cdot \nabla u + \nabla  p = \diverge (FF^\text{T})  &\text{in } \Omega (t),\\
\dt F + u \cdot \nabla F= \nabla u F&\text{in } \Omega (t),\\
\diverge u =0,\,\, \diverge (F^\text{T}) =0 &\text{in } \Omega (t).
\end{cases}
\end{equation}
In equations \eqref{elastic}, $u$ is the velocity field, $p$ is the scalar pressure, and $F=(F)_{ij}$ is the deformation tensor. $F^{\text{T}}=(F)_{ji}$ denotes the transpose of the matrix $F$, and $FF^{\text{T}}$ is the Cauchy-Green tensor for neo-Hookean elastic materials. $\Omega(t)$ denotes the domain occupied by the fluid at time $t$. Moreover, we adopt the following notation:
$$\diverge u=\partial_iu_i, \quad (\diverge F)_i=\partial_jF_{ij},\quad (\nabla u)_{ij}=\partial_j u_i, \quad
(\nabla uF)_{ij}=\partial_k u_iF_{kj}.$$
 Einstein's conventions over repeated indices are applied here and in what follows.

In this paper, we study a free boundary problem of \eqref{elastic} and prove its local well-posedness in two-dimensional spaces. We impose the following two boundary conditions on $\Gamma(t):=\pa\Omega(t)$: 
\begin{equation} \label{bdc}
\begin{cases}
V(\Gamma(t)) = u\cdot n \quad&\text{on } \Gamma(t),\\
 \left(-pI+(FF^{\text{T}}-I)\right)n=\sigma Hn\quad&\text{on }\Gamma(t),
\end{cases}
\end{equation}
where $n$ denotes the outward unit normal vector of $\Gamma(t)$, $H=-\nabla\cdot n$ is twice the mean curvature of $\Gamma(t)$, and $\sigma > 0$ is the surface tension coefficient. Here, the kinematic boundary condition \eqref{bdc}$_1$ states that the free boundary $\Gamma(t)$ moves with the fluid, where $V(\Gamma(t))$ denotes the normal velocity of $\Gamma(t)$. The dynamic boundary condition \eqref{bdc}$_2$ states the natural force balance law on the free boundary. Lastly, the initial conditions are given by
\begin{equation} \label{initial}
	(u, F)|_{t=0}=(u_0, F_0)\text{ on }\Omega(0)\text{ and } \Div u_0=0, \Div F_0^{\text{T}}=0.
\end{equation}

The study on the existence of solutions to the elastodynamic system has a long history, especially for the Cauchy problem. In $\bR^3$, the Hookean part of the incompressible elastodynamics is inherently degenerate and satisfies a null condition; and the global well-posedness of this system was established by Sideris and Thomases in \cite{ST_05, ST_06} with small initial data. For the two-dimensional case, the proof of the long-time existence for the elastodynamics is more difficult due to the weaker time decay rate. By observing the strong null structure of the system in Lagrangian coordinates, the second author \cite{Lei_16} proved the
global well-posedness using the energy method of Klainerman and Alinhac's ghost
weight approach. 

However, for the free boundary problem, the well-posedness theory, even the local version, seems far from being well understood. Only some a priori estimates and local well-posedness have been obtained under special boundary conditions (such as $p=0, F^\text{T}n=0$; see \cite{Hao_16, LWZ_2018, GW_18, Hu_19, Trak_18}). To the best of our knowledge, the problem of local existence under the force balance condition \eqref{bdc}$_2$ was left open. We also refer the reader to \cite{Chen_17,Chen_19} for the vortex sheet problem and \cite{CSW_20} for contact discontinuities in elastodynamics.

On the other hand, taking into account of the viscosity (for any fixed viscosity coefficent $\epsilon >0$), we have the incompressible viscoelasticity:
\begin{equation}
  \label{viscoelastic}
\begin{cases}
\dt u  + u \cdot \nabla u + \nabla  p-\epsilon\Delta u = \diverge (FF^{\text{T}})  &\text{in } \Omega (t),\\
\dt F + u \cdot \nabla F= \nabla u F&\text{in } \Omega (t),\\
\diverge u =0,\,\,\diverge (F^\text{T}) =0 &\text{in } \Omega (t).
\end{cases}
\end{equation}
In the whole space, the global well-posedness of the incompressible viscoelastic fluids near the equilibrium state was first obtained in \cite{LLZ_05} for the two-dimensional case (see also Lei and Zhou \cite{LZ_05}). For many related discussions, such as the three-dimensional case, and the initial-boundary problem, we refer the readers to \cite{Lei_08, Chen_06, LZ_08, L_12} and the references therein. 

For the free boundary problem of viscoelasticity, under the natural force balance law with the surface tension
\begin{equation} \label{visbdc}
 \left(-pI+(FF^{\text{T}}-I)\right)n+2\epsilon S(u)n=\sigma Hn\quad\text{on }\Gamma(t),
\end{equation}
well-posedness theories were established in \cite{Le,XZZ_13}. Here, $S(u)=\hal(\nabla u+\nabla u^{\text{T}})$ is the symmetric part of the gradient of $u$.  We emphasize that the method in \cite{Le,XZZ_13} depends greatly on the positivity viscosity. In the following discussion, we derive a new a priori estimate for the incompressible viscoelasticity, which is uniform with respect to the viscosity parameter $\epsilon$, and then a vanishing viscosity limit can be applied to establish the local well-posedness of \eqref{elastic}--\eqref{initial}.




%

\subsection{Reformulation in Lagrangian coordinates}To analyze \eqref{elastic}--\eqref{initial}, we transform the equation into Lagrangian coordinates, and then the problem is formulated on a fixed domain. The Lagrangian domain is given by
\begin{equation*}
\label{domain}
\Omega=\mathbb{T}\times(0,1),
\end{equation*}
where $\mathbb{T}$ denotes the 1-torus. The boundary of $\Omega$ is then given by the horizontally flat bottom and top:
\begin{equation*}
\Gamma =\mathbb{T} \times \left(\{0\}\cup \{1\}\right).
\end{equation*}
$N$ is the outward unit normal of $\Gamma$:
\begin{equation*}
N=e_2 \text{ when }x_2=1, \text{ and }N=-e_2\text{ when }x_2=0.
\end{equation*}

Now let $\eta(x,t)\in\Omega(t)$ denote the ``position'' of the fluid particle $x$ at time $t$; i.e.,
\begin{equation*}
\begin{cases}
\partial_t\eta(x,t)=u(\eta(x,t),t),\quad t>0,
\\ \eta(x,0)=\eta_0(x)
\end{cases}
\end{equation*}
where $\eta_0(x)$ is a volume-preserving diffeomorphism: 
\begin{equation*}
	\begin{split}
		\eta_0: \Omega&\rightarrow\Omega(0)\\
		x&\mapsto y
\end{split}
\end{equation*}
that satisfies $\text{det}\nabla\eta_0=1$ and $F_0(y)=\nabla\eta_0(\eta_0^{-1}(y))$. 

With the Lagrangian unknowns $v(x,t)=u(\eta(x,t),t)$, $F(x,t)=\nabla_x\eta(x,t)$ and $q(x,t)=p(\eta(x,t),t)$, the elastodynamic system \eqref{elastic}--\eqref{bdc} becomes the following:
\begin{equation}
\label{eq:elastic}
\begin{cases}
\partial_t\eta =v &\text{in } \Omega,\\
\partial_tv +\nabla_\eta q =\Delta\eta &\text{in } \Omega,\\
 \Div_\eta v = 0 &\text{in  }\Omega,\\
 -(q+1)\a N+FN=\partial_1\left(\dfrac{\pa_1\eta}{\abs{\pa_1\eta}}\right) & \text{on  }\Gamma,\\
 (\eta,v)\mid_{t=0} =(\eta_0, v_0), \,\,\text{det}\nabla\eta_0=1, \Div_{\eta_0}v_0=0.
 \end{cases}
\end{equation}
Here, $F_{ij}=(\nabla\eta)_{ij}$, $F^{-1}$ is the inverse of $F$, $J$ is the determinant of $F$, and $\a=JF^{-\text{T}}$. From the incompressibility condition, $J=1$ holds for all times. The differential operators are $\nabla_\eta:=(\partial_{\eta_1}, \partial_{\eta_2})$ with $\partial_{\eta_i}=F^{-1}_{ji}\partial_j$ and $\Div_\eta g=F^{-1}_{ji}\partial_jg_i$.

We now give the derivation of the boundary condition \eqref{eq:elastic}$_4$ from \eqref{bdc}. Without loss of generality, we set the surface tension coefficient $\sigma=1$. According to the kinematic boundary condition \eqref{bdc}$_1$, we have that 
\begin{equation*}
\Gamma(t)=\eta(t)(\Gamma).
\end{equation*}
Then, the unit tangent of $\Gamma(t)$ is $\tau=\frac{\pa_1\eta}{\abs{\pa_1\eta}}$ and the unit normal is $\mathfrak{n}=\frac{\pa_1\eta^\perp}{\abs{\pa_1\eta}}=\frac{(-\pa_1\eta_1,\pa_1\eta_2)^{\text{T}}}{\abs{\pa_1\eta}}$. By denoting $\pa_{\mathfrak n}:=\mathfrak n\cdot \nabla_\eta$, $\pa_\tau:=\tau\cdot \nabla_\eta$, we obtain
\begin{equation*}
\begin{split}
Hn=-\left(\nabla_\eta\cdot n\right)n=&-\left(\nabla_\eta\cdot \mathfrak{n}\right)\mathfrak{n}\\=&-\left(\mathfrak n\mathfrak n\cdot \nabla_\eta+\tau\tau\cdot\nabla_\eta \right)\cdot \mathfrak n\mathfrak n
\\=&\left(-\mathfrak n\cdot \pa_{\mathfrak n} \mathfrak n-\tau\cdot\pa_\tau \mathfrak n\right)\mathfrak n\\=&\left(\mathfrak n\cdot \pa_\tau \tau\right) \mathfrak n\\=&\, \pa_\tau\tau
\end{split}
\end{equation*}
by using the fact that $\pa_{\mathfrak n} \mathfrak n=0$, $\tau\cdot\pa_\tau\mathfrak n+\mathfrak n\cdot \pa_\tau\tau=0$ and $\pa_\tau \tau$ parellel to $\mathfrak n$.
On the other hand, we have $\pa_\tau=\tau\cdot\nabla_\eta=\dfrac{F_{i1}}{\abs{\pa_1\eta}}F^{-1}_{ji}\pa_j=\dfrac{1}{\abs{\pa_1\eta}}\pa_1$, the outward unit normal $n=\frac{\a N}{\abs{\a N}}$ and $\abs{\a N}=\abs{\pa_1 \eta}$. Substituting the above calculations into \eqref{bdc}$_2$, we arrive at \eqref{eq:elastic}$_4$.

\subsection{Main results}Before we state our main results, let us first discuss the issue of compatibility conditions for the initial data $(\eta_0, v_0)$. We work in a high-regularity context, essentially with regularity up to three temporal derivatives. This requires us to use $(\eta_0, v_0)$ to construct the initial data $\pa_t^jv(0)$ for $j = 1, 2, 3$ and $\partial_t^j q(0)$ for $j = 0, 1, 2$.
These other data must satisfy the conditions essentially obtained by applying $\pa_t^j$ to the equation \eqref{eq:elastic}$_2$ and then setting $t = 0$, which in turn requires $(\eta_0, v_0)$ to satisfy
three nonlinear compatibility conditions. We describe these conditions in detail in Section \ref{comptatible} and state them explicitly in \eqref{zcomp}, \eqref{1comp} and \eqref{2comp}. It is worth noting that these compatibility conditions will cause non-trivial difficulties in the construction of approximate solutions.

To state our results, 
we define the energy as
\begin{equation*}
\mathfrak{E}(t):=\sum_{j=0}^3\left(\norm{\pa_t^j\nabla\eta}_{H^{3-j}(\Omega)}^2+\norm{\pa_t^{j}v}_{H^{3-j}(\Omega)}^2+\abs{\pa_t^j\pa_1^{4-j}\eta\cdot n}_{L^2(\Gamma)}^2\right),
\end{equation*}
where the usual $L^2$ and Sobolev spaces $H^m$ are used. 
Now, the local well-posedness theory for the elastic system \eqref{eq:elastic} is stated as follows:
\begin{theorem} \label{thm}
	Suppose that the initial data $(\eta_0, v_0)$ satisfy the compatibility conditions \eqref{zcomp}, \eqref{1comp} and \eqref{2comp} and that $\{\eta_0\in H^4(\Omega), \pa_t^j v (0)\in H^{3-j}(\Omega), \pa_t^j\pa_1^{4-j}\eta\cdot n(0) \in L^2(\Gamma), 0\leq j\leq 3\}$. Then, there exists $T_0>0$ and a unique solution $(\eta,v, q)$ to the free boundary elastodynamics \eqref{eq:elastic} on the time interval $[0, T_0]$ that satisfies
\begin{equation}\label{enesti}
	\sup_{t \in [0,T_0]} \mathfrak E(t) \leq P\left(\mathfrak E(0)\right),
\end{equation}
where $P$ is a generic polynomial.
\end{theorem}
The proof of Theorem \ref{thm} is based on the vanishing viscosity method by establishing the local well-posedness of the associated viscoelasticity, which is uniform with respect to the viscosity. Note that the initial data $(\eta_0, v_0)$ for the elastodynamics should satisfy the compatibility conditions in Theorem \ref{thm}, which are not qualified for the viscoelasticity. Thus, the viscoelasticity is modified by adding an external smoothing force $f^\epsilon$ to serve as the approximate system. In this way, the compatibility conditions of the modified viscoelastic system will be satisfied by the regularized initial data $(\eta_0^\epsilon, v_0^\epsilon)$. By taking the limit $\epsilon\rightarrow 0$, it is then proven that $f^\epsilon \rightarrow 0$, and the modified viscoelastic system converges to the elastodynamics (see Theorem \ref{th43}). 

Now we state a byproduct, which is concise but very interesting in itself: when the initial data $(\eta_0, v_0)$ satisfy the compatibility conditions for the viscoelasticity (with $f^\epsilon=0$), the inviscid limit is justified by the following theorem:
\begin{theorem} \label{vanish}
	For sufficiently small $\epsilon$, $(\eta^\epsilon, v^\epsilon,q^\epsilon)$ is a solution to the viscoelasticity system \eqref{eq:viscoelastic} with initial data $(\eta_0, v_0)$ satisfying the compatibility condition \eqref{compforvisco}. Then, as $\epsilon \rightarrow 0$, $(\eta^\epsilon, v^\epsilon, q^\epsilon)$ converges to the limit $(\eta, v, q)$, which is the unique solution to \eqref{eq:elastic} on a time interval $[0, T_1]$ independent of $\epsilon$.
\end{theorem}

It is worth noting that the justification of the inviscid limit is quite complicated when there is a boundary. The gist of our analysis is to take advantage of the inherent structure of the elastic term on the free boundary, which also enables us to perform viscosity-independent a priori estimates in standard Sobolev spaces. This is significantly different from the fluid counterpart (such as \cite{MasRou,Wang_15}), in which the vanishing viscosity limit is normally derived in the conormal Sobolev spaces due to the existence of the boundary layer. The cause of such difference is: The elastic term $FF^T$ is essential in controlling the normal derivatives regardless of the viscosity parameter $\epsilon$; however, such a term does not exist in the Euler/Navier-Stokes equations. It may be argued that if $F=I$ at the initial time, the elastic term seems to disappear; hence, the free surface viscoelastic system reduces to a standard free surface Navier-Stokes system. Nevertheless, normally $F=I$ cannot be maintained due to the evolution equation of the deformation tensor, unless $u\equiv0, F\equiv I$, which is a trivial solution to the system. In this sense, the effect of the elastic term will still emerge in the system when $t>0$ and the elastodynamics/viscoelastodynamics with corresponding dynamic boundary conditions cannot be reduced to the free surface Euler/Navier-Stokes systems by taking $F$ as the identity initially.



\subsection{Strategy of the proof}\label{strategy}
To establish the local well-posedness, one may naturally hope to derive certain a priori estimates and then construct the approximate system which is asymptotically consistent with the a priori estimates. However, such a procedure can be highly nontrivial for a free-boundary problem of inviscid systems. Without the viscosity,  successful a priori estimates normally rely on the geometric transport-type structures of the nonlinear system, which are not easy to find. Furthermore, such structures are generally not valid in the linearized approximation; hence, the construction of approximate solutions becomes a difficult challenge. For instance, see \cite{CL_00, Lindblad05} for the case of the Euler equations. 

Regarding elastodynamics, we encounter great difficulties due to the the presence of elastic stress. More specifically, the force balance law shows a complicated interaction between the pressure and elasticity on the boundary. Recall the balance law in Lagrangian coordinates: 
\begin{equation}
\label{p1b}
 -(q+1)\a N+F N=\partial_1\left(\dfrac{\pa_1\eta}{\abs{\pa_1\eta}}\right).
\end{equation}
It has a vector form now, not the scalar one often occurring in inviscid systems, e.g., the Euler equations. This vector form balance law brings several technical difficulties in the process of obtaining local well-posedness; however, such difficulties can be overcome by studying the particular inherent structure of the law, as explained below.

Similar to the usual treatment of such free-boundary problems, our a priori estimates have two parts: tangential derivatives and normal derivatives. 
For the estimates of the tangential derivatives $\bp:=\pa_1,\pa_t$, we adopt the usual $L^2$-type energy estimate, which gives
\begin{equation*}
	\begin{split}
	\hal\dfrac{d}{dt}\int_{\Omega}\abs{\bp^3 v}^2+\abs{\nabla\bp^3\eta}^2+\int_{\Gamma}\left(\bp^3q\a N-\bp^3F N\right)\cdot\bp^3 v\,d\sigma=\sum\nolimits_q.
\end{split}
\end{equation*}
Here, $\sum_q$ stands for all the terms related to $q$.
The most difficult term is the boundary integral coming from integration by parts. To handle this term, the force balance law is introduced, yielding 
\begin{equation}\label{fea}
	\begin{split}\int_{\Gamma}\left(\bp^3q\a N-\bp^3F N\right)\cdot\bp^3 v\,d\sigma=&-\underbrace{\int_{\Gamma}\mathfrak B\bp^3(\a N)\cdot\bp^3v\,d\sigma}_{\text{ET}}\underbrace{-\int_{\Gamma}\bp^3\left(\dfrac{\partial_1^2\eta_k\a_{k2}}{\abs{\pa_1\eta}^3}\right)\a_{i2}\bp^3v_i\, d\sigma}_{\text{ST}}\\&+\intb \left[\bp^3, \mathfrak B, \a_{i\ell}N_\ell\right]\bp^3v_i\,d\sigma,
\end{split}
\end{equation}
Here, $\mathfrak B:=q+1+\dfrac{\pa_1^2\eta_k\a_{k\ell}N_\ell}{\abs{\pa_1\eta}^3}$, and $[\cdot,\cdot,\cdot]$ is a commutator term. The term ST is related to the surface tension and provides improved regularity for the boundary:
\begin{equation*}
	\begin{split}
		\text{ST}=\hal \dfrac{d}{dt}\intb \dfrac{1}{\abs{\pa_1\eta}}\abs{\bp^3\pa_1\eta\cdot n}^2\,d\sigma+\cdots\end{split}
\end{equation*}
The first main difficulty in our analysis is to deal with the term ET, which is caused by the vector form of the force balance law.
At first glance, ET seems similar to $\intb \bp^3\pa_1\eta\bp^3v$, and hence there is a loss of one derivative in using dual estimate on the boundary. To bypass this difficulty, with $\pa_t\eta=v$ and $\a_{\cdot 2}=\pa_1\eta^\perp=(-\pa_1\eta_2,\pa_1\eta_1)^{\text{T}}$, we first adopt Wu's idea in \cite{Wu1} to use time derivative of this integral:
\begin{equation*}
	\begin{split}
		\text{ET}=&\intb\mathfrak B\bp^3\pa_1\eta^\perp\cdot\bp^3v\\=&\dfrac{d}{dt}\intb \mathfrak B\bp^3\pa_1\eta^\perp\cdot\bp^3\eta-\intb \pa_t\mathfrak B\bp^3\pa_1\eta^\perp\cdot\bp^3\eta-\intb \mathfrak B\pa_t\bp^3\pa_1 \eta^\perp\cdot\bp^3\eta\\=&\dfrac{d}{dt}\intb \mathfrak B\bp^3\pa_1\eta^\perp\cdot\bp^3\eta-\intb \pa_t\mathfrak B\bp^3\pa_1\eta^\perp\cdot\bp^3\eta+\underbrace{\intb \mathfrak B\bp^3v^\perp\cdot\bp^3\pa_1\eta}_{\mathfrak A}-\intb \pa_1\mathfrak B\bp^3 v^\perp\cdot\bp^3\eta
	\end{split}
\end{equation*}
Apparently the term $\mathfrak A$ still cannot be controlled. However, by the antisymmetric property for $\eta^\perp$ and $\eta$:
$$\pa^a\eta^\perp\cdot\pa^b\eta=-\pa^b\eta^\perp\cdot\pa^a\eta,$$
we see that $\mathfrak A=-\text{ET}$ and find the symmetric structure for ET:
\begin{equation*}
	\text{ET}=-\dfrac{1}{2}\dfrac{d}{dt}\int_{\Gamma}\mathfrak B\bp^3(\a_{i\ell}N_\ell)\bp^3\eta_i\,d\sigma+\cdots
\end{equation*}
Then by using the trace estimate $\abs{f}_{L^2(\Gamma)}^2\ls \norm{f}_{L^2(\Omega)}^2+\norm{f}_{L^2(\Omega)}\norm{\nabla f}_{L^2(\Omega)}$, ET term can be controlled by $\abs{\bp^3\pa_1\eta\cdot n}_{L^2(\Gamma)}^2$ and $\norm{\bp^3\nabla\eta}_{L^2(\Omega)}^2$.

In estimating of normal derivatives, one way to control the normal derivatives for the inviscid system is to combine the estimates for divergence, vorticity and Hodge's type estimates; e.g., see \cite{CS07, GuW_2016}. However, due to the force balance law, the boundary condition is unclear when one tries to obtain energy estimates for the vorticity; thus, this approach is not currently applicable. Therefore, the second main difficulty in our analysis is to derive the normal derivative estimates. Our idea is to make use of the elastic term $\Delta\eta$ to do the elliptic estimates and control the normal derivatives inductively. 
We write the equation as
\begin{equation*}
-\pa_2^2\eta=\pa_1^2\eta-\pa_t v+\nabla_\eta q
\end{equation*}
and find that the first two terms are tangential derivatives; thus the key to applying the induction method is to derive a delicate estimate of the pressure $q$, such as
$$\norm{\bp^2\nabla q}_{L^2(\Omega)}^2(t) \ls M_0+TP\left(\sup_t\mathfrak E(t)\right).$$ 
Note that such a pressure estimate is usually obtained after all derivative estimates of $(\eta, v)$ are available. Hence, with only tangential energy estimates, it is nontrivial and cannot be obtained by the usual approach (see \cite[Lemma 12.1]{CS07}). 

The reason we can successfully establish such pressure estimates is that we find  
the following important inherent structure on the boundary: 
\begin{equation}
\label{p2b}
	 F_{i2}=\dfrac{\a_{i2}}{\abs{\a_{\cdot 2}}^2} \rightarrow \pa_2\eta=\dfrac{\pa_1\eta^\perp}{\abs{\pa_1\eta}^2},
\end{equation}
which is obtained by taking an orthogonal projection of the force balance law \eqref{p1b}.
The above structure indicates that the normal derivatives of $\eta$ can be controlled by its tangential derivatives, and this observation can actually improve the elliptic estimates for $q$. In short, when the pressure is studied through a divergence-form elliptic equation with a Neumann boundary condition:
\begin{equation*}
N_\ell\a_{i\ell}\a_{ij}\pa_j q=-\pa_tv_i\a_{i\ell}N_\ell+\Delta\eta_i\a_{i\ell}N_\ell,
\end{equation*}
$\Delta\eta_i\a_{i\ell}N_\ell$ is the most difficult to control. The normal trace estimates can only show that 
$$\abs{\bp^2\Delta\eta_i\a_{i\ell}N_\ell}_{H^{-\hal}(\Gamma)}\ls \norm{\bp^2\Delta\eta}_{L^2(\Omega)}^2+\norm{\Div_\eta\bp^2\Delta\eta}_{L^2(\Omega)}^2,$$
and both two terms contain two normal derivatives; hence, they are out of control at the moment. However, noticing that $\a_{i\ell}N_\ell=\pm\a_{i2}$ on the top or bottom boundary, with the help of \eqref{p2b} and the equality $\Delta\eta_i\a_{i2}=-\pa_1\eta_i\pa_1\a_{i2}-\pa_2\eta_i\pa_2\a_{i2}=-\pa_1\eta_i\pa_1\a_{i2}-\a_{i1}\pa_1F_{i2}$, we see that the highest-order term contains only $\pa_1$ derivatives. Then, by combining this observation with the trace estimates
$$\abs{\pa_1f}_{H^{-\hal}(\Gamma)}\ls \abs{f}_{H^{\hal}(\Gamma)}\ls \norm{f}_{H^1(\Omega)},$$
we find that $\abs{\bp^2\Delta\eta_i\a_{i\ell}N_\ell}_{H^{-\hal}(\Gamma)}$ can be controlled by the tangential energy. Thus, the required estimate for $q$ can be obtained, and the induction method is applied to control the normal derivatives. 


After obtaining the a priori estimates of the elastodynamics, the next step is to construct the approximate solutions. At this point, an important observation is that the methodology of obtaining the estimates for normal derivatives can also be used for the viscoelasticity; then, it is possible to construct approximate viscosity solutions.  

For the viscoelasticity, the tangential energy estimates are obtained almost parellel to the elastodyanmics, with an extra dissipation term. This estimate is also independent of the viscosity parameter $\epsilon$. On the other hand, the normal derivatives are studied by using the rewritten equation:
\begin{equation*}
-\pa_2^2\eta-\epsilon\abs{\a_{\cdot2}}^2\pa_2^2v=\pa_1^2\eta-\pa_t v+\nabla_\eta q+\epsilon\left(\sum_{\alpha+\beta \neq 4}\pa_\alpha(\a_{k\alpha}\a_{k\beta}\pa_\beta v)+\pa_2\abs{\a_{\cdot2}}^2\pa_2v\right).
\end{equation*}
A similar structure as \eqref{p2b} can be found on the boundary by applying the projection technique:
$$F_{i2}+\epsilon\pa_{\eta_j}v_i\a_{j2}=\dfrac{1}{\abs{\a_{\cdot 2}}^2}\left(\a_{i2}-2\epsilon\pa_t\a_{j2}\a_{j2}\a_{i2}\right)+\epsilon\pa_t\a_{i2}.$$ Of course, with the viscosity, the formula becomes more complicated. More importantly, the viscosity tensor should be combined with the elastic stress instead of being dealt with separately. With this inherent structure, the required estimates for the pressure are also obtained and the normal derivatives are controlled by the induction method, which is independent of $\epsilon$. 

When the viscosity-independent a priori esimates for the viscoelasticity are available, another technical difficulty, the issue of compatibility, arises. Due to the presence of viscosity, especially on the free boundary, the compatibility conditions for the viscoelasticity are different from those for the elastodynamics. Hence, the initial data of the elastodynamics may not be valid for the viscoelasticity, and as a result, we cannot use the viscoelasticity as an approximate system directly. Our strategy is to adopt the idea in \cite{Cheng_Shkoller_10} to smooth the initial data of the elastodynamics in a suitable way and  modify the viscoelastic system by adding an external smoothing force $f^\epsilon$. More precisely, given that the initial data $(\eta_0, v_0)$ satisfy the compatibility condition for the elasticity, the fourth-order elliptic system and the Stokes system are used to produce the regularized $(\eta_0^\epsilon, v_0^\epsilon)$, which still satisfies the compatibility condition for the elasticity. Then,` we can add a smoothing modification $\epsilon\pa_j\Psi_{ij}$
to the viscoelasticity, which is designed to cancel $\epsilon J\Delta_\eta v_i$ at $t=0$. Thus, $(\eta_0^\epsilon, v_0^\epsilon)$ and the corresponding $\pa_tv^\epsilon(0)$ satisfy the compatibility condition for the modified viscoelastic system.  In this way, the problem of noncompatibility is solved. Note that the modification term is smooth and depends only on the initial data; then, our framework to derive the viscosity-independent a priori estimates works for both the original viscoelastic system and the modified one.  Hence, the vanishing viscosity limit of the modified viscoelastic system is applied to establish the local well-posedness of the elastodynamics. Moreover, when the initial data of the elastodynamics satisfy the compatibility conditions for the viscoelasticity, we can obtain a byproduct that justifies the inviscid limit of the incompressible viscoelasticity.

At the end of this subsection, it is worth noting that the presence of elastic stress plays an important role in our analysis, especially for the estimates of the normal derivatives. First, the term $\Delta \eta$ in the momentum equation allows us to estimate the normal derivatives in a consistent way for both viscid and inviscid systems. 
Second, the term $F_{i2}$ in the dynamical boundary condition \eqref{bdc} allows us to estimate the normal derivatives from the tangential derivatives on the boundary in some sense. This fact is crucial in justifying the inviscid limit in standard Sobolev spaces.

\subsection{A Review of related results}
Now, we briefly review some related results in the literature. When the elastic effect is neglected, the free boundary problem of the Euler equations has been studied intensively by many mathematicians in recent decades. The early works focused on the irrotational fluids, which began with the work of Nalimov \cite{N} on local well-posedness for small initial data. Then Wu \cite{Wu1,Wu2} (see also Lannes \cite{Lannes}) generalized the results to general initial data. For irrotational inviscid fluids, certain dispersive effects can be used to establish the global well-posedness for small initial data; we refer to  Wu \cite{Wu3,Wu4}, Germain, Masmoudi and Shatah \cite{GMS1, GMS2}, Ionescu and Pusateri \cite{IP,IP2} and Alazard and Delort \cite{AD}. For the general incompressible Euler equations, the first local well-posedness in 3D was established by Lindblad \cite{Lindblad05} for the case without surface tension (see Christodoulou and Lindblad \cite{CL_00} for the a priori estimates) and
by Coutand and Shkoller \cite{CS07} for the case with (and without) surface tension. We also refer to the results of Shatah and Zeng \cite{SZ} and Zhang and Zhang
\cite{ZZ}. 

For the study of vanishing viscosity limits, one expects that as the fluid viscosity $\epsilon$ goes to $0$, the limit of solutions to the viscoelastic system converges to a solution to the elastodynamic system. 
For cases in which there is no boundary, this problem has been well studied in \cite{Cai}. In that work, Cai, Lei, Lin and Masmoudi justified the vanishing viscosity limit globally in time. 
The reader is also referred to Kato \cite{kato72}, Swann \cite{swann} and Masmoudi \cite{mas_07} for the case of the Navier-Stokes equations. 
However, in the presence of boundaries, the problem becomes challenging due to the possible formation of boundary layers, we refer to the results of a series works by Lopes Filho, Mazzucato, Nussenzveig Lopes and Taylor \cite{Fil2008,Fil2008ii,TaylorI}. For the free-surface problem, only recently, the inviscid limit of the incompressible Navier-Stokes equations without surface tension was justified by Masmoudi and Rousset \cite{MasRou} in conormal Sobolev spaces.
Later, Wang and Xin \cite{Wang_15} considered the surface tension case, and Mei, Wang and Xin \cite{MWX} considered the compressible case.
In contrast to these results, our inviscid limit of the viscoelasticity is justified in standard Sobolev spaces. The cause of such difference is explained briefly after Theorem \ref{vanish}. From this perspective, our analysis suggests that the existence of the elastic stress plays a crucial role in preventing the formation of the boundary layer. 
\subsection{Outline of the paper}
The remainder of this paper is organized as follows. In Section 2, we introduce the notation used throughout the paper. In Section 3, the compatibility conditions for the initial data are explained in detail. In Section 4, we recall some preliminary analysis lemmas that will be used frequently. In Section 5, a modified viscoelastic system is introduced, and viscosity-independent a priori estimates are derived in Section 6. In Section 7, local well-posedness is proved, and the inviscid limit is justified. Lastly, in the Appendix, we illustrate the construction of the smoothed initial data, which are used to address the compatibility issue between viscid and inviscid systems.

%

\section{Notation}
Einstein's summation convention is used throughout the paper. We use $\bp:=\pa_1,\pa_t$ to denote the tangential derivatives.

We work with the usual $L^p$ and Sobolev spaces $W^{m,p}$ and $H^m=W^{m,2}$ on both the domain $\Omega$ and the boundary $\Gamma$. For notational simplicity, we denote the norms of these spaces defined on $\Omega$ by $\norm{\cdot}_{L^p}, \norm{\cdot}_{W^{m,p}}$ and $\norm{\cdot}_{m}$ and the norms of these spaces defined on $\Gamma$ by $\abs{\cdot}_{L^p}, \abs{\cdot}_{W^{m,p}}$ and $\abs{\cdot}_{m}$.  For real $s\geq 0$, the Hilbert space $H^s(\Gamma)$ and the boundary norm $\abs{\cdot}_s$ (or $\abs{\cdot}_{H^s(\Gamma)}$) are defined by interpolation. The negative-order Sobolev space $H^{-s}(\Gamma)$ is defined via duality: for real $s\geq 0, H^{-s}(\Gamma):=[H^{s}(\Gamma)]^\prime.$

We also introduce the spatial-time Sobolev norms on $\Omega$ as:
$$\norm{u}_{\X^m}^2:=\sum_{\ell\leq m}\norm{\partial_t^{\ell}u}_{m-\ell}^2$$
and the spatial-time Sobolev norms on $\Gamma$ as:
$$\abs{u}_{\X^m}^2:=\sum_{\ell\leq m}\abs{\partial_t^{\ell}u}_{m-\ell}^2.$$
We use $\norm{\cdot}_{L^p_t(\textbf{X})}$ as the norm of the space $L^p([0, t];\textbf{X})$.

We use $C$ to denote generic constants, which depend only on the domain $\Omega$ and the boundary $\Gamma$, and  use $f\ls g$ to denote $f\leq Cg$.  We use $P$ to denote a generic polynomial function of its arguments, and the polynomial coefficients are generic constants $C$.

\section{The compatibility condition for the initial data}\label{comptatible}
In this section, some compatibility conditions for the initial data $(\eta_0, v_0)$ are derived since we are looking for solutions that are continuously differentiable in time. First, $\pa_tv(0)$ can be computed by using \eqref{eq:elastic}$_2$ at $t=0$:
\begin{equation}
	\label{vt0}
	\pa_tv(0)=-\nabla_{\eta_0} q_0+\Delta\eta_0.
\end{equation}
The initial pressure function $q_0$ satisfies the following equation:
\begin{equation*}
\begin{cases}
	-\Delta_{\eta_0} q_0=\partial_{\eta_{0i}} v_{0j}\partial_{\eta_{0j}}{v_0}_i-\pa_{\eta_{0i}}(\Delta\eta_{0i})&\text{in }\Omega,\\
	-\left(q_0+1+\dfrac{\pa_1^2\eta_{0\alpha}\a_{\alpha \ell}N_\ell(0)}{\abs{\pa_1\eta_0}^3}\right)\a_{i2}(0)+\pa_2\eta_{0i}=0&\text{on }\Gamma.
\end{cases}
\end{equation*}
The above equation is not solvable in general unless $\eta_0$ satisfies the zeroth-order compatibility condition
\begin{equation*}
	\Pi_0 \pa_2\eta_{0i}=0\,\,\text{on }\Gamma,
\end{equation*}
where $\Pi_0f_i:=f_i-\dfrac{f_k\a_{k2}(0)\a_{i2}(0)}{\abs{\a_{\cdot 2}(0)}^2}$ is the orthogonal projection onto the tangent plane of $\eta_0(\Gamma)$ at $t=0$. With $\a_{\cdot 2}=(-\pa_1\eta_2, \pa_1\eta_1)$, this compatibility condition results in 
\begin{equation}\label{zcomp}
	(\pa_2\eta_{01},\pa_2\eta_{02})=\dfrac{(-\pa_1\eta_{02},\pa_1\eta_{01})}{\abs{\pa_1\eta_0}^2}\,\,\text{on }\Gamma,
\end{equation}
and $q_0$ is the solution of the following elliptic equation:
\begin{equation}
	\label{q0solv}
\begin{cases}
	-\Delta_{\eta_0} q_0=\partial_{\eta_{0i}} v_{0j}\partial_{\eta_{0j}}{v_0}_i-\pa_{\eta_{0i}}(\Delta\eta_{0i})&\text{in }\Omega,\\
	q_0=\dfrac{1}{\abs{\pa_1\eta_0}^2}-1-\dfrac{\pa_1^2\eta_{0\alpha}\a_{\alpha \ell} N_\ell(0)}{\abs{\pa_1\eta_0}^3}&\text{on }\Gamma.
\end{cases}
\end{equation}

Next, after time-differentiating \eqref{eq:elastic}$_2$ and letting $t=0$, we have:
\begin{equation}
	\label{vt20}
\partial_t^2 v(0)=\partial_t\bigg(-\nabla_{\eta} q+\Delta\eta\bigg)\bigg|_{t=0}=
-\nabla v_0\nabla q_0-\nabla_{\eta_0} q_1+\Delta v_0,
\end{equation}
where $q_1:=\pa_tq(0)$ is the solution to\begin{equation*}
	\begin{cases}
            -\Delta_{\eta_0} q_1=\left(\nabla^2\eta\nabla^2v+\nabla v\nabla \pa_tv+\left[\dt, \Delta_\eta\right]q\right)(0)&\text{in }\Omega,\\
		q_1=\dfrac{1}{\abs{\pa_1\eta_0}^5}P(\pa_1\eta_0,\pa_1 v_0, \pa_1^2v_0)&\text{on }\Gamma.
	\end{cases}
\end{equation*}
Similarly, the following 1st-order compatibility condition is needed:
\begin{equation*}
	\Pi_0 \left(-\left(q_0+1+\dfrac{\pa_1^2\eta_{0\alpha}\a_{\alpha\ell}N_\ell(0)}{\abs{\pa_1\eta_0}^3}\right)\pa_t\a_{i2}(0)+\pa_2v_{0i} \right)=0\quad\text{on }\Gamma,
\end{equation*}
Inserting \eqref{q0solv}$_2$ into the above equality, we arrive at
\begin{equation}
	\label{1comp}
	\Pi_0\left(\pa_2 v_{0i}-\dfrac{\pa_t\a_{i2}(0)}{\abs{\pa_1\eta_0}^2}\right)=0.
\end{equation}

For $\pa_t^3v(0)$, we have
\begin{equation*}
\partial_t^3 v(0)=\partial_t^2\bigg(-\nabla_{\eta} q+\Delta\eta\bigg)\bigg|_{t=0}=
-\nabla \pa_tv(0)\nabla q_0-\nabla\eta_0\nabla q_2-2\nabla v_0\nabla q_1+\Delta \pa_tv(0),
\end{equation*}
where $q_2=\pa_t^2q(0)$ satisfies 
\begin{equation*}
		\begin{cases}
                -\Delta_{\eta_0} q_2=\left(\nabla^2\eta\nabla\pa_tv+2\nabla^2 v\nabla^2v+\nabla v\nabla \pa_t^2v+2\nabla\pa_t v\nabla\pa_t v+\left[\dt^2,\Delta_\eta\right]q\right)(0)&\text{in }\Omega,\\
			q_2=\dfrac{1}{\abs{\pa_1\eta_0}^7}P(\pa_1\eta_0,\pa_1v_0, \pa_1^2v_0,\pa_1^2\pa_tv(0),\pa_1\pa_tv(0))&\text{on }\Gamma.
	\end{cases}
\end{equation*}
and the following 2nd-order compatibility condition is required:
\begin{equation}
	\label{2comp}
	\Pi_0 \left(-\dfrac{1}{\abs{\pa_1\eta_0}^2}\pa_t^2\a_{i2}(0)+2\dfrac{\pa_1\eta_0\cdot\pa_1v_0}{\abs{\pa_1\eta_0}^4}\pa_t\a_{i2}(0)+\pa_2\pa_tv(0)_i \right)=0\quad\text{on }\Gamma.
\end{equation}



\section{Preliminaries}

\subsection{Commutator estimates}

Some commutator estimates are recalled below.
\begin{lemma}
The following hold:

$(i)$ For $1\leq \abs{\alpha}\leq 3$, we define the commutator
\begin{equation*}
\left[\pa^{\alpha}, g\right]h = \pa^{\alpha}(gh)-g\pa^{\alpha} h.
\end{equation*}
Then we have
\begin{equation}
\label{co1}
\norm{\left[\pa^{\alpha}, g\right]h}_0\ls\left(\norm{\pa^3 g}_{0}+\norm{\nabla g}_{\X^2}^{\hal}+\norm{g}_{\X^2}^2+\norm{g}_{\X^2}\right)\norm{h}_{\X^2}.
\end{equation}
$(ii)$ For $2\leq~\abs{\alpha}\leq 3$, we define the symmetric commutator
\begin{equation}
	\label{sc}
\left[D^{\alpha}, g, h\right] = D^{\alpha}(gh)-D^{\alpha}g h-gD^{\alpha} h.
\end{equation}
Then we have
\begin{equation}
\label{co2}
\norm{\left[\pa^{\alpha}, g, h\right]}_0\ls\norm{\pa g}_{\X^2}\left(\norm{\pa^2 h}_{0}+\norm{\pa h}_{0}^\hal\norm{\pa \nabla h}_0^\hal\right) .
\end{equation}
\end{lemma}
\begin{proof}
	We prove \eqref{co1} and \eqref{co2} with $\abs{\alpha}=3$.

	First, we have that
	\begin{equation*}
		\norm{\left[\pa^3,g\right]h}_0\ls \norm{\pa^3g}_0\norm{h}_{L^\infty}+\norm{\pa^2 g}_{L^3}\norm{\pa h}_{L^6}+\norm{\pa g}_{L^\infty}\norm{\pa^2 h}_0.
	\end{equation*}
	By the Gagliardo--Nirenberg interpolation inequality 
	\begin{equation*}
		\norm{f}_{L^3}\ls \norm{f}_0^{\frac{2}{3}}\norm{\nabla f}_0^{\frac{1}{3}}
	\end{equation*}
	and Sobolev's embedding inequality in 2D
	$$\norm{f}_{L^\infty}\ls \norm{f}_2,\,\,\norm{f}_{L^\infty}\ls \norm{f}_{W^{1,3}},\,\, \norm{f}_{L^r}\ls\norm{f}_1,$$
	we arrive at
	\begin{equation*}
		\norm{\left[\pa^3,g\right]h}_0\ls \norm{\pa^3g}_0\norm{h}_{\X^2}+\norm{\pa^2\nabla g}_0^{\frac{1}{3}}\norm{\pa^2g}_0^{\frac{2}{3}}\norm{h}_{\X^2}+\left(\norm{\pa g}_1+\norm{\pa\nabla^2 g}_0^{\frac{1}{3}}\norm{\pa\nabla g}_0^{\frac{2}{3}}\right)\norm{h}_{\X^2}.
	\end{equation*}
	Then, by Cauchy's inequality, the estimate \eqref{co1} is proved.

	For the proof of \eqref{co2}, we have 
	\begin{equation*}
		\norm{\left[\pa^3, g, h \right]}_0\ls \norm{\pa g\pa^2 h}_0+\norm{\pa^2g \pa h}_0\ls \norm{\pa g}_{\X^2}\norm{\pa^2 h}_0+\norm{\pa^2 g}_{L^4}\norm{\pa h}_{L^4}.
		\end{equation*}
		Then, by Ladyzhenskaya's inequality, \eqref{co2} is proved.
\end{proof}
We will also use the lemma below.
\begin{lemma}
It holds that
\begin{equation}
\label{co123}
\abs{gh}_{\frac{1}{2}} \ls \abs{g}_{1}\abs{h}_{\hal}, \abs{gh}_{-\hal}\ls \abs{g}_{1}\abs{h}_{-\hal}.
\end{equation}
 \end{lemma}
\begin{proof}
	It is straightforward to check that $\abs{gh}_{s} \ls \abs{g}_{1}\abs{h}_{s}$ for $s=0,1$ with the help of the Sobolev embedding $\abs{f}_{L^\infty}\ls \abs{f}_1$. Then, the estimate \eqref{co123} follows by the interpolation. The second inequality follows by the dual estimate.
\end{proof}

%
%

\subsection{Trace estimates}
First, we have the trace estimates below.
\begin{lemma}It holds that
\begin{equation}\label{tre}
\abs{\omega}_0^2\ls \norm{\omega}_0^2+\norm{\omega}_0\norm{\nabla \omega}_0
\end{equation}
\end{lemma}
\begin{proof}
By the fundamental theorem of calculus, we have
	\begin{equation*}
		\int_{\mathbb T}\abs{\omega(x_1,0)}^2\,dx_1=\int_{\mathbb T}\abs{\omega(x_1,x_2)}^2\,dx_1-\int_0^{x_2}\pa_2\left(\int_{\mathbb T}\abs{\omega(x_1,y)}^2\,dx_1\right)\,dy
\end{equation*}
Then integrating from $0$ to $1$ with respect to $x_2$, we arrive at
\begin{equation*}
	\begin{split}
		\int_{\mathbb T}\abs{\omega(x_1,0)}^2\,dx_1\cdot 1\ls &\norm{\omega}_0^2+\int_0^1\abs{\int_0^{x_2}\int_{\mathbb T}\pa_2\omega\cdot\omega \,dx_1\,dy}\,dx_2\\\ls&\norm{\omega}_0^2+\int_0^1\int_{\mathbb T}\abs{\pa_2\omega\cdot\omega}\,dx_1\,dx_2\\\ls&\norm{\omega}_0^2+\norm{\omega}_0\norm{\nabla \omega}_0
\end{split}
\end{equation*}
by using Cauchy's inequality. Similarly, we have 
\begin{equation*}
	\int_{\mathbb T}\abs{\omega(x_1,1)}^2\,dx_1\ls \norm{\omega}_0^2+\norm{\omega}_0\norm{\nabla \omega}_0
\end{equation*}
and thus we prove the lemma.
\end{proof}
Next, we have the normal trace estimates below.
\begin{lemma}\label{normal trace}
It holds that
\begin{equation}\label{gga}
	\abs{\omega_i\a_{i\ell}N_\ell}_{-\hal}\ls \norm{\nabla\eta}_{L^\infty}\left(\norm{\omega}_0+\norm{J\Div_\eta\omega}_0\right).
\end{equation}

\end{lemma}
\begin{proof}
	We refer the reader to \cite[Appendix C]{Wang_15}.
\end{proof}

Our energy estimates require the use of the following:
\begin{lemma}\label{pereaf}
Let $H^{\frac{1}{2}}(\Gamma)^\prime$ denote the dual space of $H^{\frac{1}{2}}(\Gamma)$. There exists a positive constant $C$ such that
\begin{equation}
\label{peadv}
\abs{\bar\partial F}_{-\frac{1}{2}}:=\abs{\bar\partial F}_{H^{\frac{1}{2}}(\Gamma)^\prime}\leq C\abs{F}_{\frac{1}{2}}, \forall F\in H^{\frac{1}{2}}(\Gamma).
\end{equation}

\end{lemma}
\begin{proof}
	See, for instance, \cite[Lemma 8.5]{DS_10}.
\end{proof}
\subsection{Korn's inequality}
\begin{equation}\label{korn}
        \norm{\nabla f}_0^2\ls P\left(\norm{\nabla\eta}_{2}^2\right)\left(\norm{S_\eta(f)}_0^2+\norm{f}_0^2\right).
\end{equation}
\begin{proof}
        We follow the idea in \cite{MasRou} to prove the above Korn's inequality. We set $\tilde f= Ff$, $S(\tilde f)_{ij}=\hal(\pa_j\tilde f_i+\pa_i\tilde f_j)$, then we have
        \begin{equation}
                \label{l11}
                \begin{split}
                        \norm{\nabla f}_0^2&\ls \norm{\nabla\eta}_{L^\infty}^2\norm{\nabla \tilde f}_0^2+\norm{\nabla^2\eta}_{L^4}^2\norm{\tilde f}_{L^4}^2\\&\ls \norm{\eta}_2^2\left(\norm{\nabla\tilde f}_0^2+\norm{\tilde f}_0^2\right) \\&\ls P\left(\norm{\nabla\eta}_2^2\right)\left(\norm{S(\tilde f)}_0^2+\norm{f}_0^2\right)
                \end{split}
        \end{equation}
with the classical Korn's inequality in $\Omega$.

On the other hand, by straightforward caculation, we have
        \begin{equation}
                S(\tilde f)_{ij}=\pa_j(F_{ik}f_k)+\pa_i(F_{jm}f_m)=\nabla^2\eta f+ F_{ik}\pa_jf_k+F_{jm}\pa_if_m
        \end{equation}
Using $J=\pa_1\eta_1\pa_2\eta_2-\pa_1\eta_2\pa_2\eta_1=1$, we can verify
        \begin{equation}
                \begin{split}
                        S(\tilde f)_{12}=S(\tilde f)_{21}=&\hal \left(F_{1k}\pa_2f_k+F_{2m}\pa_1f_m\right)+\nabla^2\eta f\\=&\alpha (S_{\eta}f)_{11}+\beta (S_\eta f)_{22}+\gamma\left((S_\eta f)_{12}+(S_\eta f)_{21}\right)+\nabla^2\eta f
                \end{split}
\end{equation}
where 
\begin{equation*}
        \alpha=\pa_1\eta_1\pa_1\eta_2,\beta=\pa_2\eta_1\pa_2\eta_2, \gamma=\hal(\pa_1\eta_1\pa_2\eta_2+\pa_1\eta_2\pa_2\eta_1).
\end{equation*}
Similarly, we also have
\begin{equation}
        \begin{split}
                S(\tilde f)_{11}=&(\pa_1\eta_1)^2 (S_{\eta}f)_{11}+(\pa_2\eta_1)^2(S_{\eta}f)_{22}+\pa_1\eta_1\pa_2\eta_1( (S_{\eta}f)_{12}+ (S_{\eta}f)_{21})+\nabla^2\eta f\\S(\tilde f)_{22}=&(\pa_1\eta_2)^2 (S_{\eta}f)_{11}+(\pa_2\eta_2)^2(S_{\eta}f)_{22}+\pa_1\eta_2\pa_2\eta_2( (S_{\eta}f)_{12}+ (S_{\eta}f)_{21})+\nabla^2\eta f.
\end{split}
\end{equation}
Thus, we arrive at
\begin{equation}
        \label{l2}
        \norm{S(\tilde f)}_0^2\ls P\left(\norm{\nabla\eta}_{L^\infty}^2\right)\norm{S_\eta(f)}_0^2+\norm{\nabla^2\eta}_{L^4}^2\norm{f}_{L^4}^2
\end{equation}
Combining the estimates \eqref{l11} and \eqref{l2}, we have
\begin{equation}
        \begin{split}
                \norm{\nabla f}_0^2 &\leq P\left(\norm{\nabla\eta}_2^2\right)\left(\norm{S_\eta(f)}_0^2+\norm{f}_0\norm{\nabla f}_0+\norm{f}_0^2\right)\\&\leq P\left(\norm{\nabla\eta}_2^2\right)\left(\norm{S_\eta(f)}_0^2+\norm{f}_0^2\right)+\hal\norm{\nabla f}_0^2.
\end{split}
\end{equation}
Then we arrive at the conclusion. 
\end{proof}
\subsection{Geometric identities}
We recall some useful identities that can be checked directly.
For $F$, we have the following identities for differentiating its determinant $J$ and $\a$:
\begin{align}
\label{dJ}
&\partial J=\dfrac{\partial J}{\partial {F}_{ij}}\partial {F}_{ij} =  {\mathcal{A}}_{ij}\partial {F}_{ij},\\
&\partial {F}^{-1}_{ij}  = -F^{-1}_{i\ell}\partial {F}_{\ell m}{F}^{-1}_{mj},
\label{partialF}
\end{align}
where $\partial$ can be the $\pa_1,\pa_2$, or $\partial_t$ operator. Moreover, we have the Piola identity 
\begin{equation}\label{polia}
\partial_j\left({\mathcal{A}}_{ij}\right) =0.
\end{equation}

For the outward unit normal of $\Gamma(t)$, $n=\frac{\a N}{\abs{\a N}}$, and the unit tangential vector $\tau:=\frac{\pa_1\eta}{\abs{\pa_1\eta}}$, we have the identities
\begin{align}\label{decomp}
I&=n\otimes n + \tau\otimes \tau,\\
\left(\begin{array}{cc}
0 & -N_2\\
N_2 & 0
\end{array} 
\right )&=n\otimes\tau-\tau\otimes n,\label{decomp2}
\end{align}
where we use the fact that $\a_{\cdot 2}=\pa_1\eta^\perp=(-\pa_1\eta_2,\pa_1\eta_1)^{\text{T}}$.

We also have the following important antisymmetric structure for $\eta$ and $\eta^\perp$:
\begin{equation}
\label{anti}
\pa^a\eta\cdot\pa^b\eta^\perp=-\pa^b\eta\cdot\pa^a\eta^\perp.
\end{equation}
As a result, the following equalities hold:
\begin{align}\label{in2}
\partial^a\a_{k2}\partial^b\pa_1\eta_k&=-\partial^b\a_{k2}\partial^a\pa_1\eta_k,\\
\label{in3}\pa^a\a N\cdot\tau&=-\pa^a\pa_1\eta\cdot n,\\
\label{in4}\pa^a\a N \cdot n&=\pa^a\pa_1\eta\cdot \tau.
\end{align}
The above equalities are quite useful in the estimation of the boundary integral of the tangential energy estimates.
\section{Modified viscoelastic system}
In this section, we introduce a modified viscoelastic system with a smooth externel force as our approximate system:
\begin{equation}\label{approximate}
\begin{cases}
\partial_t\eta =v&\text{in } \Omega,\\
\partial_tv  +J\nabla_{\eta} q -\Delta\eta-\epsilon J\Delta_\eta v+f^\epsilon =0 &\text{in } \Omega,\\
\Div_{\eta} v = 0 &\text{in  }\Omega,\\
-(q+1)\a N+FN+2\epsilon S_\eta(v)\a N-\epsilon\Psi N= \partial_1\left(\dfrac{\partial_1\eta}{\abs{\pa_1\eta}}\right)& \text{on  }\Gamma,\\
 (\eta,v)\mid_{t=0} =(\eta_0^\epsilon, v_0^\epsilon).
 \end{cases}
\end{equation}
Here, $f^\epsilon:=\phi^\epsilon(x)+\epsilon\nabla\cdot \Psi(x,t)$; the definitions of $\phi^\epsilon$ and $\Psi$ are explained later. Here, $(\eta_0^\epsilon, v_0^\epsilon)$ is the set of smoothed initial data that converge to $(\eta_0, v_0)$ when $\epsilon\rightarrow 0$. It satisfies $\text{det}\nabla\eta_0^\epsilon=J_0^\epsilon\geq \hal, J_0^\epsilon\big|_\Gamma=1, \Div_{\eta_0}v^\epsilon_0=0$ and the compatibility condition \eqref{zcomp}, \eqref{1comp} and \eqref{2comp}. The construction of these smoothed initial data will be explained in the Appendix. Moreover, from the incompressibility equation \eqref{approximate}$_3$, we have 
\begin{equation}
	\label{J}
	\pa_t J=0, \,\,J=J_0^\epsilon(x).
\end{equation}

When $\phi^\epsilon=\Psi=0$ and $J=1$, the above approximate system becomes 
\begin{equation}\label{eq:viscoelastic}
\begin{cases}
\partial_t\eta =v&\text{in } \Omega,\\
\partial_tv  +\nabla_{\eta} q -\Delta\eta-\epsilon \Delta_\eta v =0 &\text{in } \Omega,\\
\Div_{\eta} v = 0 &\text{in  }\Omega,\\
-(q+1)\a N+FN+2\epsilon S_\eta(v)\a N= \partial_1\left(\dfrac{\partial_1\eta}{\abs{\pa_1\eta}}\right)& \text{on  }\Gamma,
 \end{cases}
\end{equation}
which is the Lagrangian formulation of the viscoelasticity \eqref{viscoelastic}. However, we cannot use this directly as an approximate system due to the issue of the compatibility of the initial data. That is, in general, the initial data $(\eta_0, v_0)$ used in \eqref{eq:elastic} do not meet the requirement for the initial data of the viscoelasticity system.

First, the viscoelasticity system requires higher-regularity initial data. For instance, in the viscoelasticity system, $\pa_tv(0)$ can be computed by using \eqref{eq:viscoelastic}$_2$:
\begin{equation*}
	\pa_tv(0)=\Delta \eta_0-\nabla_{\eta_0} q_0 +\epsilon \Delta_{\eta_0} v_0.
\end{equation*}
Comparing \eqref{vt0}, achieving the same regularity of $\pa_tv(0)$ will require one more spatial regularity for $v_0$. On the other hand, to ensure the solvability of the elliptic system to decide $\pa_t^kq(0)$, the initial data need to satisfy the following compatibility condition on $\Gamma$, $i=0,1,2$:
\begin{equation}
	\label{compforvisco}
	\Pi \left(-\sum_{\ell <i}C_\ell\left(\pa_t^\ell q\pa_t^{i-\ell}\a_{j2}+\pa_t^\ell\dfrac{\pa_1^2\eta_\alpha\a_{\alpha \ell}N_\ell}{\abs{\a_{\cdot 2}}^3}\pa_t^{i-\ell}\a_{j2}\right)+\pa_t^iF_{j2}+2\epsilon\pa_t^i\left(S_\eta(v)\a\right)_{k2} \right)\Bigg|_{t=0}=0,
\end{equation}
which may not hold for the initial data $(\eta_0, v_0)$ in \eqref{eq:elastic}.

To overcome these issues of compatibility, we introduce smoothed initial data $(\eta_0^\epsilon, v_0^\epsilon)$ and add modification terms $f^\epsilon:=\phi^\epsilon+\epsilon\Psi$ in the approximate system \eqref{approximate}. First, we construct a smoothed $(\eta_0^\epsilon, v_0^\epsilon)$ satisfying the zeroth and 1st compatibility conditions \eqref{zcomp} and \eqref{1comp}. Next, we introduce $\phi^\epsilon(x)$ to make $\pa_tv^\epsilon(0)$ satisfy the 2nd compatibility condition \eqref{2comp}. The details of the construction can be found in the Appendix. Lastly, we introduce the compensator matrix $\Psi$ so that $(\eta_0^\epsilon, v_0^\epsilon)$ and the corresponding $\pa_tv^\epsilon(0), \pa_t^2v^\epsilon(0)$, which are calculated by \eqref{vt0} and \eqref{vt20}, satisfy the compatibility condition for the approximate system:
\begin{equation*}
	\Pi \left(-\sum_{\ell <i}C_\ell\left(\pa_t^\ell q\pa_t^{i-\ell}\a_{j2}+\pa_t^\ell\dfrac{\pa_1^2\eta_\alpha\a_{\alpha \ell}N_\ell}{\abs{\a_{\cdot 2}}^3}\pa_t^{i-\ell}\a_{j2}\right)+\pa_t^iF_{j2}+2\epsilon\pa_t^i\left(S_\eta(v)\a\right)_{j2}-\epsilon\pa_t^i\Psi_{j2} \right)\Bigg|_{t=0}=0.
\end{equation*}
To achieve this, we construct $\Psi$ as follows:
\begin{equation}
\Psi (x,t) =2\left(S_\eta(v)\a(0)+\pa_t\left(S_\eta(v)\a\right)(0)t+\hal\pa_t^2\left(S_{\eta}(v)\a\right)(0)t^2\right).
\end{equation}
In this way, at $t=0$, we essentially add nothing to equations \eqref{eq:elastic}$_2$ and \eqref{eq:elastic}$_4$; then, $\pa_t^iv^\epsilon(0)$ are the same and the compatibility conditions for the approximate system are the same as \eqref{zcomp}, \eqref{1comp} and \eqref{2comp}.

Finally, by choosing a suitable mollifier parameter (see Appendix), we obtain the following estimates
\begin{equation}
	\label{estforpsi}
	\norm{\phi^\epsilon}_2^2+\norm{\sqrt\epsilon\nabla\Psi}_{\X^2}^2\leq M_0=P(\mathfrak E(0))
\end{equation}
for the modification terms.
\section{Viscosity-independent a priori estimates}
The purpose of this section is to derive the $\epsilon$-independent estimates of the solutions to \eqref{approximate}. We define the high-order energy functional for the approximate viscosity system \eqref{approximate}:
\begin{equation}
\begin{split}
	\mathfrak{E}^{\epsilon}(t) =&\int_0^t\left(\norm{\nabla\eta^\epsilon}_{\X^{3}}^2+\norm{\epsilon\nabla v^\epsilon}_{\X^3}^2+\norm{\sqrt\epsilon \bp^2\pa_1\nabla v^\epsilon}_0^2\right)\,d\mathsf{t}\\&+\norm{\bp^2\pa_1 \eta^\epsilon}_{\X^1}^2(t)+\abs{\pa_1^2\bp^2\eta^\epsilon\cdot n}_{0}^2(t)+\norm{\eta^\epsilon}_{\X^3}^2(t)+\norm{\sqrt\epsilon\nabla^2\eta^\epsilon}_{\X^2}^2(t)\\&+\int_0^t\norm{\pa_t^3v^\epsilon}_0^4+\norm{\pa_t^3\nabla\eta^\epsilon}_0^4+\abs{\pa_1\pa_t^3\eta^\epsilon\cdot n}_0^4+\epsilon^2\left(\int_0^\mathsf t\norm{\pa_t^3\nabla v^\epsilon}_0^2\right)^2\,d\mathsf t.
\end{split}
\end{equation}
and establish the following theorem:
\begin{theorem} \label{th43}
For sufficiently small $\epsilon$, $(v^\epsilon,\eta^\epsilon,q^\epsilon)$ is a solution to \eqref{approximate}, and there exists a time $T_1$ independent of $\epsilon$ such that
\begin{equation*}
\label{bound}
\sup_{t\in [0,T_1]}\mathfrak{E}^{\epsilon}(t)\leq 2M_0,
\end{equation*}
where $M_0=P\left(\mathfrak E(0)\right).$
\end{theorem}
Since $J_0\geq\hal$, there exists a $c_0$ such that $\abs{\pa_1\eta_0}\geq c_0>0$. We take the time $T_{\epsilon}>0$ to be sufficiently small so that for $t\in[0,T_{\epsilon}]$,
\begin{align}
	\label{inin3}\abs{\a_{ij}(t)-\a_{ij}(0)}\leq \dfrac{1}{8}c_0, \,\, \abs{\pa_1\eta}\geq \dfrac{1}{4}c_0.
\end{align}

We begin our $\epsilon$-independent a priori estimates with the pressure estimates. For notational simplification, we will not explicitly write the dependence of the solution on $\epsilon$ in the rest of this section.
\subsection{Pressure estimate}\label{pressure1}

In the estimates in the later sections, one needs to estimate the pressure $q$. For this, applying $J\Div_\eta$ to \eqref{approximate}$_2$, and using \eqref{approximate}$_3$ and the Piola indentity \eqref{polia}, one obtains
\begin{equation}\label{ell}
	-\Div(E\nabla q )=G_1+\epsilon \a_{ik}\pa_{kj}\Psi_{ij}+\a_{jk}\pa_k\phi^\epsilon_j,
\end{equation}
where the matrix $E=\a^T\a$ and the function $G_1$ are given by
\begin{align}
&G_1=\partial_t\a_{ij}\partial_j v_i+\pa_k\a_{ij}\pa_{kj}\eta_i-\Delta J\sim\pa_1v\nabla v+\nabla^2\eta\nabla^2\eta-\Delta J.
\end{align}
Here, we use $\a_{ij}\pa_{jk}\eta_i=\pa_j(\a_{ij}F_{ik})=\Delta J$.
Note that by \eqref{inin3}, the matrix $E$ is symmetric and positive.

For the boundary condition for $q$, we have the Dirichlet boundary condition 
\begin{equation}
\label{bdq}
q+1=G_2:=-\dfrac{\pa_1^2\eta_i\a_{i\ell}N_\ell}{\abs{\pa_1\eta}^3}+\dfrac{1+2\epsilon\pa_t \a_{j2}\a_{j2}+\epsilon \Psi_{i2}\a_{i2}}{\abs{\pa_1\eta}^2}
\end{equation}
by timing $\dfrac{\a N}{\abs{\a N}^2}$ on the boundary condition \eqref{approximate}$_4$ and using \eqref{J} and $\a_{ik}\pa_k v_j\a_{j2}=-J\pa_t\a_{i2}$.

We also have the Neumann boundary condition
\begin{equation}
	\a_{i2}\a_{ij}\pa_j q=G_3:=-\pa_tv_i \a_{i2}+\Delta\eta_i \a_{i2}+\epsilon J\Delta_\eta v_i\a_{i2}-\epsilon\pa_j\Psi_{ij}\a_{i2}-\phi_i^\epsilon \a_{i2}
\end{equation}
by timing $\a_{i2}$ on equation \eqref{approximate}$_2$. Both the Dirichlet condition and the Neumann condition are used.

We now prove the estimate for the pressure $q$.
\begin{proposition}\label{pressure}For sufficiently small $\epsilon$, the following estimate holds with $T\leq T_\epsilon$ and $0<\delta<1$:
\begin{equation}
\label{press2}
\int_0^t\norm{q}_{\X^2}^2+\norm{\nabla q}^2_{\X^2}\,d{\mathsf t}\leq T^{1/4}P\left(\sup_{t\in [0,T]}\mathfrak E^{\epsilon}(t)\right)+\delta\sup_{t\in [0,T]}\mathfrak E^{\epsilon}(t).
\end{equation}
\end{proposition}
\begin{proof}

First, we multiply the equation \eqref{ell} by $q$ and then integrate over $\Omega$; we obtain
\begin{equation}\label{tpe}
\begin{split}
	&\int_\Omega E\nabla q \cdot \nabla q \\&= \int_\Omega G_1 q+\intb E\nabla q \cdot N q+\intb \left(\epsilon\pa_j\Psi_{ij}+\phi^\epsilon_i\right)\a_{i\ell}N_\ell q-\intd \a_{ik}\left(\epsilon\pa_{j}\Psi_{ij}+\phi^\epsilon_i\right) \pa_kq
\\&\ls \norm{G_1}_{L^{\frac{4}{3}}}\norm{q}_1+\abs{G_3+\left(\epsilon\pa_j\Psi_{ij}+\phi^\epsilon_i\right)\a_{i2}}_{-\hal}\norm{q}_1 +\norm{\a_{ik}\left(\epsilon\pa_j\Psi_{ij}+\phi^\epsilon_i\right)}_0\norm{q}_1
\end{split}
\end{equation}
Here, we used Sobolev's embedding theorem $\norm{f}_{L^4}\ls \norm{f}_1$.
With the Poincare inequality 
\begin{equation*}
\norm{f}_0\ls \abs{f}_0+\norm{\pa_2f}_0
\end{equation*}
and the estimate \eqref{estforpsi}, we have
\begin{equation}
\label{pee}
\begin{split}
	\norm{q}_1^2&\ls \norm{G_1}_{L^{\frac{4}{3}}}^2+\abs{G_2-1}_0^2+\abs{G_3+\left(\epsilon\pa_j\Psi_{ij}+\phi^\epsilon_i\right)\a_{i2}}_{-\hal}^2+\norm{\epsilon\nabla\eta\nabla\Psi+\nabla\eta\phi^\epsilon}_0^2\\&\ls\left(\abs{\epsilon\Psi_{i2}}_0^2+1\right)P\left(\norm{\eta}_{\X^3}^2\right)+\norm{\nabla\eta}_{L^\infty}\norm{\epsilon\nabla\Psi+\phi^\epsilon}_0^2\ls P\left(\norm{\eta}_{\X^3}^2\right).
\end{split}
\end{equation}

Next, applying $\bar\partial^\ell $ and $\abs{\ell}\leq 2$ to  equation \eqref{ell}, we obtain
\begin{equation*}
	-\Div\left(\bar\partial^\ell\left(E\nabla q\right) \right)= \bp^\ell G_1+\bp^\ell\left(\epsilon\a_{ik}\pa_{kj}\Psi_{ij}+\a_{jk}\pa_k\phi^\epsilon_j\right),
\end{equation*}
where
\begin{equation*}
	\bp^\ell G_1=\underbrace{\bp^\ell(\pa_1v\nabla v)+[\bp^\ell, \nabla^2\eta,\nabla^2\eta]}_{g_1}+\underbrace{\bp^\ell\nabla^2\eta\nabla^2\eta}_{g_2}+\underbrace{\bp^\ell\Delta J}_{g_3},
\end{equation*} 

By a similar method as for \eqref{tpe}, we have
\begin{equation*}
	\begin{split}
	\intd \bp^\ell(E\nabla q)\cdot \nabla \bp^\ell q =&\intd \bp^\ell G_1\bp^\ell q+\intb \bp^\ell\left(G_3+\epsilon\pa_j\Psi_{ij}\a_{i2}+\phi^\epsilon_i\a_{i2}\right)\bp^\ell q\\&-\intd \bp^\ell\left(\a_{ik}(\epsilon\pa_j\Psi_{ij}+\phi_i^\epsilon)\right)\pa_k\bp^\ell q
	\\=&\intd g_1\bp^\ell q-\intd \bp^\ell\nabla\eta \nabla^3\eta\bp^\ell q-\intd \bp^\ell\nabla\eta\nabla^2\eta\nabla\bp^\ell q+\intb \bp^\ell\nabla\eta\nabla^2\eta\bp^\ell q\\&+\intb \bp^\ell\left(G_3+\epsilon\pa_j\Psi_{ij}\a_{i2}+\phi^\epsilon_i\a_{i2}\right)\bp^\ell q-\intd \bp^\ell\left(\a_{ik}(\epsilon\pa_j\Psi_{ij}+\phi_i^\epsilon)\right)\pa_k\bp^\ell q\\&-\intd \pa_1^{\ell-1}\Delta J\pa_1\bp^\ell q
	\end{split}
\end{equation*}
where integration by parts is used to deal with $\intd g_2\bp^\ell q$ and $\intd g_3\bp^\ell q$. Here, we also use the fact that $J=J_0^\epsilon$, which depends only on $x$.

Then, we obtain
\begin{equation}\label{dddddddd}
	\begin{split}
\norm{ \bar\partial^\ell q}_{1}^2
\ls &\norm{J}_{3}^2+\norm{g_1}_{L^{\frac{4}{3}}}^2+\norm{\bp^\ell\nabla\eta\nabla^3\eta}_{L^{\frac{4}{3}}}^2+\norm{\bp^\ell\nabla\eta\nabla^2\eta}_0^2+\abs{\bp^\ell\nabla\eta\nabla^2\eta}_{L^{\frac{4}{3}}}^2\\
&+\abs{\bp^\ell \left(G_3+\epsilon\pa_j\Psi_{ij}+\a_{j2}\phi^\epsilon_j\right)}_{-\hal}^2+\abs{\bp^\ell G_2}_0^2+\norm{\left[\bar\partial^{\ell},E\nabla\right] q}_{0}^2\\&+\norm{\epsilon\bp^\ell(\a_{ik}\pa_j\Psi_{ij})+\bp^\ell(\a_{jk}\phi^\epsilon_j)}_0^2.
\end{split}
\end{equation}
Here, we use H\"older's inequality, the Sobolev embedding theorem $\abs{f}_{L^4}\ls\abs{f}_{\hal}$, and the trace theorem to obtain 
\begin{equation*}
	\intb \bp^\ell\nabla\eta\nabla^2\eta\bp^\ell q \ls \abs{\bp^\ell\nabla\eta\nabla^2\eta}_{L^{\frac{4}{3}}}\abs{\bp^\ell q}_{\hal}\ls \abs{\bp^\ell\nabla\eta\nabla^2\eta}_{L^{\frac{4}{3}}}\norm{\bp^\ell q}_{1}.
\end{equation*}

We then estimate the right-hand side of \eqref{dddddddd}. First, we have
\begin{equation*}
	\begin{split}
		&\norm{g_1}_{L^{\frac{4}{3}}}^2+\norm{\bp^\ell\nabla\eta\nabla^3\eta}_{L^{\frac{4}{3}}}^2+\norm{\bp^\ell\nabla\eta\nabla^2\eta}_0^2\\\ls& \norm{\bp^\ell\nabla v}_0^2\norm{\nabla v}_{L^4}^2+\norm{\bp^\ell\nabla\eta}_{L^4}^2\norm{\nabla^3\eta}_{0}^2+(\ell-1)\norm{\bp\nabla^2\eta}_{L^4}^2\norm{\bp\nabla^2\eta}_{0}^2\\&+(\ell-1)\norm{\bp\pa_1v}_{L^4}^2\norm{\bp\nabla v}_{0}^2+\norm{\bp^\ell\nabla\eta}_{L^4}^2\norm{\nabla^2\eta}_{L^4}^2\\\ls &\left(\norm{\bp^\ell\nabla v}_{0}^2+\norm{\bp^\ell\nabla\eta}_{1}+(\ell-1)\norm{\bp^2\pa_1\eta}_{1}^2+(\ell-1)\norm{\bp\nabla^3\eta}_0\right)P\left(\norm{\eta}_{\X^3}\right).
	\end{split}
\end{equation*}
Here, we use Ladyzhenskaya's inequality
\begin{equation*}
	\norm{f}_{L^4}\ls \norm{f}_0^{\frac{1}{2}}\norm{\nabla f}_0^{\frac{1}{2}}
\end{equation*}
to obtain
\begin{equation*}
	\begin{split}
		\norm{\bp^\ell\nabla\eta}_{L^4}^2\ls \norm{\bp^\ell\nabla\eta}_0\norm{\bp^\ell\nabla^2\eta}_0, 
 \norm{\bp\nabla^2\eta}_{L^4}^2\ls \norm{\bp\nabla^2\eta}_0\norm{\bp\nabla^3\eta}_0.
\end{split}
\end{equation*}
Next, by using the trace estimate \eqref{tre} and $\abs{f}_{L^4}\ls\abs{f}_{\hal}\ls\norm{f}_1$, we obtain
\begin{equation*}
	\begin{split}
		\abs{\bp^\ell\nabla\eta\nabla^2\eta}_{L^{\frac{4}{3}}}^2\leq \abs{\bp^\ell\nabla\eta}_0^2\abs{\nabla^2\eta}_{L^4}^2&\ls\left(\norm{\bp^\ell\nabla\eta}_0^2+\norm{\bp^\ell\nabla\eta}_0\norm{\bp^\ell\nabla^2\eta}_0\right)\norm{\eta}_{\X^3}^2\\&\ls \left(1+\norm{\bp^\ell\nabla^2\eta}_{0}\right)P\left(\norm{\eta}_{\X^3}\right).
\end{split}
\end{equation*}

Second, we estimate $\abs{\bp^\ell G_2}_0$ and obtain
\begin{equation}
\label{est2}
\begin{split}
	\abs{\bp^\ell G_2}_0^2 &\ls \abs{\pa_1^2\bp^\ell\eta_i\a_{i2}}_0^2+\left(\abs{\epsilon  \bp^\ell\pa_1v}_{0}^2+\abs{\bp^\ell\pa_1\eta}_{L^4}^2+(\ell-1)\left(\abs{\bp\pa_1^2\eta}_{L^4}^2+\abs{\epsilon\bp\pa_1 v}_{L^4}^2\right)+1\right)P\left(\norm{\eta}^2_{\X^3}\right)\\&\quad+\left(\norm{\epsilon\nabla\Psi}_{\X^2}^2+\norm{\epsilon\Psi}_0^2\right)\left(\norm{\eta}_{\X^3}^2+\norm{\bp^\ell\pa_1\eta}_1^2\right)\\&\ls \abs{\pa_1^2\bp^\ell\eta_i\a_{i2}}_0^2+\left(\norm{\epsilon  \bp^\ell\pa_1v}_0^2+\norm{\epsilon\bp^\ell\pa_1v}_0\norm{\epsilon\bp^\ell\pa_1\nabla v}_0+\norm{\bp^2\pa_1\eta}_1^2+1\right)P\left(\norm{\eta}^2_{\X^3}\right)\\&\quad+\left(\norm{\epsilon\nabla\Psi}_{\X^2}^2+\norm{\epsilon\Psi}_0^2\right)\left(\norm{\eta}_{\X^3}^2+\norm{\bp^\ell\pa_1\eta}_1^2\right)
\end{split}
\end{equation}
with $\abs{f}_{L^p}\ls\abs{f}_{\hal}\ls\norm{f}_1$ and the trace estimate \eqref{tre}.

Third, we estimate $\abs{\bp^\ell \left(G_3+\epsilon\pa_j\Psi_{ij}\a_{i2}+\a_{j2}\phi^\epsilon_j\right)}_{-\hal}$. To do this, we recall the dynamic boundary condition:
\begin{equation*}
\label{dyb}
-(q+1)\a_{i2}+F_{i2}+2\epsilon S_\eta(v)_{ij}\a_{j2}-\epsilon\Psi_{i2}=\dfrac{\pa_1^2\eta_k\a_{k\ell}N_\ell}{\abs{\pa_1\eta}^3}\a_{i2}
\end{equation*}
and take the projection $\Pi f_i=f_i-\dfrac{f_k\a_{k2}\a_{i2}}{\abs{\a_{\cdot 2}}^2}$ on the above condition. Then, we obtain
\begin{equation*}
	\Pi (F_{i2}+2\epsilon S_{\eta}(v)_{ij}\a_{j2}+\epsilon\Psi_{i2})=0.
\end{equation*}
Thus, we have 
\begin{equation}
\label{pa1}
F_{i2}+\epsilon\pa_{\eta_j}v_i\a_{j2}+\epsilon\Pi\Psi_{i2}=\dfrac{1}{\abs{\a_{\cdot 2}}^2}\left(\a_{i2}-2\epsilon\pa_t\a_{j2}\a_{j2}\a_{i2}\right)+\epsilon\pa_t\a_{i2}
\end{equation}
by using $\pa_t\a_{j2}=-\pa_{\eta_j}v_i\a_{i2}$.

The key observation here is that the right-hand side of \eqref{pa1} contains only tangential derivatives, since $\a_{\cdot 2}=(-\pa_1\eta_2, \pa_1\eta_1)^{\text{T}}$. On the other hand, we have $G_3+\epsilon\pa_j\Psi_{ij}\a_{i2}+\a_{j2}\phi^\epsilon_j=\Delta\eta_i\a_{i2}+\epsilon J \Delta_\eta v_i\a_{i2}$, which can be calculated  as
\begin{equation}
\label{bde}
\begin{split}
	&\Delta \eta_i\a_{i2}+\epsilon J\Delta_{\eta}v_i\a_{i2}\\=&\pa_2J-\pa_1\eta_i\pa_1\a_{i2}-\a_{i1}\pa_1F_{i2}+\epsilon\a_{k1}\pa_1(\pa_{\eta_k}v_i)\a_{i2}+\epsilon\a_{k2}\pa_2(\pa_{\eta_k}v_i)\a_{i2}\\=&\pa_2J-\pa_1\eta_i\pa_1\a_{i2}-\a_{i1}\pa_1F_{i2}-\epsilon\a_{k1}\pa_1\pa_t\a_{k2}-\epsilon\a_{k1}\pa_1\a_{i2}\pa_{\eta_k}v_i\\&+\epsilon J\pa_{\eta_i}(\pa_{\eta_k}v_i)\a_{k2}-\epsilon\a_{i1}\pa_1(\pa_{\eta_k}v_i)\a_{k2}\\=&\pa_2J-\pa_1\eta_i\pa_1\a_{i2}-\a_{i1}\pa_1F_{i2}-\epsilon\a_{k1}\pa_1\pa_t\a_{k2}-\epsilon\a_{k1}\pa_1\a_{i2}\pa_{\eta_k}v_i\\&-\epsilon\a_{i1}\pa_1(\pa_{\eta_k}v_i\a_{k2})+\epsilon\a_{i1}\pa_1\a_{k2}\pa_{\eta_k}v_i\\=&\pa_2J-\pa_1\eta_i\pa_1\a_{i2}-\a_{i1}\pa_1(F_{i2}+\epsilon\pa_{\eta_k}v_i\a_{k2})-\epsilon\a_{k1}\pa_1\pa_t\a_{k2}-\epsilon\a_{k1}\pa_1\a_{i2}(\pa_{\eta_k}v_i-\pa_{\eta_i}v_k).
\end{split}
\end{equation}
Then, we estimate the terms of \eqref{bde} one by one. Noticing that $J=J_0^\epsilon(x)$ again, we have
\begin{equation*}
	\abs{\bp^\ell\pa_2 J}_{-\hal}^2\ls \norm{\pa_1\pa_2 J}_1^2\ls M_0.
\end{equation*}
Next, using the estimate \eqref{co123}, we have
\begin{equation*}
\begin{split}
\abs{\bp^\ell(\pa_1\eta_i\pa_1\a_{i2})}_{-\hal}^2&\ls \abs{\bp^\ell\pa_1^2\eta}_{-\hal}^2\abs{\pa_1\eta}_{1}^2+\abs{\bp^\ell\pa_1\eta}_{-\hal}^2\abs{\pa_1^2\eta}_{1}^2+\abs{\left[\bp^\ell, \pa_1^2\eta,\pa_1\eta\right]}_{-\hal}^2\\&\ls P\left(\norm{\bp^2\pa_1\eta}_1^2,\norm{\eta}_{\X^3}^2\right).
\end{split}
\end{equation*}
Using the expression \eqref{pa1} and the estimate \eqref{estforpsi},
\begin{equation*}
	\begin{split}
		&\abs{\bp^\ell\left(\a_{i1}\pa_1\left(\dfrac{1}{\abs{\a_{\cdot 2}}^2}\left(\a_{i2}-2\epsilon\pa_t\a_{j2}\a_{j2}\a_{i2}\right)+\epsilon\pa_t\a_{i2}-\epsilon\Pi\Psi_{i2}\right)\right)}_{-\hal}^2\\&\ls \abs{\bp^\ell(\nabla\eta \pa_1^2\eta\pa_1\eta)}_{-\hal}^2+\abs{\epsilon\bp^\ell(\nabla\eta\pa_1^2v\pa_1\eta)}_{-\hal}^2+\abs{\epsilon\bp^\ell(\nabla\eta\pa_1v\pa_1^2\eta\pa_1\eta)}_{-\hal}^2+\abs{\epsilon\bp^\ell(\a_{i1}\pa_1\Pi \Psi_{i2})}_{-\hal}^2\\&\ls\abs{\bp^\ell\pa_1^2\eta}_{-\hal}^2P\left(\norm{\eta}_{\X^3}^2,\norm{\epsilon \pa_1^2\eta}_{\X^2}^2\right)+\abs{\sqrt\epsilon\bp^\ell\nabla\eta}_{L^4}^2\abs{\sqrt\epsilon\pa_1^2 v+\sqrt\epsilon\pa_1^2\eta\pa_1v}_{L^4}^2P\left(\norm{\eta}_{\X^3}^2\right)\\&\quad+\abs{\epsilon\bp^\ell\pa_1^2v}_{-\hal}^2P\left(\norm{\eta}_{\X^3}^2\right)+\abs{\bp^\ell\nabla\eta}_{0}^2\abs{\pa_1^2\eta}_{1}^2P\left(\norm{\eta}_{\X^3}^2\right)+\abs{\sqrt\epsilon\bp^\ell\pa_1v}_{-\hal}^2\abs{\sqrt\epsilon\pa_1^2\eta}_1^2P\left(\norm{\eta}_{\X^3}^2\right)\\&\quad+\abs{\left[\bp^\ell,\nabla\eta,\pa_1^2\eta+\epsilon\pa_1^2v+\epsilon\pa_1^2\eta\pa_1v\right]}_0^2+\abs{\sqrt\epsilon\bp^\ell\nabla\eta}_{L^4}^2\abs{\sqrt\epsilon\pa_1\Pi\Psi}_{L^4}^2+\abs{\sqrt\epsilon\nabla\eta}_1^2\abs{\bp^\ell\pa_1(\Pi\Psi)}_{-\hal}^2\\&\ls P\left(\norm{\sqrt\epsilon\nabla^2\eta}_{\X^2}^2, \norm{\eta}_{\X^3}^2,\norm{\bp^2\pa_1\eta}_{1}^2\right)\left(1+\norm{\sqrt\epsilon\bp^\ell\nabla v}_0^2+\norm{\bp^\ell\nabla\eta}_{0}\norm{\bp^\ell\nabla^2\eta}_{0}+\norm{\bp\pa_1\nabla v}_0^2\right)\\&\quad+\epsilon P\left(\norm{\eta}_{\X^3}^2\right)\norm{\sqrt\epsilon\bp^\ell\pa_1v}_1^2+P\left(\norm{\eta}_{\X^3}^2,\norm{\sqrt\epsilon\nabla^2\eta}_{\X^2}^2\right)\left(\norm{\sqrt\epsilon\nabla\Psi}_{\X^2}^2+\norm{\sqrt\epsilon\Psi}_0^2\right).
	\end{split}
\end{equation*}
The remaining terms of \eqref{bde} can be estimated as:
\begin{equation*}
	\begin{split}
		\abs{\bp^\ell(\epsilon\a_{k1}\pa_1\pa_t\a_{k2})}_{-\hal}^2&\ls \abs{\sqrt\epsilon\bp^\ell\pa_1^2v}_{-\hal}^2\abs{\sqrt\epsilon\nabla\eta}_{1}^2+\abs{\sqrt\epsilon\pa_1^2v}_{L^4}^2\abs{\sqrt\epsilon\bp^\ell\nabla\eta}_{L^4}^2\\&\quad+\abs{\left[\bp^\ell,\sqrt\epsilon\pa_1^2v,\sqrt\epsilon\nabla\eta\right]}_{0}^2\\&\ls \epsilon\norm{\eta}_{\X^3}^2\norm{\sqrt\epsilon\bp^\ell\pa_1 v}_1^2+P\left(\norm{\eta}_{\X^3}^2,\norm{\sqrt\epsilon\nabla^2\eta}_{\X^2}^2\right),
	\end{split}
\end{equation*}
and
\begin{equation*}
\begin{split}
&\abs{\bp^\ell(\epsilon\a_{k1}\pa_1\a_{i2}\pa_{\eta_k}v_i)}_{-\hal}^2\\&\ls \abs{\epsilon\bp^\ell\nabla v}_{0}^2\abs{\nabla\eta\pa_1^2\eta\nabla\eta}_{1}^2+\abs{\bp^\ell\pa_1^2\eta}_{-\hal}^2\abs{\epsilon\nabla\eta\nabla v\nabla\eta}_{1}^2+\abs{\sqrt\epsilon\bp^\ell\nabla\eta}_{L^4}^2\abs{\sqrt\epsilon\pa_1^2
\eta\nabla v\nabla\eta}_{L^4}^2\\&\quad+\abs{\left[\bp^\ell,\pa_1^2\eta, \epsilon\nabla\eta\nabla v\nabla\eta\right]}_{0}^2\\&\ls \left(\norm{\epsilon\bp^\ell\nabla v}_{0}^2+\norm{\epsilon\bp^\ell\nabla v}_0\norm{\epsilon\bp^\ell\nabla^2v}_0+\norm{\sqrt\epsilon\nabla^2\eta}_{\X^2}^2+1\right)P\left(\norm{\eta}_{\X^3}^2,\norm{\bp^2\pa_1\nabla\eta}_0^2\right).
\end{split}
\end{equation*}
Here we use $\abs{f}_{-\hal}\leq \abs{f}_0, \abs{f}_{L^p}\ls\abs{f}_{\hal}\leq \norm{f}_1$ and the estimate \eqref{peadv}.


Last, we have the estimate that
\begin{equation*}
\begin{split}
	&\abs{\bp^\ell(\pa_tv_i\a_{i2})}_{-1/2}^2\\&\ls \abs{\bp^\ell\pa_t v_i\a_{i2}}_{-\hal}^2+\abs{\pa_tv_i\bp^\ell\a_{i2}}_{-\hal}^2+\abs{\left[\bp^\ell, \pa_t v,\pa_1\eta\right]}_0^2\\&\ls\norm{\eta}_{\X^3}^2\left(\norm{\bp^\ell\pa_tv}_0^2+\norm{J\Div_\eta\bp^\ell\pa_t v}_0^2\right)+\norm{\eta}_{\X^3}^2\norm{\bp^2\pa_1\eta}_1^2+(\ell-1)\norm{\bp\pa_tv}_1^2\norm{\eta}_{\X^3}^2\\&\ls \norm{\eta}_{\X^3}^2\left(\norm{\bp^\ell\pa_tv}_0^2+\norm{\bp^\ell\pa_t\nabla\eta}_0^2\norm{\nabla v}_{L^\infty}^2+\norm{\bp^2\pa_1\eta}_1^2+\norm{\bp\nabla\eta}_{L^\infty}^2\right)\\&\quad+(\ell-1)\norm{\bp^3\eta}_1^2\left(\norm{\eta}_{\X^3}^2+\norm{\bp\nabla\eta}_{L^\infty}^2+\norm{\bp^2\nabla \eta}_{L^3}^2\right)\\&\ls \norm{\eta}_{\X^3}^2\left(\norm{\bp^\ell\pa_t v}_0^2+\norm{\bp^\ell\pa_t\nabla\eta}_0^2\norm{\nabla v}_{L^\infty}^2+\norm{\bp^2\pa_1\eta}_1^2+\norm{\bp\nabla\eta}_{L^\infty}^2\right)\\&\quad+(\ell-1)\norm{\bp^3\eta}_1^2\left(\norm{\eta}_{\X^3}^2+\norm{\nabla\eta}_{\X^3}^{2/3}\norm{\eta}_{\X^3}^{4/3}\right),
\end{split}
\end{equation*} 
where we use the divergence-free condition $\Div_\eta v=J^{-1}(\pa_1\eta\nabla v+\nabla\eta\pa_1v) =0$, the Sobolev embedding inequality $\norm{f}_{L^p}\ls \norm{f}_1$, and $\norm{f}_{L^\infty}\ls\norm{f}_{W^{1,3}}\ls \norm{f}_1+\norm{\nabla f}_0^{2/3}\norm{\nabla^2f}_0^{1/3}$.

Now, we arrive at
\begin{equation}
\label{est3}
\begin{split}
	&\abs{\bp^\ell \left(G_3+\epsilon\pa_j\Psi_{ij}\a_{i2}+\a_{j2}\phi^\epsilon_j\right)}_{-\hal}^2\\&\ls \norm{\bp^\ell\pa_tv}_0^2+ P\left(\norm{\bp^2\pa_1\eta}_1^2, \norm{\sqrt\epsilon\nabla^2\eta}_{\X^2}^2, \norm{\eta}_{\X^3}^2\right)\left(1+\norm{\sqrt\epsilon\bp^\ell\nabla v}_0^2+\norm{\bp\pa_1\nabla v}_0^2\right)\\&\quad+P\left(\norm{\bp^2\pa_1\eta}_1^2, \norm{\sqrt\epsilon\nabla^2\eta}_{\X^2}^2, \norm{\eta}_{\X^3}^2\right)\left(\norm{\bp^\ell\nabla\eta}_{0}\norm{\bp^\ell\nabla^2\eta}_{0}+\norm{\epsilon\bp^\ell\nabla v}_0^2+\norm{\epsilon\bp^\ell\nabla v}_0\norm{\epsilon\bp^\ell\nabla^2v}_0\right)\\&\quad +\norm{\eta}_{\X^3}^2\left(\norm{\bp^\ell\pa_t\nabla\eta}_0^2\norm{\nabla v}_{L^\infty}^2+\norm{\bp\nabla\eta}_{L^\infty}^2\right)\\&\quad +(\ell-1)\norm{\bp^3\eta}_1^2\norm{\nabla\eta}_{\X^3}^{2/3}\norm{\eta}_{\X^3}^{4/3}+\epsilon P\left(\norm{\eta}_{\X^3}^2\right)\norm{\sqrt\epsilon\bp^\ell\pa_1 v}_1^2.
\end{split}
\end{equation}
For $\norm{\bp^\ell(\a_{jk}\phi^\epsilon_j)}_0^2$, we have
\begin{equation*}
	\norm{\bp^\ell(\a_{jk}\phi^\epsilon_j)}_0^2 \ls \norm{\eta}_{\X^3}^2\norm{\phi^\epsilon}_2^2 \ls P\left(\norm{\eta}_{\X^3}^2\right).
\end{equation*}
Finally, we estimate the commutator term. For $\ell=1$, we have
\begin{equation*}\label{qqqq10}
	\begin{split}
		\norm{[\bar\partial, E\nabla]q}_0^2 \ls \norm{\bp E}_{L^\infty}^2\norm{\nabla q}_{0}^2\ls \norm{\nabla\bp\eta}_{L^\infty}^2P\left(\norm{\eta}_{\X^3}^2\right)\ls P\left(\norm{\eta}_{\X^3}^2\right)\left(1+\norm{\eta}_{\X^3}\norm{\nabla\eta}_{\X^3}\right),		
\end{split}
\end{equation*}
Then, combining the above estimates, we arrive at
\begin{equation}
\label{q1}
	\begin{split}
	\norm{\bp q}_1^2\ls \left(\norm{\nabla\bp\eta}_{L^\infty}^2+1\right)P\left(\norm{\eta}_{\X^3}^2,\norm{\sqrt\epsilon\nabla^2\eta}_{\X^2}^2,\norm{\bp^2\pa_1\eta}_1^2,\abs{\bp\pa_1^2\eta_i\a_{i2}}_0^2\right).
\end{split}
\end{equation}
On the other hand, equation \eqref{ell} gives
\begin{equation}
\label{np}
-\partial_{22}q=\dfrac{1}{E_{22}}\left(G_1+\sum_{i+j\neq4}\partial_i(E_{ij}\partial_jq)+\partial_2E_{22}\partial_2q+\epsilon\a_{ik}\pa_{kj}\Psi_{ij}+\a_{jk}\pa_k\phi^\epsilon_j\right).
\end{equation}
This implies that we can estimate the normal derivatives of $q$ in terms of the $q$ terms with less normal derivatives. Hence, using equation \eqref{np} and estimates \eqref{pee} and \eqref{q1}, we obtain
\begin{equation}
\label{ofe}
		\begin{split}
		\norm{\pa_2^2q}_0^2&\ls  \left(\norm{G_1}_0^2+\norm{\nabla\eta}_{L^\infty}^2\norm{\bp\nabla q}_0^2+\norm{\nabla E}_{L^4}^2\norm{\pa_2q}_{0}\norm{\nabla\pa_2 q}_0+\norm{\nabla\eta}_{L^4}^2\norm{\epsilon\nabla^2\Psi+\nabla\phi^\epsilon}_{L^4}^2\right)\\&\leq P\left(\norm{\eta}_{\X^3}^2\right)\left(1+\norm{\bp q}_1^2\right)+\hal \norm{\pa_2^2q}_0^2.
\end{split}
\end{equation}
Thus, we arrive at
\begin{equation}
\label{est1p}
	\begin{split}
		\norm{\nabla q}_1^2&\ls \left(\norm{\nabla\bp\eta}_{L^\infty}^2+1\right)P\left(\norm{\eta}_{\X^3}^2,\norm{\sqrt\epsilon\nabla^2\eta}_{\X^2}^2,\norm{\bp^2\pa_1\eta}_1^2,\abs{\bp\pa_1^2\eta_i\a_{i2}}_0^2\right)\\&\ls \left(\norm{\eta}_{\X^3}\norm{\nabla\eta}_{\X^3}+1\right)P\left(\norm{\eta}_{\X^3}^2,\norm{\sqrt\epsilon\nabla^2\eta}_{\X^2}^2,\norm{\bp^2\pa_1\eta}_1^2,\abs{\bp\pa_1^2\eta_i\a_{i2}}_0^2\right).
\end{split}
\end{equation}
Integrating from $0$ to $t$, \eqref{q1} and \eqref{est1p} yield
\begin{equation}
	\int_0^t\norm{\nabla q}_{\X^1}^2\ls \sqrt TP\left(\sup_{t\in [0,T]}\mathfrak E^\epsilon(t)\right).
\end{equation}
For $\ell=2$, we obtain
\begin{equation*}
	\begin{split}
		\norm{[\bar\partial^2,E\nabla] q}_{0}^2&\ls \norm{\bp E\bp\nabla q}_0^2+\norm{\bp^2 E\nabla q}_0^2\\&\ls \left(\norm{\nabla\bp\eta}_{L^\infty}^4+\norm{\nabla\bp\eta}_{L^\infty}^2\right)P\left(\norm{\eta}_{\X^3}^2,\norm{\sqrt\epsilon\nabla^2\eta}_{\X^2}^2,\norm{\bp^2\pa_1\eta}_1^2,\abs{\bp\pa_1^2\eta_i\a_{i2}}_0^2\right)\\&\quad+\norm{\nabla\bp^2\eta}_{0}\norm{\nabla\eta}_{\X^3}\norm{q}_1\norm{\nabla q}_1.
\end{split}
\end{equation*}
Combining the estimates \eqref{est2} and \eqref{est3} and using  $\norm{f}_{L^\infty}\ls\norm{f}_{W^{1,3}}\ls \norm{f}_1+\norm{\nabla f}_0^{2/3}\norm{\nabla^2f}_0^{1/3}$, we obtain
\begin{equation}
\begin{split}
\label{q2}
&\norm{ \bar\partial^2 q}_{1}^2 \\& \ls  \Bigg(\norm{\bp^3\eta}_1^2+\norm{\bp^2\pa_tv}_0^2+1\Bigg)P\left(\norm{\eta}_{\X^3}^2,\norm{\sqrt\epsilon\nabla^2\eta}_{\X^2}^2,\norm{\bp^2\pa_1\eta}_1^2,\abs{\bp^2\pa_1^2\eta_i\a_{i2}}_0^2\right)\\&\quad+ \left(\norm{\nabla\bp\eta}_{L^\infty}^4+\norm{\nabla\bp\eta}_{L^\infty}^2\right)P\left(\norm{\eta}_{\X^3}^2,\norm{\sqrt\epsilon\nabla^2\eta}_{\X^2}^2,\norm{\bp^2\pa_1\eta}_1^2, \abs{\bp\pa_1^2\eta_i\a_{i2}}_0^2\right)\\&\quad+\norm{\nabla\bp^2\eta}_{0}\norm{\nabla\eta}_{\X^3}\norm{q}_1\norm{\nabla q}_1+\epsilon P\left(\norm{\eta}_{\X^3}^2\right)\norm{\sqrt\epsilon\bp^2\pa_1\nabla v}_0^2\\&\quad +P\left(\norm{\eta}_{\X^3}^2\right)\left(\norm{\bp^3\eta}_1^2+1\right)\norm{\nabla\eta}_{\X^3}^{2/3}\norm{\eta}_{\X^3}^{4/3}+\delta\left(\norm{\nabla\eta}_{\X^3}^2+\norm{\epsilon\nabla v}_{\X^3}^2\right)
\\&\ls \Bigg(\norm{\bp^3\eta}_1^2+\norm{\bp^2\pa_tv}_0^2+\norm{\bp^3\eta}_1^3+1\Bigg)P\left(\norm{\eta}_{\X^3}^2,\norm{\sqrt\epsilon\nabla^2\eta}_{\X^2}^2,\norm{\bp^2\pa_1\eta}_1^2,\abs{\bp^2\pa_1^2\eta_i\a_{i2}}_0^2\right)\\&\quad+\norm{\nabla\eta}_{\X^3}\norm{\nabla q}_1P\left(\norm{\eta}_{\X^3}\right)+\epsilon P\left(\norm{\eta}_{\X^3}^2\right)\norm{\sqrt\epsilon\bp^2\pa_1\nabla v}_0^2+\delta\left(\norm{\nabla\eta}_{\X^3}^2+\norm{\epsilon\nabla v}_{\X^3}^2\right),
\end{split}
\end{equation}
where Young's inequality is also employed.
Thus, for sufficient small $\epsilon$ such that $\epsilon P(M_0)\leq C\delta$, we have 
\begin{equation*}
	\int_0^t\norm{\bp^2q}_1^2\ls T^{1/4}P\left(\sup_{t\in [0,T]}\mathfrak E^\epsilon(t)\right)+\delta \sup_{t\in [0,T]}\mathfrak E^\epsilon(t) .
\end{equation*}
Finally, by using equation \eqref{np} and a similar approach as in \eqref{ofe}--\eqref{est1p}, we can obtain the estimate of $\norm{\bp\nabla^2q}_0^2, \norm{\nabla^3q}_0^2$ and prove the lemma by adjusting $0<\delta <1$.

\end{proof}

After we have the estimates for pressure, we can use the Dirichlet boundary condition \eqref{bdq} of the pressure to bound $\bp^2\pa_1^2\eta_i\a_{i\ell}N_\ell$ on $\Gamma$:
\begin{equation*}
	-\bp^2\pa_1^2\eta_i\a_{i\ell}N_\ell=\bp^2\left(\abs{\a_{\cdot 2}}^3\left(q+1-\dfrac{1+\epsilon\pa_t\a_{j2}\a_{j2}+\epsilon\Psi_{i2}\a_{i2}}{\abs{\a_{\cdot2}}^2}\right)\right)-\sum_{\abs{\alpha}<2}C_\alpha\bp^{2-\alpha}\a_{i2}\bp^\alpha\pa_1^2\eta_i
\end{equation*}
By using the trace theorem $\abs{f}_\hal \ls \norm{f}_1$, we have
\begin{equation*}
	\begin{split}
		\abs{\bp^2\pa_1^2\eta_i\a_{i\ell}N_\ell}_{\hal}&\ls \norm{\bp^2\left(\abs{\pa_1\eta}^3\left(q+1-\dfrac{1+\epsilon\pa_t\a_{j2}\a_{j2}+\epsilon\Psi_{i2}\a_{i2}}{\abs{\pa_1\eta}^2}\right)\right)-\sum_{\abs{\alpha}<2}\bp^{2-\alpha}\pa_1\eta\bp^\alpha\pa_1^2\eta}_1\\&\ls P\left(\norm{\eta}_{\X^3}\right)\left(\norm{\bp^2\pa_1\eta}_1\norm{q}_{2}+\norm{\bp q}_2+\norm{\bp^2 q}_1\right)+\norm{\bp\pa_1\eta}_2\norm{\bp^2\pa_1\eta}_1\\&\quad+P(\norm{\eta}_{\X^3})\left(\norm{\bp^2\pa_1\eta}_1\norm{\epsilon\pa_1 v}_2+\norm{\sqrt\epsilon\pa_1v}_1+\norm{\sqrt\epsilon\bp^2\pa_1 v}_1\right)
\end{split}
\end{equation*}
Thus, we arrive at
\begin{equation}
	\label{sigmae}
	\int_0^t\abs{\bp^2\pa_1^2\eta_i\a_{i\ell}N_\ell}_{\hal}^2(\tau)\,d{\mathsf t} \ls P\left(\sup_{t\in [0,T]}\mathfrak E^\epsilon(t)\right).
\end{equation}
This estimate will be used later in the high-order tangential energy estimates.

\subsection{Basic energy estimates}
We have the following basic $L^2$ energy estimates:
\begin{proposition}\label{basic}
For $t\in [0,T]$ with $T\le T_\epsilon$,
\begin{equation}\label{00estimate}
	\norm{v(t)}_0^2+\norm{\nabla\eta(t)}_0^2+2\intb \abs{\pa_1\eta}\,d\sigma+\epsilon\int_0^t\norm{\nabla v}_0^2 \ls M_0+TP\left(\sup_{t\in [0,T]}\mathfrak E^\epsilon\right).
\end{equation}
\end{proposition}
\begin{proof}
Taking the $L^2(\Omega)$ inner product of equation \eqref{approximate}$_2$ with $v$ yields
\begin{equation}\label{hhl1}
	\dfrac{1}{2}\dfrac{d}{dt}\int_{\Omega}\abs{v}^2 +\int_{\Omega} J\nabla_{\eta} q\cdot v -\int_{\Omega}\Delta\eta \cdot v-2\epsilon\intd J\pa_{\eta_j}(S_\eta(v)_{ij})v_i+\intd \epsilon \pa_j\Psi_{ij}v_i = -\intd \phi^\epsilon\cdot v
\end{equation}
By integration by parts and using \eqref{approximate}$_3$ and \eqref{approximate}$_4$, we have
\begin{equation*}
\begin{split}
&\int_{\Omega} J\nabla_{\eta} q\cdot v-\int_{\Omega}\Delta \eta\cdot v-2\epsilon\intd J\pa_{\eta_j}(S_\eta(v)_{ij})v_i+\intd \epsilon \pa_j\Psi_{ij}v_i \\=&\int_{\Gamma} q \a_{i\ell}N_\ell v_i\,d\sigma-\int_{\Gamma}F_{i\ell}N_\ell v_i\,d\sigma
+\int_{\Omega}\partial_j\eta_i\partial_j v_i+2\epsilon\intd S_\eta(v)_{ij}J\pa_{\eta_j} v_i-\intb 2\epsilon S_\eta(v)_{ij}\a_{j\ell}N_\ell v_i\,d\sigma\\&-\intd \epsilon\Psi_{ij}\pa_jv_i+\intb \epsilon \Psi_{i\ell}N_\ell v_i\,d\sigma\\=&\dfrac{1}{2}\dfrac{d}{dt}\int_{\Omega}|\nabla\eta|^2-\int_{\Gamma}\partial_1\left(\dfrac{\partial_1\eta_i}{\abs{\pa_1\eta}}\right)v_i\,d\sigma-\intb \a_{i\ell}N_\ell v_i\, d\sigma+\epsilon\intd J\abs{S_\eta(v)}^2-\intd \epsilon\Psi_{ij}\pa_jv_i\\
=&\dfrac{1}{2}\dfrac{d}{dt}\int_{\Omega}|\nabla\eta|^2+\int_{\Gamma}\dfrac{\partial_1\eta_i\pa_1v_i}{\abs{\pa_1\eta}}+\epsilon\intd J\abs{S_\eta(v)}^2-\intd \epsilon\Psi_{ij}\pa_jv_i\\=&\dfrac{1}{2}\dfrac{d}{dt}\int_{\Omega}|\nabla\eta|^2+\dfrac{d}{dt}\int_{\Gamma} \abs{\pa_1\eta}\,d\sigma+\epsilon\intd J\abs{S_\eta(v)}^2-\intd \epsilon\Psi_{ij}\pa_jv_i
\end{split}
\end{equation*}
Thus, we arrive at 
\begin{equation*}
	\begin{split}
		&\frac{1}{2}\dfrac{d}{dt}\left(\int_{\Omega}\abs{v}^2+\abs{\nabla\eta}^2 +2\int_{\Gamma} \abs{\pa_1\eta}\,d\sigma\right)+\epsilon\intd J\abs{S_\eta(v)}^2\ls \norm{\phi^\epsilon}_0\norm{v}_0+\norm{\sqrt\epsilon\Psi}_0\norm{\sqrt\epsilon \nabla v}_0.
\end{split}
\end{equation*}
Integrating directly over time in the above inequality and using Korn's inequality \eqref{korn} yields \eqref{00estimate}.
\end{proof}

%

\subsection{High-order tangential energy estimates}\label{tan}
We now derive the high-order tangential energy estimates. These estimates contain two parts: one part addresses the case in which the tangential derivatives have at least one $\pa_1$ derivative, and the other part addresses the case in which all tangetial derivatives are $\pa_t$ derivatives. The reason for doing this is mainly that there are no, estimates $\pa_t^3q$ in Proposition \ref{press2} and hence estimate \eqref{sigmae} cannot be applied to the $\pa_t^3$-energy estimates; more details can be found in \cite{Wang_15}.
\subsubsection{$\bp^2\pa_1$ energy estimates}
\begin{proposition}\label{tane1}
For $t\in [0,T]$ with $T\le T_\epsilon$, for $\epsilon$ sufficiently small, it holds that
\begin{equation}
\label{teee}
\begin{split}
	\norm{\bar\partial^2\pa_1 v(t)}_0^2+\norm{\bar\partial^2\pa_1\nabla\eta(t)}_0^2+\dfrac{1}{2}\abs{\pa_1^2\bar\partial^2\eta_i \a_{i2}(t)}_0^2+\epsilon\int_0^t\norm{\nabla\bp^2\pa_1 v}_0^2\,d{\mathsf t}\\\leq M_0+\delta\sup_{t\in[0,T]}\mathfrak{E}^{\epsilon}(t)+T^{\frac{1}{4}} P\left(\sup_{t\in[0,T]}\mathfrak{E}^{\epsilon}(t)\right).
\end{split}
\end{equation}
\end{proposition}
\begin{proof}
Applying $\bp^2\pa_1$ to equation \eqref{approximate}$_{2}$ and then taking the $L^2(\Omega)$ inner product of equation \eqref{eqValpha} with $\bp^2\pa_1v$ yields
\begin{equation*}
	\begin{split}
	\dfrac{1}{2}\dfrac{d}{dt}\int_{\Omega}\abs{\bp^2\pa_1 v}^2+\int_{\Omega} J\nabla_\eta \bp^2\pa_1q \cdot \bp^2\pa_1 v-\int_{\Omega} \Delta\bp^2\pa_1\eta\cdot \bp^2\pa_1v-2\epsilon\intd \pa_k(\bp^2\pa_1(S_\eta(v)_{ij}\a_{jk}))\bp^2\pa_1v_i\\+\intd\epsilon\pa_k(\bp^2\pa_1\Psi_{ik})\bp^2\pa_1 v_i+\intd\bp^2\pa_1\phi^\epsilon_i\bp^2\pa_1v_i=\int_{\Omega} \left[\bp^2\pa_1,\a_{ij}\right] \pa_j q \bp^2\pa_1 v_i.
 \end{split}
\end{equation*}
By the estimates \eqref{co1} and \eqref{press2}, we have
\begin{equation}
\label{ggg0}
\begin{split}
	\int_0^t\intd\left[\bp^2\pa_1,\a_{ij}\right]\pa_j q \bp^2\pa_1 v_i &\ls\int_0^t\norm{\left[\bp^2\pa_1,\a_{ij}\right]\pa_j q}_0\norm{\bp^2\pa_1 v}_0\\&\ls \int_0^t \left(P\left(\norm{\bp^2\pa_1\nabla\eta}_0, \norm{\eta}_{\X^3}\right)+\norm{\nabla\eta}_{\X^3}^\hal\right)\norm{\nabla q}_{\X^2}\norm{\bp^2\pa_1v}_0\\&\ls \delta \sup_{t\in [0,T]} \mathfrak E^{\epsilon}(t)+ T^{\frac{1}{4}}P\left(\sup_{t\in [0,T]} \mathfrak E^{\epsilon}(t)\right).
\end{split}
\end{equation}
Next, by the fact that $v=\partial_t\eta$, we have
\begin{equation*}
\begin{split}
&\int_{\Omega} J\nabla_\eta \bp^2\pa_1q \cdot \bp^2\pa_1 v-\int_{\Omega} \Delta\bp^2\pa_1\eta\cdot \bp^2\pa_1v-2\epsilon\intd \pa_k(\bp^2\pa_1(S_\eta(v)_{ij}\a_{jk}))\bp^2\pa_1v_i\\&\quad+\intd\epsilon\pa_k(\bp^2\pa_1\Psi_{ik})\bp^2\pa_1 v_i\\= &\int_{\Omega} \bp^2\pa_1q \a_{ik}\pa_k(\bp^2\pa_1v_i)+\int_{\Gamma} \bp^2\pa_1q \a_{i\ell}N_\ell\bp^2\pa_1v_i\,d\sigma+\int_{\Omega}  \bp^2\pa_1\partial_j\eta_i \partial_j\bp^2\pa_1v_i \\&\quad-\int_{\Gamma}\partial_\ell\bp^2\pa_1\eta_i N_\ell\bp^2\pa_1v_i\,d\sigma+\intd \bp^2\pa_1(2\epsilon S_\eta(v)_{ij}\a_{jk}-\epsilon\Psi_{ik})\pa_{k}\bp^2\pa_1v_i\\&\quad-\intb \bp^2\pa_1(2\epsilon S_\eta(v)_{ij}\a_{j\ell}-\epsilon\Psi_{i\ell})N_\ell\bp^2\pa_1v_i\,d\sigma\\=&\dfrac{1}{2}\dfrac{d}{dt}\int_{\Omega}\abs{\nabla\bp^2\pa_1\eta}^2+\underbrace{\int_{\Gamma}\left(\bp^2\pa_1q \a_{i\ell}-\bp^2\pa_1(F_{i\ell}-2\epsilon (S_\eta(v)\a)_{i\ell}+\epsilon\Psi_{i\ell})\right)N_\ell\bp^2\pa_1v_i\,d\sigma}_{R_a}\\&\quad+\underbrace{\intd \bp^2\pa_1q \a_{ik}\pa_k\bp^2\pa_1v_i}_{R_b}+\underbrace{\intd 2\epsilon\bp^2\pa_1(S_\eta(v)_{ij}\a_{jk})\pa_k \bp^2\pa_1v_i}_{R_c}\underbrace{-\intd \epsilon \bp^2\pa_1\Psi_{ik}\pa_k\bp^2\pa_1v_i}_{R_d}
\end{split}
\end{equation*}
For $R_a$, we rewrite the boundary condition \eqref{approximate}$_4$ as
\begin{equation}
\label{tb}
	\left(q+1+\frac{\pa_1^2\eta_k\a_{k\ell}N_\ell}{\abs{\pa_1\eta}^3}\right)\a N= FN+2\epsilon S_\eta(v)\a N-\epsilon \Psi N,
\end{equation}
by calculating the $\pa_1\left(\dfrac{\pa_1\eta}{\abs{\pa_1\eta}}\right)$ directly with $\a_{\cdot 2}=(-\partial_1\eta_2, \partial_1\eta_1)^\text{T}$.

Denoting $\mathfrak B:=q+1+\frac{\pa_1^2\eta_k\a_{k\ell}N_\ell}{\abs{\pa_1\eta}^3}$, we have
\begin{equation}
\label{et2}
\begin{split}
	R_{a}=&\underbrace{-\int_{\Gamma}\bp^2\pa_1\left(\dfrac{\partial_1^2\eta_k\a_{k\ell}N_\ell}{\abs{\pa_1\eta}^3}\right)(\a N)_{i}\bp^2\pa_1v_i\,d\sigma}_{R_{a1}} -\underbrace{\int_{\Gamma}\mathfrak B\bp^2\pa_1(\a_{i\ell}N_\ell)\bp^2\pa_1v_i\,d\sigma}_{R_{a2}}\\&\underbrace{-\intb\left[\bp^2\pa_1, \mathfrak B, \a_{i\ell}N_\ell\right]\bp^2\pa_1v_i\,d\sigma}_{R_{a3}}.
\end{split}
\end{equation}
Substituting the boundary condition \eqref{bdq} into $R_{a3}$, we have
\begin{equation*}
	R_{a3}=-\intb \left[\bp^2\pa_1,\dfrac{1+\epsilon\pa_t\a_{j2}\a_{j2}+\epsilon\Psi_{j2}\a_{j2}}{\abs{\pa_1\eta}^2}, \a_{i\ell}N_\ell\right]\bp^2\pa_1v_i\,d\sigma.
\end{equation*}
Thus, with the help of \eqref{co2} and \eqref{peadv}, we obtain
\begin{equation}
	\label{grgr}
\begin{split}
	\int_0^t R_{a3}\,d{\mathsf t} \ls &\int_0^t\norm{\left[\bp^2\pa_1,\dfrac{1+\epsilon\pa_t\a_{j2}\a_{j2}+\epsilon\Psi_{j2}\a_{j2}}{\abs{\pa_1\eta}^2}, \a_{i\ell}N_\ell\right]}_1\abs{\bp^2\pa_1v_i}_{-1/2}\\\ls & \int_0^t\left(\norm{\nabla\eta}_{\X^3}+\norm{\epsilon\nabla v}_{\X^3}\right)\left(\norm{\bp^3\eta}_1+1\right)P\left(\norm{\bp^2\pa_1\eta}_1,\norm{\eta}_{\X^3},\norm{\sqrt\epsilon\nabla^2\eta}_{\X^2}\right)
	\\&+\int_0^t \norm{\sqrt\epsilon\bp^2\pa_1 v}_1\norm{\sqrt\epsilon\nabla^2\eta}_{\X^2}\left(\norm{\bp^3\eta}_1+\norm{\bp^2\pa_1\eta}_1\norm{\eta}_{\X^3}^2\right)\\\ls& T^{\frac{1}{4}}P\left(\sup_{t\in[0,T]}\mathfrak E^\epsilon(t)\right).
\end{split}
\end{equation}
For $R_{a2}$, we need the following symmetric structure. First, by integration by parts with $t$ and $x_1$, we have
\begin{equation}
\label{pol}
\begin{split}
&\intb\mathfrak B\bp^2\pa_1\a_{i\ell}N_\ell\bp^2\pa_1v_i\,d\sigma\\=&\dfrac{d}{dt}\intb\mathfrak B\bp^2\pa_1\a_{i\ell}N_\ell\bp^2\pa_1\eta_i\,d\sigma-\intb\partial_t\mathfrak B\bp^2\pa_1\a_{i\ell}N_\ell\bp^2\pa_1\eta_i\,d\sigma-\intb\mathfrak B\pa_t\bp^2\pa_1\a_{i\ell}N_\ell\bp^2\pa_1\eta_i\,d\sigma\\=&\dfrac{d}{dt}\intb\mathfrak B\bp^2\pa_1\a_{i\ell}N_\ell\bp^2\pa_1\eta_i\,d\sigma-\intb\pa_t\mathfrak B\bp^2\pa_1\a_{i\ell}N_\ell\bp^2\pa_1\eta_i\,d\sigma+\intb \pa_1\mathfrak B\bp^2\pa_t\a_{i\ell}N_\ell\bp^2\pa_1\eta_i\,d\sigma\\&+\intb\mathfrak B\bp^2\pa_t\a_{i\ell}N_\ell\bp^2\pa_1^2\eta_i\,d\sigma
\end{split}
\end{equation}
With equality \eqref{in2}, we see that
\begin{equation}
\label{pol2}
\intb\mathfrak B\bp^2\pa_t\a_{i\ell}N_\ell\bp^2\pa_1^2\eta_i\,d\sigma=-\intb \mathfrak B \bp^2\pa_t\pa_1\eta_i \bp^2\pa_1\a_{i\ell}N_\ell\,d\sigma.
\end{equation}
Then, we arrive at 
\begin{equation}
	\label{44}
\begin{split}
	R_{a2}=&\dfrac{1}{2}\dfrac{d}{dt}\int_{\Gamma}\mathfrak B\bp^2\pa_1(\a_{i\ell}N_\ell)\bp^2\pa_1\eta_i\,d\sigma\\&\underbrace{-\dfrac{1}{2}\int_{\Gamma}\pa_t\mathfrak B\bp^2\pa_1(\a_{i\ell}N_\ell)\bp^2\pa_1\eta_i\,d\sigma-\hal\intb\pa_1\mathfrak B\bp^2(\a_{i\ell}N_\ell)\bp^2\pa_1v_i\,d\sigma}_{R_{a21}}.
\end{split}
\end{equation}
Furthermore, we have
\begin{equation*}
	\begin{split}
		\int_0^t \abs{R_{a21}}\,d{\mathsf t} &\ls \int_0^t\left(\norm{\bp^2\pa_1\eta}_1+\norm{\bp^2\pa_t\eta}_1\right)P\left(\norm{\bp^2\pa_1\eta}_1,\norm{\eta}_{\X^3},\norm{\sqrt\epsilon\nabla^2\eta}_{\X^2}\right)\left(1+\norm{\epsilon\nabla v}_{\X^3}\right)\\&\ls T^{\frac{1}{4}}P\left(\sup_{t\in[0,T]}\mathfrak E^\epsilon(t)\right).
	\end{split}
\end{equation*}
Now, we turn to the estimate of $R_{a1}$. 
\begin{equation}\label{444}
\begin{split}
R_{a1}=&\underbrace{-\intb \dfrac{1}{\abs{\pa_1\eta}^3}\bp^2\pa_1^3\eta_k\a_{k2}\a_{i2}\bp^2\pa_1v_i\,d\sigma}_{R_{a11}}\underbrace{-\intb \bp^2\pa_1\left(\dfrac{\a_{k2}}{\abs{\pa_1\eta}^3}\right)\pa_1^2\eta_k\a_{i2}\bp^2\pa_1v_i\,d\sigma}_{R_{a12}}\\&\underbrace{-\intb\left[\bp^2\pa_1,\pa_1^2\eta_k,\dfrac{\a_{k2}}{\abs{\pa_1\eta}^3}\right]\a_{i2}\bp^2\pa_1v_i\,d\sigma}_{R_{a13}}
\end{split}
\end{equation}
For $R_{a12}$, we have
\begin{equation*}
	R_{a12}=\intb \bp^2\pa_1^2\eta\dfrac{\pa_1^2\eta}{\abs{\pa_1\eta}^3} \a_{i2}\bp^2\pa_1v_i\,d\sigma+R_L
\end{equation*}
and $\int_0^t R_L\,d{\mathsf t}$ can be bounded by $T^{\frac{1}{4}}P\left(\sup\limits_{t\in[0,T]}\mathfrak E^\epsilon(t)\right)$ in a similar way as \eqref{grgr}. For the remaining term, we have
\begin{equation*}
\begin{split}
&\intb \bp^2\pa_1^2\eta\dfrac{\pa_1^2\eta}{\abs{\pa_1\eta}^3}\a_{i2}\bp^2\pa_1v_i\,d\sigma\\=&\underbrace{\dfrac{d}{dt} \intb \bp^2\pa_1^2\eta\dfrac{\pa_1^2\eta}{\abs{\pa_1\eta}^3}\a_{i2}\bp^2\pa_1\eta_i\,d\sigma}_{R_{a121}}\\&\underbrace{+\intb \bp^2\pa_1v\pa_1\left(\dfrac{\pa_1^2\eta\a_{i2}}{\abs{\pa_1\eta}^3}\right)\bp^2\pa_1\eta_i\,d\sigma-\intb \bp^2\pa_1^2\eta\pa_t\left(\dfrac{\pa_1^2\eta\a_{i2}}{\abs{\pa_1\eta}^3}\right)\bp^2\pa_1\eta_i\,d\sigma}_{R_{a122}}\\&\quad \underbrace{+\intb \bp^2\pa_1v\dfrac{\pa_1^2\eta}{\abs{\pa_1\eta}^3}\a_{i2}\bp^2\pa_1^2\eta_i\,d\sigma}_{R_{a123}}.
\end{split}
\end{equation*}
With the estimate \eqref{peadv} and \eqref{sigmae}, we obtain
\begin{equation}
	\label{l1}
	\int_0^t R_{a123}\,d{\mathsf t} \ls \int_0^t \abs{\bp^2\pa_1v}_{-\hal}\abs{\bp^2\pa_1^2\eta_i\a_{i2}}_\hal\abs{\dfrac{\pa_1^2\eta}{\abs{\pa_1\eta}^3}}_1\ls T^{\frac{1}{4}}P\left(\sup_{t\in[0,T]}\mathfrak E^\epsilon(t)\right). 
\end{equation}
For $R_{a122}$, we have
\begin{equation}
        \label{est111}
	\begin{split}
	\int_0^t R_{a122}\,d{\mathsf t}&\ls \int_0^t\abs{\pa_1\bp^3\eta}_{-\hal}\abs{\pa_1^2\bp\eta}_{\hal}\abs{\bp^2\pa_1\eta_i\a_{i2}}_{1}+\int_0^t\abs{\pa_1\bp^3\eta}_{-\hal}\abs{\pa_1^2\eta\pa_1\bp\eta}_1\abs{\bp^2\pa_1\eta}_{\hal}\\&\ls T^{\frac{1}{4}}P\left(\sup_{t\in[0,T]}\mathfrak E^\epsilon(t)\right).
\end{split}
\end{equation}
For $R_{a121}$, by integration by parts, the fundamental theorem of calculus, the estimate \eqref{est111}, the interpolation inequality, the trace estimate \eqref{tre} and Young's inequality, we have
\begin{equation}
	\label{l3}
\begin{split}
        \int_0^t R_{a121}\, d{\mathsf t} \ls&  \abs{\intb \bp^2\pa_1^2\eta\dfrac{\pa_1^2\eta}{\abs{\pa_1\eta}^3}\a_{i2}\bp^2\pa_1\eta_i(t)} +M_0\\\ls &M_0+\abs{\bp^2\pa_1\eta}_{L^4}\abs{\dfrac{\pa_1^2\eta}{\abs{\pa_1\eta}^3}}_{L^4}\abs{\bp^2\pa_1^2\eta_i\a_{i2}}_0(t)+\abs{\intb \bp^2\pa_1\eta\pa_1\dfrac{\pa_1^2\eta}{\abs{\pa_1\eta}^3}\bp^2\pa_1\eta_i\a_{i2}}(t)
                                          \\&+\abs{\bp^2\pa_1\eta}_{0}\abs{\bp^2\pa_1\eta}_{L^4}\abs{\dfrac{\pa_1^2\eta}{\abs{\pa_1\eta}^3}\pa_1\a_{i2}}_{L^4}(t)\\\ls &M_0+\abs{\bp^2\pa_1\eta}_0^{\frac{1}{4}}\abs{\bp^2\pa_1\eta}_{L^6}^{\frac{3}{4}}P\left(\norm{\eta}_{\X^3}\right)\abs{\bp^2\pa_1^2\eta_i\a_{i2}}_0(t)\\&+\abs{\intb \bp^2\pa_1\eta(0) \pa_1\dfrac{\pa_1^2\eta}{\abs{\pa_1\eta}^3}\bp^2\pa_1\eta_i\a_{i2}(t)}+\abs{\intb \int_0^t\bp^2\pa_1v(\tau)\,d\tau \pa_1\dfrac{\pa_1^2\eta}{\abs{\pa_1\eta}^3}\bp^2\pa_1\eta\a_{i2}(t)}\\&+\left(\abs{\bp^2\pa_1\eta}_0^2\abs{\pa_1^2\eta\bp^2\pa_1\eta}_0(t)+\abs{\bp^2\pa_1\eta}_0\abs{\bp^2\pa_1\eta}_{L^4}\right)P(\norm{\eta}_{\X^3})(t)\\\leq &M_0+\delta\left(\abs{\bp^2\pa_1^2\eta\cdot n}_0^2+\norm{\bp^2\pa_1\eta}_1^2\right)+C_\delta P\left(\norm{\eta}^2_{\X^3}\right).\\\leq &M_0+T^{\frac{1}{4}}P\left(\sup_{t\in[0,T]}\mathfrak E^\epsilon(t)\right)+\delta\sup_{t\in[0,T]}\mathfrak E^\epsilon(t).
\end{split}
\end{equation}
Thus, combining the estimate \eqref{l1} and \eqref{l3}, we arrive at 
\begin{equation}\label{e01}
\int_0^t R_{a12}\,d{\mathsf t} \ls M_0+\delta\sup_{t\in [0,T]}\mathfrak E^\epsilon(t)+T^{\frac{1}{4}}P\left(\sup_{t\in[0,T]}\mathfrak E^\epsilon(t)\right).
\end{equation}
The estimate for $R_{a13}$ can be obtained by a similar approach for $R_{a12}$:
\begin{equation}
	\label{e02}
	\int_0^t R_{a13}\,d{\mathsf t} \ls M_0+\delta\sup_{t\in [0,T]}\mathfrak E^\epsilon(t)+T^{\frac{1}{4}}P\left(\sup_{t\in[0,T]}\mathfrak E^\epsilon(t)\right). 
\end{equation}
Returning to the estimate of $R_{a11}$, by integration by parts with $x_1$, we have
\begin{equation}\label{e03}
\begin{split}
R_{a11}=&\intb \dfrac{1}{\abs{\pa_1\eta}^3}\bp^2\pa_1^2\eta_k\a_{k2}\a_{i2}\bp^2\pa_1^2v_i+\intb \bp^2\pa_1^2\eta_k\pa_1\left(\dfrac{\a_{k2}\a_{i2}}{\abs{\pa_1\eta}^3}\right)\bp^2\pa_1v_i
\\=&\hal \dfrac{d}{dt}\intb \dfrac{1}{\abs{\pa_1\eta}}\abs{\bp^2\pa_1^2\eta\cdot n}^2\,d\sigma\underbrace{+\intb \bp^2\pa_1^2\eta_k\pa_1\left(\dfrac{\a_{k2}\a_{i2}}{\abs{\pa_1\eta}^3}\right)\bp^2\pa_1v_i}_{R_{a111}}\\&-\hal\intb \pa_t\left( \dfrac{1}{\abs{\pa_1\eta}^3}\right)\abs{\bp^2\pa_1^2\eta_k\a_{k2}}^2-\intb\bp^2\pa_1^2\eta_k\dfrac{\a_{k2}}{\abs{\pa_1\eta}^3}\bp^2\pa_1^2\eta_j\pa_t\a_{j2}.
\end{split}
\end{equation}
The first term of the last line of \eqref{e03} is bounded by $P\left(\norm{\bp^2\pa_1\eta}_1,\abs{\bp^2\pa_1^2\eta\cdot n}_0,\norm{\eta}_{\X^3}\right)$. The second term is bounded by
\begin{equation*}
	\abs{\bp^2\pa_1^2\eta_k\a_{k2}}_\hal\norm{\bp^2\pa_1\eta}_{1}P\left(\norm{\bp^2\pa_1\eta}_1,\norm{\eta}_{\X^3}\right)
\end{equation*}
with the help of \eqref{peadv} and the trace theorem. Then, with the estimate \eqref{sigmae}, we have 
\begin{equation*}
	\int_0^t \intb\dfrac{\bp^2\pa_1^2\eta_k\a_{k2}}{\abs{\pa_1\eta}^3}\bp^2\pa_1^2\eta_j\pa_t\a_{j2} \ls \sqrt TP\left(\sup_{t\in[0,T]}\mathfrak E^\epsilon(t)\right).
\end{equation*}
The term $\int_0^t R_{a111}\,d{\mathsf t}$ can be bounded by $M_0+T^{\frac{1}{4}}P\left(\sup_{t\in[0,T]}\mathfrak E^\epsilon(t)\right) $ with a similar approach as that used for $R_{a12}$.

For $R_b$, by using the pressure estimate \eqref{press2}, we have
\begin{equation}\label{e3}
	\int_0^t R_b\,d{\mathsf t}\ls M_0+\delta \sup_{t\in[0,T]}\mathfrak E^\epsilon(t)+T^{1/4}P\left(\sup_{t\in[0,T]}\mathfrak E^\epsilon(t)\right).
\end{equation}

For $R_c$, we have
\begin{equation}
	\label{e4}
\begin{split}
R_c=&\intd \epsilon J\abs{S_\eta(\bp^2\pa_1v)}^2+P\left(\norm{\sqrt\epsilon\nabla^2\eta}_{\X^2},\norm{\eta}_{\X^3},\norm{\bp^2\pa_1\eta}_1\right)(1+\norm{\bp^3\eta}_1)\norm{\sqrt\epsilon\nabla\bp^2\pa_1v}_0\\&+P(\norm{\eta}_{\X^3})\norm{\sqrt\epsilon\bp^2\nabla\eta}_{1}\norm{\bp^2\nabla\eta}_{L^4}\norm{\sqrt\epsilon\nabla\bp^2\pa_1v}_0,
\end{split}
\end{equation}
and for $R_d$, we have
\begin{equation}
	R_d \ls \norm{\sqrt\epsilon\nabla\Psi}_{\X^2}\norm{\sqrt\epsilon\bp^2\pa_1\nabla v}_0.
\end{equation}

Last, noticing that $\phi^\epsilon_i$ depends only on $x$, we have
\begin{equation*}
	\intd \bp^2\pa_1\phi_i^\epsilon \bp^2\pa_1v_i=\dfrac{d}{dt}\intd \bp^2\phi_i^\epsilon \bp^2\pa_1^2\eta_i
\end{equation*}
by using integration by parts with respect to $x_1$. Then we arrive at
\begin{equation}
	\label{e5}
	\int_0^T \intd \bp^2\pa_1\phi_i^\epsilon \bp^2\pa_1v_i \leq M_0+\delta \norm{\bp^2\pa_1^2\eta}_0^2
\end{equation}
Combining the estimates \eqref{ggg0}--\eqref{grgr}, \eqref{44}--\eqref{444}, and \eqref{e01}--\eqref{e5}, and using Korn's inequality and the interpolation inequality, we have
\begin{equation}\label{ohoh1}
\begin{split}
&\norm{\bp^2\pa_1v(t)}_0^2+\norm{\bar\partial^2\pa_1 \nabla \eta(t)}_0^2+\abs{\bp^2\pa_1^2\eta\cdot n(t)}^2_{0}+\epsilon\int_0^t\norm{\bp^2\pa_1\nabla v}_0^2+\mathfrak G(t)\\&\leq M_0+\delta\sup_{t\in[0,T]} \mathfrak{E}^{\epsilon}(t)+ T^{\frac{1}{4}}P\left(\sup_{t\in[0,T]} \mathfrak{E}^{\epsilon}(t)\right),
\end{split}
\end{equation}
where
\begin{equation*}
\begin{split}
	\mathfrak G:=&-\hal\int_{\Gamma}\mathfrak B\bp^2\pa_1\a_{i\ell}N_\ell\bp^2\pa_1\eta_i\,d\sigma.
\end{split}
\end{equation*}
To estimate $\mathfrak G$, recalling the identity \eqref{decomp} and using equalities \eqref{in3} and \eqref{in4}, we have 
\begin{equation}
\label{sufc}
\begin{split}
	\mathfrak G&=-\hal \intb \mathfrak B\bp^2\pa_1\eta_i(n_in_j+\tau_i\tau_j)\bp^2\pa_1\a_{j\ell}N_\ell\,d\sigma\\&= -\hal\intb \mathfrak B\bp^2\pa_1\eta\cdot n \bp^2\pa_1^2\eta\cdot\tau\,d\sigma+\hal\intb \mathfrak B\bp^2\pa_1\eta\cdot\tau\bp^2\pa_1^2\eta\cdot n\,d\sigma\\&=\hal\intb \mathfrak B\bp^2\pa_1^2\eta\cdot n\bp^2\pa_1\eta\cdot\tau\,d\sigma+\hal\intb \mathfrak B\bp^2\pa_1\eta\cdot\tau\bp^2\pa_1^2\eta\cdot n\,d\sigma\\&\quad-\hal \intb \pa_1\left(\mathfrak B\tau_in_j\right)\bp^2\pa_1\eta_i\bp^2\pa_1\eta_j\,d\sigma. 
\end{split}
\end{equation}
Then, by the trace estimate \eqref{tre} and the fundamental theorem of calculus, we arrive at
\begin{equation*}
\begin{split}
\abs{\mathfrak G}&\ls \abs{\bp^2\pa_1^2\eta\cdot n}_0\abs{\mathfrak B}_1\abs{\bp^2\pa_1\eta}_0+\left(\abs{\pa_1(\mathfrak B\tau_in_j)}_{L^4}\right)\abs{\bp^2\pa_1\eta}_0\abs{\bp^2\pa_1\eta}_{L^4}\\&\ls \delta\left(\abs{\bp^2\pa_1^2\eta\cdot n}_0^2+\norm{\bp^2\pa_1\nabla\eta}_0^2\right)+C_{\delta}P\left(\norm{\eta}_{\X^3}^2,\norm{\epsilon\pa_1^2\nabla v}_0^2,\norm{\epsilon\Psi}_{2}^2\right)\\&\leq M_0+ T^{\frac{1}{4}}P\left(\sup_{t\in[0,T]}\mathfrak E^\epsilon(t)\right)+\delta\sup_{t\in[0,T]}\mathfrak E^\epsilon(t)
\end{split}
\end{equation*}
for $\epsilon$ sufficiently small;
thus, the proposition is proved.
\end{proof}
\subsubsection{$\pa_t^3$ energy estimates}
Due to the lack of the estimates $\norm{\pa_t^3 q}_0$ and $\abs{\pa_t^3\pa_1\eta\cdot n}_\hal$, some approaches to estimating $\bp^2\pa_1$ derivatives are not applicable for the $\pa_t^3$ problem, such as the estimate for \eqref{e03}. To overcome this difficulty, we adopt the idea in \cite{GuW_2016} to use Alinhac's good unknowns
\begin{equation}
\label{gun}
\mathcal{V}=\pa_t^3 v- \pa_t^3\eta\cdot\nabla_\eta v,\quad \mathcal{Q}=\pa_t^3 q- \pa_t^3\eta\cdot\nabla_\eta q.
\end{equation}
to derive the $\pa_t^3$ energy estimates.

The equation for $\mathcal V$ and $\mathcal Q$ is derived by applying $\pa_t^3$ to equations \eqref{approximate}$_2$ and \eqref{approximate}$_3$:
\begin{multline}\label{eqValpha}
	\pa_t\mathcal{V}_i + J\pa_{\eta_i} \mathcal{Q}-\Delta\pa_t^3\eta_i-2\epsilon\pa_k(\pa_t^3 (S_\eta(v)_{ij}\a_{jk}))=F_i:= - \dt\left(\pa_t^3\eta\cdot\nabla_\eta v_i\right) - \mathcal{C}_i(q)\,\,  \text{ in }\Omega,
\end{multline}
and
\begin{equation} \label{divValpha}
  J\nabla_\eta\cdot \mathcal{V}=- \mathcal{C}_i(v_i) \text{ in }\Omega,
\end{equation}
where the commutator $\mathcal{C}_i(f)$ is given by
\begin{equation}
	\begin{split}
		\mathcal{C}_i(f)=&\left[\pa_t^3, {\a_{ij}} ,\pa_j f\right]+J\pa_t^3\eta\cdot\pa_{\eta} \na f
		-J^{-1}\left[\pa_t^2, \a_{i\ell}\a_{mj}\right]\pa_t\pa_\ell \eta_m\pa_j f\\=&\left[\pa_t^3, \pa_1\eta ,\pa_2 f\right]+\left[\pa_t^3,\pa_2\eta,\pa_1 f\right]+J\pa_t^3\eta\cdot\pa_{\eta} \na f
		-J^{-1}\left[\pa_t^2, \a_{i\ell}\a_{mj}\right]\pa_t\pa_\ell \eta_m\pa_j f.
\end{split}
\end{equation}
The details of the above calculation can be found in \cite[Section 4.2.4]{GuW_2016}.

Before we begin the energy estimate for equation \eqref{eqValpha}, we give the estimates of the commutator $\mathcal C_i(f)$.
\begin{lemma}\label{lemma54}
The following estimate holds:
\begin{equation}\label{comest}
	\begin{split}
		\norm{\mathcal C_i (f)}_0&\leq P(\norm{\eta}_{\X^3})\left(\norm{\pa_t^3\eta}_1\norm{\nabla f}_1^{\frac{2}{3}}\norm{\nabla^3f}_0^{\frac{1}{3}}+\norm{\pa_t^2\nabla f}_0+\norm{\nabla f}_1\right)\\&\quad+\norm{\eta}_{\X^3}^{\frac{2}{3}}\norm{\nabla\eta}_{\X^3}^{\frac{1}{3}}\left(\norm{\pa_t^2\nabla f}_{0}+\norm{\nabla f}_1\right).
\end{split}
\end{equation}
\end{lemma}
\begin{proof} By straightforward calculation and using Sobolev's embedding inequality, Galedo-Nirenberg's intepolation inequality, we have
	\begin{equation*}
		\norm{J\pa_t^3\eta\cdot\pa_{\eta}\nabla f}_0\ls \norm{\pa_t^3\eta}_{L^6}\norm{\nabla\eta}_{L^\infty}\norm{\nabla^2 f}_{L^3}\ls P\left(\norm{\eta}_{\X^3}\right)\norm{\pa_t^3\eta}_1\norm{\nabla^2 f}_0^{\frac{2}{3}}\norm{\nabla^3 f}_0^{\frac{1}{3}},
	\end{equation*}
	\begin{equation*}
		\begin{split}
			\norm{\left[\pa_t^3,\pa_1\eta,\pa_2 f\right]}_0&\ls \norm{\pa_t^2\pa_1\eta}_{L^6}\norm{\pa_t\pa_2 f}_{L^3}+\norm{\pa_t\pa_1\eta}_{L^\infty}\norm{\pa_t^2\nabla f}_0\\&\ls \norm{\pa_t^2\pa_1\eta}_1\norm{\pa_t\pa_2 f}_0^{\frac{2}{3}}\norm{\pa_t\nabla^2 f}_0^{\frac{1}{3}}+\left(\norm{\eta}_{\X^3}+\norm{\eta}_{\X^3}^{\frac{2}{3}}\norm{\nabla\eta}_{\X^3}^{\frac{1}{3}}\right)\norm{\pa_t^2\nabla f}_0,
		\end{split}
	\end{equation*}
\begin{equation*}
		\begin{split}
			\norm{\left[\pa_t^3,\pa_2\eta,\pa_1 f\right]}_0&\ls \norm{\pa_t^2\pa_2\eta}_{L^3}\norm{\pa_t\pa_1 f}_{L^6}+\norm{\pa_t\pa_2\eta}_{L^\infty}\norm{\pa_t^2\pa_1 f}_0\\&\ls \norm{\pa_t^2\nabla\eta}_0^{\frac{2}{3}}\norm{\pa_t^2\nabla^2 \eta}_0^{\frac{1}{3}}\norm{\pa_t\pa_1 f}_1+\left(\norm{\eta}_{\X^3}+\norm{\eta}_{\X^3}^{\frac{2}{3}}\norm{\nabla\eta}_{\X^3}^{\frac{1}{3}}\right)\norm{\pa_t^2\nabla f}_0,
		\end{split}
\end{equation*}
\begin{equation*}
		\begin{split}
			&\norm{J^{-1}\left[\pa_t^2,\nabla\eta\nabla\eta\right]\pa_t\nabla\eta\nabla f}_0\\&\ls \norm{\pa_t^2(\nabla\eta\nabla\eta)}_{L^3}\norm{\nabla v\nabla f}_{L^6}+\norm{\pa_t\nabla\eta\nabla\eta}_{L^{12}}\norm{\pa_t(\nabla v\nabla f)}_{L^{\frac{12}{5}}}\\&\ls \left(\norm{\eta}_{\X^3}^{\frac{2}{3}}\norm{\nabla f}_1\norm{\nabla\eta}_{\X^3}^{\frac{1}{3}}+\norm{\pa_t\nabla^2 f}_0^{\frac{1}{3}}\norm{\pa_t\nabla f}_0^{\frac{2}{3}}+\norm{\nabla f}_1\right) P\left(\norm{\eta}_{\X^3}\right).
		\end{split}
	\end{equation*}
Thus, we prove the lemma.
\end{proof}

Now, we prove the following $\pa_t^3$ energy estimates:
\begin{proposition}\label{tane2}
For $t\in [0,T]$ with $T\le T_\epsilon$, and for $\epsilon$ sufficiently small, it holds that
\begin{equation}
\label{tee3et}
\begin{split}
\int_0^t\norm{\pa_t^3 v(\mathsf t)}_0^4+\norm{\pa_t^3\nabla\eta(\mathsf t)}_0^4+\abs{\pa_1\pa_t^3\eta\cdot n(\mathsf t)}_0^4\,d{\mathsf t}+\int_0^t\left(\int_0^{\mathsf t}\norm{\sqrt\epsilon\nabla\pa_t^3 v}_0^2\right)^2\,d{\mathsf t}\\\leq M_0+\delta\sup_{t\in[0,T]}\mathfrak{E}^{\epsilon}+TP\left(\sup_{t\in[0,T]}\mathfrak{E}^{\epsilon}(t)\right).
\end{split}
\end{equation}
\end{proposition}
\begin{proof}
%
%
Taking the $L^2(\Omega)$ inner product of equation \eqref{eqValpha} with $\mathcal V$ yields
\begin{equation}\label{t1t}
	\begin{split}
		\dfrac{1}{2}\dfrac{d}{dt}\int_{\Omega}\abs{\mathcal V}^2+\int_{\Omega} J\nabla_\eta \mathcal Q \cdot \mathcal V-\int_{\Omega} \Delta\pa_t^3\eta_i \mathcal{V}_i-2\epsilon\intd \pa_k(\pa_t^3(S_\eta(v)_{ij}\a_{jk}))\mathcal V_i=\int_{\Omega} F\cdot \mathcal V.
 \end{split}
\end{equation}
By the commutator estimates \eqref{comest}, the pressure estimate \eqref{press2} and the definition of $\mathcal{V}$, we have
\begin{equation}
\label{ggg0t}
\begin{split}
&\int_0^{\mathsf t}\int_{\Omega} F\cdot \mathcal V\\&\le\int_0^{\mathsf t}\left(\norm{\mathcal C_i (q)}_0+\norm{\dt\left(\pa_t^3\eta\cdot\nabla_\eta v\right)}_0\right)\norm{\mathcal V}_0
\\&\ls \int_0^{\mathsf t}\left(\norm{\mathcal C_i(q)}_{0}+\norm{\pa_t^3 v}_{0}\norm{\nabla_\eta v}_{L^\infty}+\norm{\pa_t^3\eta}_{L^4}\norm{\pa_t(\nabla_\eta v)}_{L^4}\right)\norm{\mathcal V}_0
\\   &\ls \norm{\nabla q}_{L_{\mathsf t}^2(\X^2)}\left(\norm{\nabla\eta}_{L^2_{\mathsf t}(\X^3)}^\hal+\norm{\pa_t^3\nabla\eta}_{L^4_{\mathsf t}(L^2)}\right)\norm{\V}_{L_{\mathsf t}^4(L^2)}P\left(\sup_{\mathsf t\in [0,T]}\norm{\eta}_{\X^3}\right) 
\\&\quad +\norm{\nabla_\eta v}_{L^2_{\mathsf t}(H^2)}\norm{\V}_{L_{\mathsf t}^4(L^2)}\norm{\pa_t^3v}_{L^4_{\mathsf t}(L^2)}+\norm{\pa_t(\nabla_\eta v)}_{L^2_{\mathsf t}(H^1)}\norm{\V}_{L^4_{\mathsf t}(L^2)}\norm{\pa_t^3\eta}_{L^4_{\mathsf t}(H^1)}\\&\ls P\left(\sup_{\mathsf t\in[0,T]}\mathfrak E^\epsilon(\mathsf t)\right).
\end{split}
\end{equation}
We now estimate the three terms on the left-hand side of \eqref{t1t}. By the fact that $v=\partial_t\eta$, we have that for the third term,
\begin{equation}
\label{gg1t}
\begin{split}
&\int_{\Omega} J\na_\eta \Q \cdot \V-\int_{\Omega}  \pa_t^3\Delta\eta \cdot \V-\intd 2\epsilon \pa_k(\pa_t^3(S_\eta(v)_{ij}\a_{jk})\mathcal V_i\\= &\int_{\Omega} \Q \mathcal{C}_i(v_i)+\int_{\Gamma} \Q \a_{i\ell}N_\ell\V_i\,d\sigma+\int_{\Omega}  \pa_t^3\partial_j\eta_i \partial_j\V_i -\int_{\Gamma}\partial_\ell\pa_t^3\eta_iN_\ell\V_i\,d\sigma\\&\quad+2\epsilon\intd \pa_t^3(S_\eta(v)_{ij}\a_{jk})\pa_{k}\mathcal V_i-\intb 2\epsilon\pa_t^3(S_\eta(v)_{ij}\a_{j\ell})N_\ell\V_i\,d\sigma\\=&\dfrac{1}{2}\dfrac{d}{dt}\int_{\Omega}\abs{\nabla\pa_t^3\eta}^2-\underbrace{\int_{\Omega}\pa_t^3\pa_j \eta_i \pa_j(\pa_t^3\eta\cdot \na_\eta v_i) }_{R_a}+\underbrace{\int_{\Gamma}\left(\Q \a_{i\ell}-\pa_t^3F_{i\ell}-2\epsilon\pa_t^3(S_\eta(v)\a)_{i\ell}\right)N_\ell\V_i\,d\sigma}_{R_b}\\&\quad+\underbrace{\intd \Q \mathcal{C}_i(v_i)}_{R_c}+\underbrace{\intd 2\epsilon\pa_t^3(S_\eta(v)_{ij}\a_{jk})\pa_k \V_i}_{R_d}.
\end{split}
\end{equation}
For $R_a$, we have
\begin{equation*}
\int_0^{\mathsf t}\abs{R_a}\ls \norm{\abs{\pa_t^3\nabla\eta(t)}^2}_{L^2_{\mathsf t}(L^2)}\norm{\nabla_\eta v}_{L^2_{\mathsf t}(H^2)}^2\ls P\left(\sup_{\mathsf t\in[0,T]}\mathfrak E^\epsilon(\mathsf t)\right).
\end{equation*}
For $R_b$, we have
\begin{equation*}
R_b=\underbrace{\int_{\Gamma}\left(\pa_t^3(q+1)\a_{i\ell}-\pa_t^3F_{i\ell}-2\epsilon\pa_t^3(S_\eta(v)\a)_{i\ell}\right)N_\ell\V_i\,d\sigma}_{R_{b1}}+\underbrace{\intb \pa_t^3\eta \cdot \nabla_\eta q \a_{i\ell}N_\ell\V_i\,d\sigma}_{R_{b2}}\
\end{equation*}
First, with the normal trace estimate \eqref{gga} and $\abs{f}_{L^4}\ls\abs{f}_{\hal}\ls\norm{f}_1$, we have
\begin{equation*}
\begin{split}
R_{b2}&\ls\abs{\pa_t^3\eta \nabla_\eta q}_{\hal}\abs{\a_{i\ell}N_\ell
\pa_t^3v_i}_{-\hal}+\abs{\pa_t^3\eta}_{L^4}^2\abs{\nabla_\eta q \nabla_\eta v}_{0} \\&\ls \norm{\pa_t^3\eta}_1\left(\norm{\nabla q}_{L^\infty}+\norm{\nabla^2 q}_{L^3}\right)P(\norm{\eta}_{\X^3})\left(\norm{\pa_t^3v}_0+\norm{J\Div_{\eta}\pa_t^3v}_0\right)+\norm{\pa_t^3\eta}_1^2\abs{\nabla_\eta q \nabla_\eta v}_{0}.
\end{split}
\end{equation*}
Here, we use $\Div_\eta v=0$ to bound
\begin{equation*}
	\norm{J\Div_\eta \pa_t^3v}_0\ls \left(\norm{\pa_t^3\nabla\eta}_{0}\norm{\nabla^3v}_0^\hal\norm{\eta}_{\X^3}^\hal+\norm{\pa_t^2\pa_1\eta}_1\norm{\nabla\eta}_{\X^3}^\hal\norm{\pa_t^2\nabla\eta}_0^\hal\right)P(\norm{\eta}_{\X^3}).
\end{equation*}
Then, with the estimate of pressure \eqref{press2}, we obtain
\begin{equation*}
\begin{split}
\int_0^{\mathsf t}\abs{R_{b2}}\ls P\left(\sup_{\mathsf t\in[0,T]}\mathfrak E^\epsilon(\mathsf t)\right).
\end{split}
\end{equation*}
Second, using the boundary condition \eqref{tb}, we have
\begin{equation*}
\begin{split}
	R_{b1}=&\underbrace{-\int_{\Gamma}\pa_t^3\left(\dfrac{\partial_1^2\eta_k\a_{k\ell}N_\ell}{\abs{\pa_1\eta}^3}\right)\a N\cdot\V\,d\sigma}_{R_{b11}} -\underbrace{\int_{\Gamma}\mathfrak B\pa_t^3\a_{i\ell}N_\ell\V_i\,d\sigma}_{R_{b12}}\underbrace{-\intb \left[\pa_t^3, \mathfrak B,\a_{i\ell}N_\ell\right]\V_i\,d\sigma}_{R_{b13}}.
\end{split}
\end{equation*}
Regarding $R_{b13}$, it cannot be estimated as \eqref{grgr} since $\abs{\pa_t^3v}_{-\hal}$ is out of control. To handle this, we use the identity \eqref{decomp} to decompose $R_{b13}$ into two parts:
\begin{equation*}
\begin{split}
R_{b13}=&\intb \left[\pa_t^3, \mathfrak B, \a_{i\ell}N_\ell\right](n_in_j+\tau_i\tau_j) \V_j\,d\sigma\\=&
\intb \left[\pa_t^3, \mathfrak B, \a N\right]\cdot n\V\cdot n\,d\sigma+\intb \left[\pa_t^3, \mathfrak B, \a N\right]\cdot \tau\V\cdot \tau\,d\sigma
\end{split}
\end{equation*}
The first term can be estimated as 
\begin{equation*}
\ls \abs{\left[\pa_t^3,\mathfrak B, \a N\right]\pa_1\eta}_{\hal}\abs{\pa_t^3v_i\a_{i\ell}N_\ell}_{-\hal}-\intb \left[\pa_t^3,\mathfrak B, \a N\right]\cdot n \left(\pa_t^3\eta\cdot\nabla_\eta v\cdot n\right)\,d\sigma
\end{equation*}
Using the equality \eqref{in3}, the second term can be written as
\begin{equation*}
\begin{split}
&-\intb \left[\pa_t^3,\mathfrak B, \pa_1\eta\right]\cdot n \pa_t^3v\cdot \tau\,d\sigma+\intb \left[\pa_t^3,\mathfrak B, \pa_1\eta\right]\cdot n \left(\pa_t^3\eta\cdot\nabla_\eta v\cdot \tau\right)\,d\sigma\\=&-\dfrac{d}{dt}\intb \left[\pa_t^3,\mathfrak B, \pa_1\eta\right]\cdot n \pa_t^3\eta\cdot \tau\,d\sigma+\intb \pa_t\left(\left[\pa_t^3,\mathfrak B, \pa_1\eta\right]\cdot n\tau_j\right)\pa_t^3\eta_j\,d\sigma\\&+\intb \left[\pa_t^3,\mathfrak B, \pa_1\eta\right]\cdot n \left(\pa_t^3\eta\cdot\nabla_\eta v\cdot \tau\right)\,d\sigma.
\end{split}
\end{equation*}
Thus, with the normal trace estimate \eqref{gga}, we arrive at
\begin{equation*}
\begin{split}
\int_0^{\mathsf t}R_{b13}\ls &M_0+\intb\left[\pa_t^3,\mathfrak B, \pa_1\eta\right]\cdot n\pa_t^3\eta\cdot\tau\,d\sigma(\mathsf t)+\int_0^{\mathsf t}\abs{\pa_t\left(\left[\pa_t^3,\mathfrak B, \pa_1\eta\right]\cdot n\tau_j\right)}_{-\hal}\abs{\pa_t^3\eta}_{\hal}\\&+\int_0^{\mathsf t}\abs{\left[\pa_t^3,\mathfrak B, \a N\right]\pa_1\eta}_{\hal}\abs{\pa_t^3v_i\a_{i\ell}N_\ell}_{-\hal}+\int_0^{\mathsf t}\abs{\left[\pa_t^3,\mathfrak B, \a N\right]\nabla\eta\nabla v}_{-\hal}\abs{\pa_t^3\eta}_{\hal}\\\ls &M_0+\delta\norm{\pa_t^3\eta}_1^2(\mathsf t)+\epsilon\int_0^{\mathsf t}\norm{\sqrt\epsilon\pa_t^3\nabla v}_0^2+P\left(\sup_{\mathsf t\in[0,T]}\mathfrak E^\epsilon(\mathsf t)\right).
\end{split}
\end{equation*}
For the term $R_{b12}$, the derivation of the symmetric structure is different from that of \eqref{pol}--\eqref{pol2}. First, we have
\begin{equation*}
\begin{split}
	\intb\mathfrak B\pa_t^3\a_{i\ell}N_\ell\pa_t^3v_i\,d\sigma=&\dfrac{d}{dt}\int_{\Gamma}\mathfrak B\pa_t^3\a_{i\ell}N_\ell\pa_t^3\eta_i\,d\sigma-\intb \pa_t\mathfrak B\pa_t^3\a_{i\ell}N_\ell\pa_t^3\eta_i\,d\sigma-\intb \mathfrak B \pa_t^4\a_{i\ell}N_\ell\pa_t^3\eta_i\,d\sigma
\end{split}
\end{equation*}
Then, with the equalities \eqref{in3} and \eqref{in4}, we arrive at
\begin{equation}
        \label{ttttt}
\begin{split}
&-\intb \mathfrak B \pa_t^4\a_{i\ell}N_\ell\pa_t^3\eta_i\,d\sigma\\=&-\intb \mathfrak B \pa_t^4\a_{i\ell}N_\ell(n_in_j+\tau_i\tau_j)\pa_t^3\eta_j\,d\sigma\\=&\intb \mathfrak B \pa_t^4\pa_1\eta\cdot n \pa_t^3\eta\cdot\tau\,d\sigma-\intb \mathfrak B \pa_t^4\pa_1\eta\cdot\tau \pa_t^3\eta\cdot n\,d\sigma\\=&-\intb \mathfrak B \pa_t^3v\cdot n \pa_1\pa_t^3\eta\cdot\tau\,d\sigma+\intb \mathfrak B \pa_t^3v\cdot\tau \pa_1\pa_t^3\eta\cdot n\,d\sigma-\intb \pa_1(\mathfrak B n_i\tau_j-\mathfrak B n_j\tau_i)\pa_t^3v_i\pa_t^3\eta_j\,d\sigma\\=&-\intb \mathfrak B \pa_t^3v_i\pa_t^3\a_{i\ell}N_\ell\, d\sigma-\intb \pa_1\mathfrak B (n_i\tau_j-n_j\tau_i)\pa_t^3v_i\pa_t^3\eta_j\,d\sigma
\end{split}
\end{equation}
where the identity \eqref{decomp2} is applied to obtain $\pa_1(n_i\tau_j-n_j\tau_i)=0$.

Thus, we have 
\begin{equation}
        \label{tttttt}
\begin{split}
R_{b12}=&\hal \dfrac{d}{dt}\intb \mathfrak B\pa_t^3\a_{i\ell}N_\ell\pa_t^3\eta_i\,d\sigma\underbrace{-\hal \intb \pa_t\mathfrak B\pa_t^3\a_{i\ell}N_\ell\pa_t^3\eta_i\,d\sigma}_{R_{b121}}\\&\underbrace{-\hal\intb \pa_1\mathfrak B (n_i\tau_j-n_j\tau_i)\pa_t^3v_i\pa_t^3\eta_j\,d\sigma}_{R_{b122}}\underbrace{-\intb \mathfrak B\pa_t^3\a_{i\ell}N_\ell\pa_t^3\eta\cdot\nabla_\eta v_i\,d\sigma}_{R_{b123}}
\end{split}
\end{equation}
and $\int_0^{\mathsf t} R_{b121}+R_{b123}$ is bounded by 
\begin{equation*}
\begin{split}
\ls\int_0^{\mathsf t}\abs{\pa_t^3\pa_1\eta}_{-\hal}\abs{\pa_t^3\eta}_{\hal}\norm{\nabla\eta}_{\X^3}P\left(\norm{\eta}_{\X^3},\norm{\epsilon \pa_1^2\pa_tv}_1\right) \leq P\left(\sup_{\mathsf t\in[0,T]}\mathfrak E^\epsilon(\mathsf t)\right).
\end{split}
\end{equation*}
For $R_{b122}$, we see that
\begin{equation*}
\begin{split}
R_{b122}=&-\hal\intb\pa_1\mathfrak B \pa_t^3v\cdot n\pa_t^3\eta\cdot \tau\,d\sigma +\hal\dfrac{d}{dt}\intb \pa_1\mathfrak B\pa_t^3\eta\cdot n\pa_t^3\eta\cdot\tau\,d\sigma\\&-\hal\intb \pa_1\mathfrak B\pa_t^3v\cdot n\pa_t^3\eta\cdot\tau\,d\sigma-\hal\intb \pa_t\pa_1\mathfrak B \pa_t^3\eta\cdot n\pa_t^3\eta\cdot\tau\,d\sigma.
\end{split}
\end{equation*}
Then, we obtain
\begin{equation}
        \label{tttttttt}
\begin{split}
\int_0^{\mathsf t}R_{b122}\ls& M_0+\abs{\pa_1\mathfrak B}\abs{\pa_t^3\eta}_0^2(\mathsf t)+\int_0^{\mathsf t}\abs{\pa_t^3v_i\a_{i\ell}N_\ell}_{-\hal}\abs{\pa_t^3\eta}_{\hal}\abs{\dfrac{\pa_1\mathfrak B\tau}{\abs{\pa_1\eta}}}_1\\&+\int_0^{\mathsf t}\abs{\pa_t\pa_1\mathfrak B}_{0}\abs{\pa_t^3\eta\cdot n}_1\abs{\pa_t^3\eta}_0 \leq \delta\norm{\pa_t^3\eta}_0^2+P\left(\sup_{\mathsf t\in[0,T]}\mathfrak E^\epsilon(\mathsf t)\right).
\end{split}
\end{equation}

Next, for $R_{b11}$, we have
\begin{equation*}
\begin{split}
	R_{b11}=&\underbrace{-\intb \dfrac{1}{\abs{\pa_1\eta}^3}\pa_t^3\pa_1^2\eta_k\a_{k2}\a_{i2}\V_i}_{R_{b111}}\underbrace{-\intb \pa_t^3\left(\dfrac{\a_{k2}}{\abs{\pa_1\eta}^3}\right)\pa_1^2\eta_k\a_{i2}\V_i}_{R_{b112}}\\&\underbrace{-\intb 3 \pa_t\left(\dfrac{\a_{k2}}{\abs{\pa_1\eta}^3}\right)\pa_t^2\pa_1^2\eta_k\a_{i2}\V_i}_{R_{b113}}\underbrace{+\intb 3\pa_t^{2}\left(\dfrac{\a_{k2}}{\abs{\pa_1\eta}^3}\right)\pa_t\pa_1^2\eta_k\a_{i2}\V_i}_{R_{b114}}
\end{split}
\end{equation*}
For $R_{b114}$, using integragtion by parts with $t$, we have 
\begin{equation*}
\begin{split}
	\int_0^{\mathsf t} R_{b114}=&\intb \pa_t^2\pa_1\eta\pa_t\pa_1^2\eta\pa_t^3\eta_i\a_{i2}(\mathsf t)+M_0+\int_0^{\mathsf t}\intb \pa_t\left(\pa_t^2\pa_1\eta\pa_1^2\pa_t\eta\a_{i2}\right)\pa_t^3\eta_i\\&+\int_0^{\mathsf t} \intb\pa_1v\pa_1v\pa_1^2v\V_i\a_{i2}+\int_0^{\mathsf t} \intb\pa_t^2\pa_1\eta\pa_1^2v\pa_1\eta\pa_t^3\eta\cdot\nabla_\eta v\\ 
	\ls&\intb \pa_t^2\pa_1\eta\pa_t\pa_1^2\eta\pa_t^3\eta_i\a_{i2}(\mathsf t)+\int_0^{\mathsf t} \abs{\pa_t^3\pa_1\eta}_{-\hal}\abs{\pa_1^2\pa_t\eta}_{\hal}\abs{\pa_t^3\eta_i\a_{i2}}_1\\&+\int_0^{\mathsf t}\abs{\V_i\a_{i2}}_{-\hal}\abs{\pa_1^2v}_{\hal}\abs{(\pa_1v)^2}_1\\\ls& \delta\abs{\pa_t^3\pa_1\eta_i\a_{i2}}_0^2(\mathsf t)+P\left(\sup_{\mathsf t\in [0,T]}\mathfrak E^\epsilon(\mathsf t)\right).
\end{split}
\end{equation*}
For $R_{b111}$, using intergation by parts with $x_1$, we have
\begin{equation}\label{preadt}
\begin{split}
	R_{b111}=&\intb \dfrac{1}{\abs{\pa_1\eta}^3}\pa_t^3\pa_1\eta_k\a_{k2}\a_{i2}\left(\pa_1\pa_t^3v_i-\pa_1\pa_t^3\eta\cdot\nabla_\eta v_i\right)\\&+\intb \pa_t^3\pa_1\eta_k\pa_1\left(\dfrac{\a_{k2}\a_{i2}}{\abs{\pa_1\eta}^3}\right)\V_i-\intb\pa_t^3\pa_1\eta_k\dfrac{\a_{k2}\a_{i2}}{\abs{\pa_1\eta}^3}\pa_t^3\eta\cdot\pa_1(\nabla_\eta v_i)
\\=&\hal \dfrac{d}{dt}\intb \dfrac{1}{\abs{\pa_1\eta}^3}\abs{\pa_t^3\pa_1\eta_k\a_{k2}}^2\underbrace{+\intb \pa_t^3\pa_1\eta_k\pa_1\left(\dfrac{\a_{k2}\a_{i2}}{\abs{\pa_1\eta}^3}\right)\V_i}_{R_{b1111}}\\&-\hal\intb \pa_t\left( \dfrac{1}{\abs{\pa_1\eta}^3}\right)\abs{\pa_t^3\pa_1\eta_k\a_{k2}}^2-\intb\pa_t^3\pa_1\eta_k\dfrac{\a_{k2}\a_{i2}}{\abs{\pa_1\eta}^3}\pa_t^3\eta\cdot\pa_1(\nabla_\eta v_i)
\end{split}
\end{equation}
The last line of \eqref{preadt} is bounded by $\left(\norm{\pa_t^3\eta}_1^2+\abs{\pa_t^3\pa_1\eta_i\a_{i2}}_0^2\right)P\left(\norm{\eta}_{\X^3},\norm{\pa_1^2v}_1\right)$.

The term $R_{b1111}$ needs to be combined with $R_{b112}$. By straightforward calculation, we have
\begin{equation}
\begin{split}
&R_{b112}+R_{b1111}\\=&-\intb \left(\pa_t^3\left(\dfrac{\a_{k2}}{\abs{\pa_1\eta}^3}\right)\pa_1^2\eta_k\a_{i2}-\pa_t^3\pa_1\eta_k\pa_1\left(\dfrac{\a_{k2}\a_{i2}}{\abs{\pa_1\eta}^3}\right)\right)\V_i\\=&\underbrace{-
\intb \left(\dfrac{\pa_t^3\a_{k2}\pa_1^2\eta_k}{\abs{\pa_1\eta}^3}-\dfrac{\pa_t^3\pa_1\eta_k\pa_1\a_{k2}}{\abs{\pa_1\eta}^3}\right)\a_{i2}\V_i+\intb \dfrac{\pa_t^3\pa_1\eta_k\a_{k2}}{\abs{\pa_1\eta}^3}\pa_1\a_{i2}\V_i}_{I_a}\\&\underbrace{+\intb \left(\dfrac{3\pa_t^3\a_{m2}\a_{m2}\a_{k2}\pa_1^2\eta_k}{\abs{\pa_1\eta}^5}-\dfrac{3\pa_t^3\pa_1\eta_k\pa_1\a_{j2}\a_{j2}\a_{k2}}{\abs{\pa_1\eta}^5}\right)\a_{i2}\V_i}_{I_b}
\\&-\intb \left(\left[\pa_t^3,\a_{k2},\dfrac{1}{\abs{\pa_1\eta}^3}\right]-3\left[\pa_t^2, \pa_t\a_{m2},\dfrac{\a_{m2}}{\abs{\pa_1\eta}^5}\right]\a_{k2}\right)\pa_1^2\eta_k\a_{i2}\V_i\\&+3\intb \pa_t\a_{m2}\pa_t^{2}\dfrac{\a_{m2}}{\abs{\pa_1\eta}^5}\pa_1^2\eta_k\a_{k2}\a_{i2}\V_i
\end{split}
\label{lplplt}
\end{equation}
Substituing $\a_{\cdot2}=(-\pa_1\eta_2,\pa_1\eta_1)^\text{T}$ into $I_b$, we have
\begin{equation*}
I_b=-\intb \dfrac{3\pa_t^3\pa_1\eta_k\pa_1\a_{k2}}{\abs{\pa_1\eta}^3}\a_{i2}\V_i
\end{equation*}
by direct calculation.
Then, using \eqref{in2}, we arrive at 
\begin{equation*}
\begin{split}
I_a+I_b=&-\intb\dfrac{\pa_t^3\pa_1\eta_k\pa_1\a_{k2}}{\abs{\pa_1\eta}^3}\a_{i2}\V_i-\dfrac{\pa_t^3\pa_1\eta_k\a_{k2}}{\abs{\pa_1\eta}^3}\pa_1\a_{i2}\V_i\\=&\intb \dfrac{1}{\abs{\pa_1\eta}^3}\pa_t^3\pa_1\eta_k\V_i\left(\a_{k2}\pa_1\a_{i2}-\a_{i2}\pa_1\a_{k2}\right)=\intb \dfrac{\pa_1^2\eta_k\a_{k2}}{\abs{\pa_1\eta}^3} \pa_t^3\a_{i2}\V_i
\end{split}
\end{equation*}
Thus, $I_a+I_b$ can be estimated by using the integration-by-parts technique, similar to \eqref{ttttt}-\eqref{tttttttt}. The other terms of \eqref{lplplt} can be bounded by $\norm{\pa_t^3\eta}_1\norm{v}_{\X^3}P\left(\norm{\eta}_{\X^3},\norm{\pa_t^2\pa_1\eta}_1\right)+P\left(\sup\limits_{t\in [0,T]}\mathfrak E^\epsilon(t)\right)$. Thus, we obtain
\begin{equation*}
\begin{split}
R_{b112}+R_{b1111}=&\hal\dfrac{d}{dt}\intb\dfrac{\pa_1^2\eta_k\a_{k2}}{\abs{\pa_1\eta}^3}\pa_t^3\a_{i2}\pa_t^3\eta_i+\norm{\bp^3\eta}_1^2P(\norm{\eta}_{\X^3})
\end{split}
\end{equation*}

The estimate of the term $R_{b113}$ is more involved since we do not have $\abs{\pa_t^2\pa_1^2\eta}_{\hal}$ or $\abs{\pa_t^3\pa_1\eta_i\a_{i2}}_{\hal}$. It needs to be combined with $R_c$. First, by intergartion by parts, we have
\begin{equation*}
\begin{split}
R_c=&\intd 3\Q\pa_t\a_{ij}\pa_t^{2}\pa_j v_i+\intd  3\Q\pa_t^{2}\a_{ij}\pa_t\pa_jv_i,
\\=&\underbrace{-\intd 3\pa_j\pa_t^3q\pa_t\a_{ij}\pa_t^{2}v_i}_{R_{c1}}\underbrace{+\intb 3\pa_t^3q\pa_t\a_{i2}\pa_t^{2}v_i}_{R_{c2}}\underbrace{+\intd  3\pa_t^3q\pa_t^{2}\a_{ij}\pa_t\pa_jv_i}_{R_{c3}}
\\&\underbrace{-\intd 3\pa_t^3\eta\cdot\na_\eta q\pa_t\a_{ij}\pa_t^2\pa_j v_i-\intd 3\pa_t^3\eta\cdot\nabla_\eta q \pa_t^{2}\a_{ij}\pa_t\pa_jv_i}_{R_{c4}}
\end{split}
\end{equation*}
\begin{equation*}
	\begin{split}
		\int_0^{\mathsf t} R_{c4}\,d{\mathsf t}&\ls	\int_0^{\mathsf t}\norm{\nabla_\eta q\nabla v}_{L^\infty}\norm{\pa_t^3\eta}_0\norm{\pa_t^3\nabla\eta}_0+\norm{\pa_t^2\nabla\eta}_0\norm{\pa_t^3\eta}_{L^6}\norm{\nabla_\eta q}_{L^6}\norm{\pa_t\nabla v}_{L^6}\\&\ls P\left(\sup_{t\in [0,T]}\mathfrak E^\epsilon(t)\right).
	\end{split}
\end{equation*}
For $R_{c3}$, by integration by parts with time, we have 
\begin{equation*}
R_{c3}=\dfrac{d}{dt}\intd \pa_t^2q\pa_t^2\a_{ij}\pa_t\pa_jv_i-\intd \pa_t^2q \pa_t\left(\pa_t^2 \a_{ij}\pa_t\pa_jv_i\right),
\end{equation*}
the second term is bounded by $\norm{\pa_t^2q}_{1}\left(\norm{\pa_t^3\eta}_1\norm{\pa_t\nabla v}_{L^4}+\norm{\pa_t^2\nabla\eta}_{L^4}\norm{\pa_t^2\nabla v}_0\right)$.

For $R_{c1}$, 
\begin{equation*}
R_{c1}=-\dfrac{d}{dt}\intd 3\pa_j\pa_t^2 q \pa_t\a_{ij}\pa_t^2v_i +\intd 3\pa_j\pa_t^2q\pa_t\left(\pa_t\a_{ij}\pa_t^2v_i\right),
\end{equation*}
the second term is bounded by $\norm{\nabla q}_{\X^2}\left(\norm{\pa_t^3v}_0\norm{\nabla v}_{L^\infty}+\norm{\pa_t^2v}_{L^4}\norm{\pa_t^2\nabla \eta}_{L^4}\right)$.

For $R_{c2}$, with the boundary condition \eqref{bdq}, it can be written as
\begin{equation*}
\begin{split}
&\underbrace{-3\intb \pa_t^3\left(\dfrac{\partial_1^2\eta_k\a_{k2}}{\abs{\pa_1\eta}^3}\right)\pa_t\a_{i2}\pa_t^2v_i}_{R_{c21}}\underbrace{-3\intb \pa_t^3\left(-\dfrac{1+\epsilon\pa_t\a_{j2}\a_{j2}}{\abs{\pa_1\eta}^2}\right)\pa_t\a_{i2}\pa_t^{2}v_i}_{R_{c22}}
\end{split}
\end{equation*}
Combining with $R_{b113}$, we obtain
\begin{equation*}
\begin{split}
R_{b113}+R_{c21}=&-3\intb \pa_t\a_{i2}\pa_t^{2}v_i\pa_t^3\dfrac{\pa_1^2\eta_k\a_{k2}}{\abs{\pa_1\eta}^3}+\pa_t\left(\dfrac{\a_{k2}}{\abs{\pa_1\eta}^3}\right)\pa_t^{2}\pa_1^2\eta_k\a_{i2}\pa_t^3v_i\\&+\intb \pa_t\left(\dfrac{\a_{k2}}{\abs{\pa_1\eta}^3}\right)\pa_t^{2}\pa_1^2\eta_k\a_{i2}\pa_t^3\eta\cdot\nabla_\eta v_i\\=&\underbrace{-3\intb \pa_t\a_{i2}\pa_t^{2}v_i\dfrac{\pa_t^3\pa_1^2\eta_k\a_{k2}}{\abs{\pa_1\eta}^3}+\pa_t\left(\dfrac{\a_{k2}}{\abs{\pa_1\eta}^3}\right)\pa_t^{2}\pa_1^2\eta_k\a_{i2}\pa_t^3v_i}_{\mathfrak J}\\&-\sum_{\abs{\alpha}<3}3C_\alpha\intb \pa_t\a_{i2}\pa_t^{2}v_i\pa_t^\alpha\pa_1^2\eta_k\pa_t^{3-\alpha}\dfrac{\a_{k2}}{\abs{\pa_1\eta}^3}\\&+\intb \pa_t\left(\dfrac{\a_{k2}}{\abs{\pa_1\eta}^3}\right)\pa_t^{2}\pa_1^2\eta_k\a_{i2}\pa_t^3\eta\cdot\nabla_\eta v_i
\end{split}
\end{equation*}
The last two lines are bounded by $\norm{\pa_t^3\eta}_1\left(\norm{\pa_t^3\eta}_1+\norm{\pa_t^2\pa_1\eta}_1\right)P\left(\norm{\eta}_{\X^3}\right)$.
For $\mathfrak J$, we have
\begin{equation*}
\begin{split}
\mathfrak J=&-3\dfrac{d}{dt}\intb  \pa_t^{2}\pa_1^{2}\eta_k\a_{k2}\pa_t^{2}v_i\left(\dfrac{\pa_t\a_{i2}}{\abs{\pa_1\eta}^3}\right)\underbrace{+3\intb \pa_t^{2}\pa_1^{2}\eta_k\pa_t^{3}v_i\left(\dfrac{\pa_t\a_{i2}\a_{k2}-\a_{i2}\pa_t\a_{k2}}{\abs{\pa_1\eta}^3}\right)}_{\mathfrak J_a}\\&\underbrace{+3\intb\pa_t^{2}\pa_1^{2}\eta_k\pa_t^{2}v_i\pa_t\left(\dfrac{\pa_t\a_{i2}\a_{k2}}{\abs{\pa_1\eta}^3}\right)}_{\mathfrak J_b}\underbrace{-3\intb\pa_t\dfrac{1}{\abs{\pa_1\eta}^3} \pa_t^{2}\pa_1^{2}\eta_k\a_{k2}\pa_t^{3}v_i\a_{i2}}_{\mathfrak J_c}.
 \end{split}
\end{equation*}
For $\mathfrak J_a$, when $i=k$, it is zero. When $i\neq k$, we have
\begin{equation}
\begin{split}
&3\intb \left(\pa_t^{2}\pa_1^{2}\eta_1\pa_t^{3}v_2-\pa_t^{2}\pa_1^{2}\eta_2\pa_t^{3}v_1\right)\left(\dfrac{\pa_t\a_{12}\a_{22}-\a_{12}\pa_t\a_{22}}{\abs{\pa_1\eta}^3}\right)\\=&3\dfrac{d}{dt}\intb \left(\pa_t^2\pa_1^2\eta_1\pa_t^{3}\eta_2-\pa_t^2\pa_1^2\eta_2\pa_t^3\eta_1\right)\left(\dfrac{-\pa_t\pa_1\eta_i\a_{i2}}{\abs{\pa_1\eta}^3}\right)\\&-3\intb \left(\pa_t^2\pa_1^2\eta_1\pa_t^{3}\eta_2-\pa_t^{2}\pa_1^2\eta_2\pa_t^{3}\eta_1\right)\pa_t\left(\dfrac{-\pa_t\pa_1\eta_i\a_{i2}}{\abs{\pa_1\eta}^3}\right)\\&+3\intb \left(\pa_t^3\pa_1\eta_1\pa_t^3\eta_2-\pa_t^3\pa_1\eta_2\pa_t^3\eta_1\right)\pa_1\left(\dfrac{-\pa_t\pa_1\eta_i\a_{i2}}{\abs{\pa_1\eta}^3}\right).
\end{split}
\label{ebt}
\end{equation}
The last two terms of \eqref{ebt} can be bounded by
\begin{equation*}
	\left(\norm{\pa_t^3\eta}_1+\norm{\pa_t^2\pa_1\eta}_1\right)\norm{\bp^3\eta}_1\abs{\bp^2\pa_1\eta_i\a_{i2}}_1P\left(\norm{\eta}_{\X^3}\right).
\end{equation*}
For $\mathfrak J_b$, it can be bounded by
\begin{equation*}
	\abs{\pa_t^{2}\pa_1^{2}\eta}_{-\hal}\abs{\pa_t^3\eta}_{\hal}\abs{\pa_1v\pa_1v}_{1}+\abs{\pa_t^2\pa_1^2\eta_k\a_{k2}}_0\abs{\pa_t^2v\pa_t^2\pa_1\eta}_0P\left(\norm{\eta}_{\X^3}\right)
\end{equation*}
For $\mathfrak J_c$, 
\begin{equation*}
\begin{split}
3\intb\pa_t\dfrac{1}{\abs{\pa_1\eta}^3} \pa_t^{2}\pa_1^{2}\eta_k\a_{k2}\pa_t^{3}v_i\a_{i2}\ls \abs{\pa_t^2\pa_1^2\eta_k\a_{k2}}_\hal\abs{\pa_t^3v_i\a_{i2}}_{-\hal}\abs{\pa_1v}_1.
\end{split}
\end{equation*}
Thus, we arrive at
\begin{equation*}
	\int_0^{\mathsf t} R_{b113}+R_{c21} \ls P\left(\sup_{t\in [0,T]}\mathfrak E^\epsilon(t)\right).
\end{equation*}
For $R_{c22}$, we have
\begin{equation*}
	\int_0^{\mathsf t} R_{c22}\ls \delta \int_0^{\mathsf t} \epsilon\norm{\pa_t^3\pa_1v}_1^2+P\left(\sup_{t\in [0,T]}\mathfrak E^\epsilon(t)\right)
\end{equation*}

For $R_d$, we have
\begin{equation*}
R_d=\intd \epsilon\abs{S_\eta(\pa_t^3v)}^2+\sqrt\epsilon P\left(\norm{\pa_t^3\eta}_{1},\norm{\eta}_{\X^3},\norm{v}_{\X^3}\right)\norm{\sqrt\epsilon\nabla\pa_t^3v}_0.
\end{equation*}
Integrating with respect to time, combining the estimates and using Korn's inequality, we have
\begin{equation}\label{ohoh1t}
\begin{split}
&\norm{\mathcal V(\mathsf t)}_0^2+\norm{\pa_t^3\nabla \eta(t)}_0^2+\dfrac{1}{4}\abs{\pa_t^3\pa_1\eta_i \a_{i2}(\mathsf t)}^2_{0}+\hal\epsilon\int_0^{\mathsf t}\norm{\pa_t^3\nabla v}_0^2+\mathfrak F \\&\leq M_0+\delta\left(\norm{\pa_t^3\eta}_1^2+\norm{\epsilon\pa_t^3\nabla v}_0^2+\abs{\pa_t^3\pa_1\eta\cdot n}_0^2\right)(\mathsf t)+T^{\frac{1}{4}}P\left(\sup_{\mathsf t\in [0,T]}\mathfrak E^\epsilon(\mathsf t)\right).
\end{split}
\end{equation}
where 
\begin{equation*}
\begin{split}
\mathfrak F=&-\hal\intb \mathfrak B\pa_t^3\a_{i\ell}N_\ell\pa_t^3\eta_i\,d\sigma-3\intb \dfrac{-\pa_t\pa_1\eta_i\a_{i2}}{\abs{\pa_1\eta}^3}\pa_t^2\pa_1\a_{i2}\pa_t^3\eta_i-3\intb  \pa_t^{2}\pa_1^{2}\eta_k\a_{k2}\pa_t^{2}v_i\left(\dfrac{\pa_t\a_{i2}}{\abs{\pa_1\eta}^3}\right)\\&+\intd 3\pa_t^2q\pa_t^2\a_{ij}\pa_t\pa_jv_i+\intd 3\pa_j\pa_t^2 q \pa_t\a_{ij}\pa_t^2v_i.
\end{split}
\end{equation*}
Since we only have the estimate for $\int_0^t\norm{\pa_t^2q}_{1}^2$, we square both sides of \eqref{ohoh1t} and integrate from $0$ to $t$ again. Combining the techniques used for \eqref{sufc}, we obtain
\begin{equation*}
\begin{split}
&\int_0^t\norm{\V(\tau)}_0^4+\norm{\pa_t^3\nabla\eta(\tau)}_0^4+\abs{\pa_t^3\pa_1\eta\cdot n(\tau)}_0^4\,d{\mathsf t}+\int_0^t\epsilon^2\left(\int_0^{\mathsf t} \norm{\pa_t^3\nabla v}_0^2\right)^2\,d{\mathsf t}\\&\ls  M_0+\delta\sup_{t\in[0,T]}\mathfrak{E}^{\epsilon}(t)+T^{\frac{1}{4}} P\left(\sup_{t\in[0,T]}\mathfrak{E}^{\epsilon}(t)\right).
\end{split}
\end{equation*}
By the definition of $\mathcal{V}$, using \eqref{ohoh1}  and the fundamental theorem of calculous, we obtain
\begin{equation*}
\begin{split}
\int_0^t\norm{\pa_t^3 v(\mathsf t)}_0^4\,d\mathsf t&\ls\int_0^t\norm{\mathcal V(\mathsf t)}_0^4\,d\mathsf t+\int_0^t\norm{\pa_t^3\eta}_0^4\norm{\a_{jk}\partial_{k} v}_{L^{\infty}}^4(\mathsf t)\,d\mathsf t
\\&\leq M_0+\delta\sup_{t\in[0,T]}\mathfrak{E}^{\epsilon}(t)+T^{\frac{1}{4}} P\left(\sup_{t\in[0,T]}\mathfrak{E}^{\epsilon}(t)\right).
\end{split}
\end{equation*}
We thus conclude the proposition.
\end{proof}
\subsection{Normal deriviate estimates}\label{curle}
Now, we derive the normal derivative estimates by using the equation:
\begin{equation}
\label{noreq}
-\pa_2^2\eta-\epsilon\abs{\a_{\cdot2}}^2\pa_2^2v=\pa_1^2\eta-\pa_t v+J\nabla_\eta q+\epsilon\left(\sum_{\alpha+\beta \neq 4}\pa_\alpha(\a_{k\alpha}\a_{k\beta}\pa_\beta v)+\pa_2\abs{\a_{\cdot2}}^2\pa_2v\right)-\phi^\epsilon+\epsilon\nabla\cdot\Psi
\end{equation}
This implies that we can estimate the normal derivatives of $\eta, v$ in terms of the $\eta,v$ terms with less normal derivatives, and we obtain the following proposition:
\begin{proposition}\label{normal}
For $t\in [0,T]$ with $T\le T_\epsilon$, it holds that
\begin{equation}
\label{tnn}
\begin{split}
\int_0^t\norm{v(t)}_{\X^3}^2+\norm{\nabla\eta(t)}_{\X^3}^2+\norm{\epsilon\nabla v}_{\X^3}^2\,d{\mathsf t}+\norm{\sqrt\epsilon \nabla^2\eta(t)}_{\X^2}^2+\norm{\eta(t)}_{\X^3}^2\\\leq M_0+\delta\sup_{t\in[0,T]}\mathfrak{E}^{\epsilon}(t)+ T^{\frac{1}{4}} P\left(\sup_{t\in[0,T]}\mathfrak{E}^{\epsilon}(t)\right).
\end{split}
\end{equation}
\end{proposition}
\begin{proof}
	First, applying $\pa_t^2$ in equation \eqref{noreq} and then squaring both sides, we have 
\begin{equation*}
\begin{split}
&\norm{\pa_t^2\pa_2^2\eta(t)}_0^2+\norm{\epsilon\pa_t^2\pa_2^2v(t)}_0^2+\dfrac{d}{dt}\intd \epsilon\abs{\a_{\cdot2}}^2\abs{\pa_t^2\pa_2^2\eta}^2-\intd \epsilon \pa_t\abs{\a_{\cdot 2}}^2\abs{\pa_t^2\pa_2^2\eta}^2\\&\ls \norm{\nabla\pa_t^2\pa_1\eta}_0^2+\norm{\pa_t^3v}_0^2+\norm{\pa_t^2(\nabla_\eta q)}_0^2+\norm{\epsilon\pa_t^2\left(\sum_{\alpha+\beta \neq 4}\pa_\alpha(\a_{k\alpha}\a_{k\beta}\pa_\beta v)+\pa_2\abs{\a_{\cdot2}}^2\pa_2v\right)}_0^2\\&\quad+\norm{\epsilon\nabla\cdot\pa_t^2\Psi}_0^2.
\end{split}
\end{equation*}
Integrating over time and using the pressure estimate \eqref{press2} and the $\pa_t^3$ energy estimate \eqref{tee3et}, we obtain
\begin{equation*}
\begin{split}
&\int_0^t\norm{\pa_t^2\pa_2^2\eta}_0^2+\int_0^t\norm{\epsilon\pa_t^2\pa_2^2v}_0^2+\sup_t\norm{\sqrt\epsilon\pa_t^2\pa_2^2\eta}_0^2\\&\ls M_0+\int_0^t\left(\norm{\epsilon\pa_t^2\nabla^2\eta}_0^2+\norm{\epsilon\pa_t^2\pa_1 v}_1^2\right)P\left(\norm{\eta}_{\X^3}^2\right)+\delta \sup_{t\in [0,T]}\mathfrak E^\epsilon(t)+T^{\frac{1}{4}}P\left(\sup_{t\in [0,T]}\mathfrak E^\epsilon(t)\right).
\end{split}
\end{equation*}
%
In a similar way, we can obtain  
$$\int_0^t\left(\norm{\bp^2\pa_2^2\eta}_0^2+\norm{\epsilon\bp^2\pa_2^2v}_0^2\right)\,d{\mathsf t}+\sup_t \norm{\sqrt\epsilon\bp^2\pa_2^2\eta}_0^2\ls M_0+T^{\frac{1}{4}}P\left(\sup_{t\in [0,T]}\mathfrak E^\epsilon(t)\right).$$
Then, we can obtain higher-order normal derivatives by induction. By applying $\bp^\ell\pa_2^{2-\ell}, \ell\leq 2$, we obtain
\begin{equation*}
\begin{split}
&\int_0^t\left(\norm{\bp^\ell\pa_2^{4-\ell}\eta}_0^2+\norm{\epsilon\bp^\ell\pa_2^{4-\ell}v}_0^2\right)\,d{\mathsf t}+\sup_t \norm{\sqrt\epsilon\bp^\ell\pa_2^{4-\ell}\eta}_0^2\\&\leq M_0+\int_0^t\norm{\nabla\bp^{\ell+1}\pa_2^{2-\ell}\eta}_0^2+\norm{\bp^\ell\pa_2^{2-\ell}\pa_t v}_0^2\,d{\mathsf t}\\&\quad+\int_0^t\left(\norm{\nabla q}_{\X^2}^2+\norm{\epsilon\bp^{\ell}\pa_2^{3-\ell}\nabla^2\eta}_0^2+\norm{\epsilon \bp^{\ell+1}\pa_2^{3-\ell}\nabla v}_0^2\right)P\left(\norm{\eta}_{\X^3}^2\right)\\&\leq M_0+T^{\frac{1}{4}}P\left(\sup_{t\in [0,T]}\mathfrak E^\epsilon(t)\right).
\end{split}
\end{equation*}
Thus, we conclude the proposition.
\end{proof}
\subsection{Synthesis}

We now collect the estimates derived previously to obtain our estimates and verify the a priori assumptions \eqref{inin3}. That is, we now present the
\begin{proof}[Proof of Theorem \ref{th43}]
Combined with Proposition \ref{tane1}, \ref{tane2} and \ref{normal}, we finally obtain that
\begin{equation*}
	\sup_{[0,T]}\mathfrak{E}^{\epsilon}(t)\leq M_0+\delta\sup_{t\in[0,T]}\mathfrak{E}^{\epsilon}(t)+T^{\frac{1}{4}} P\left(\sup_{t\in[0,T]}\mathfrak{E}^{\epsilon}(t)\right).
\end{equation*}
By taking $\delta$ to be sufficiently small, we obtain a time of existence $T_1$ independent of $\epsilon$ and an estimate on $[0,T_1]$ independent of $\epsilon$ of the type:
\begin{equation*}
\sup_{[0,T_1]}\mathfrak{E}^{\epsilon}(t)\leq 2M_0.
\end{equation*}

The proof of Theorem \ref{th43} is thus completed.
\end{proof}
\section{Local well-posedness for the elastodynamic system}
In this section, we present the
\begin{proof}[Proof of Theorem \ref{thm}]
	For each $\epsilon>0$, we recover the dependence of the solutions to the viscoelasticity system \eqref{approximate} on $\epsilon$ as $(v^\epsilon,q^\epsilon,\eta^\epsilon)$, which can be constructed
by the approach in \cite{Le,XZZ_13} (the reader is also referred to \cite{WangTan}).
The viscosity-independent estimates \eqref{bound} combined with a continuation argument (see \cite[Section 9]{CS06}) imply that  $(v^\epsilon,q^\epsilon,\eta^\epsilon)$ is indeed a solution of \eqref{approximate} on the time interval $[0,T_1]$ and yields the strong convergence of $(v^\epsilon,q^\epsilon,\eta^\epsilon)$ to a limit $(v ,q ,\eta )$, up to extraction of a subsequence, which is more than sufficient for us to pass to the limit as $\epsilon\rightarrow0$ in \eqref{approximate} for each $t\in[0,T_1]$. We then find that $(v ,q ,\eta )$ is a strong solution to \eqref{eq:elastic} on $[0,T_1]$. 
This shows the existence of solutions to \eqref{eq:elastic}.

After we obtain a strong solution $(v,q,\eta)$ by taking the viscosity vanishing limit, the estimate \eqref{enesti} does not hold directly by taking $\epsilon=0$. This is due to the different time regularities of $\mathfrak E(t)$ and $\mathfrak E^\epsilon(t)$. For instance, for $\mathfrak E(t)$, $\pa_t^3v\in L^\infty_t(L^2)$, while it is only $L^4_t(L^2)$ for the energy $\mathfrak E^\epsilon(t)$. Thus, to prove the estimate \eqref{enesti}, we need to do a priori estimate of $(v, q,\eta)$ with energy $\mathfrak E(t)$. The strategy of a priori estimates is very similar to our viscosity-independent one: hence, we only point out the different places to improve the time regularity.

The main improvement is the $\pa_t^3$ energy estimate; in fact, without the viscosity, the $\pa_t^3$ energy estimate can be improved to the $L^\infty_t$ space:
\begin{equation*}
\norm{\pa_t^3v}_0^2(t)+\norm{\nabla\pa_t^3\eta}_0^2(t)+\abs{\pa_t^3\pa_1\eta\cdot n}_0^2(t) \ls M_0+T^{\frac{1}{4}}P\left(\sup_{t\in [0,T]}\mathfrak E(t)\right)+\delta\sup_{t\in [0,T]}\mathfrak E(t).
\end{equation*}
This because we can have $L^\infty_t$ bound for the pressure $\norm{\dt^2\nabla q}_1$ (not only $L^2_t$ bound when the viscosity exists).

Next, with the improved tangential energy estimate, the pressure estimate can be improved from \eqref{press2} to: 
\begin{equation*}
\norm{q}_{\X^2}^2(t)+\norm{\nabla q}_{\X^2}^2(t)\ls M_0+T^{\frac{1}{4}}P\left(\sup_{t\in [0,T]}\mathfrak E(t)\right)+\delta\sup_{t\in [0,T]}\mathfrak E(t).
\end{equation*}
Then, the estimates of the normal derivatives ultimately improve.

Last, the uniqueness can be obtained by a similar approach as for the a priori estimates; the reader is also referred to \cite[Section 6]{GuW_2016}.
\end{proof}
\begin{proof}[Proof of Theorem \ref{vanish}]
	As we explained in Section 4, the difference between our approximate viscosity system and the viscoelasticity is due to the compatibility issue of the initial data. When the inital data are prepared well for visocoelasticity, the difference disappears, our viscosity-indepedent a priori esimtates also valid for the viscoelasticity; then, we justify the inviscid limit.
\end{proof}

%
\section*{Appendix}
In this apprendex, we contrust the smoothed initial data $(\eta_0^\epsilon,v_0^\epsilon)$ satisfying the compatibility condition \eqref{zcomp}, \eqref{1comp}. Moreover, we introduce a smoothing function $\phi^\epsilon$ to make the corresponding $\pa_tv(0)$ (calculated by the equation) satisfy the compatibility condition \eqref{2comp}. Some of this smoothing approach is motivated by Cheng and Shkoller \cite[Section 5.1]{Cheng_Shkoller_10}. 

First, we denote $\mathcal E_\Omega$ as a Sobolev extension operator from $\Omega$ to $\mathbb T\times \bR$, where $\rho_\kappa$ is a 2-D standard mollifier, $\Lambda_\kappa$ is a 1-D standard mollifier, and $\kappa$ is the mollifier parameter.
\subsection*{Smoothing $\eta_0$}
We smooth $\eta_0^\kappa$ by solving the following fourth-order elliptic equation:
\begin{equation*}
	\begin{cases}
		\Delta^2\eta_{0i}^\kappa=\rho_\kappa*\mathcal E_\Omega(\Delta^2\eta_{0i}), &\text{in } \Omega,\\
	\eta_{0i}^\kappa=\Lambda_\kappa*\eta_{0i}, &\text{on } \Gamma,\\
	\pa_2\eta_{0i}^\kappa= \dfrac{-\epsilon_{ij}\pa_1\left(\Lambda_\kappa*\eta_{0j}\right)}{\abs{\pa_1(\Lambda_\kappa*\eta_0)}^2}\,\,&\text{on } \Gamma.
	\end{cases}
\end{equation*}
Here, $\epsilon_{ij}$ is a permutation symbol such that $\epsilon_{12}=1, \epsilon_{21}=-1$ and $\epsilon_{kk}=0$.

Thus, $\eta_0^\kappa \in C^\infty(\overline\Omega)$, and
\begin{equation*}
	\Pi_0^\kappa\pa_2\eta_0^\kappa=\pa_2\eta_0^\kappa-\dfrac{\ak_{\cdot 2}(0)}{\abs{\ak_{\cdot 2}(0)}^2}=0, \,\,J^\kappa_0=1
\end{equation*}
is satisfied on $\Gamma$, which is the zeroth-order compatibility condition.
\subsection*{Smoothing $v_0$}
With the zeroth compatibility condition \eqref{zcomp}, we have 
\begin{equation*}
	\a_{i1}\a_{i2}\Big|_{t=0}=0\,\, \text{on } \Gamma.
\end{equation*}
Then, by straightforward calculation, we see
\begin{equation*}
	2S_{\eta}(v)_{ij}\a_{j2}\Big|_{t=0}=\left(\abs{\a_{\cdot 2}}^2\pa_2v_i-\pa_t\a_{i2}\right)\Big|_{t=0}=\abs{\pa_1\eta_0}^2\left(\pa_2v_{0i}-\dfrac{\pa_t\a_{i2}(0)}{\abs{\pa_1\eta_0}^2}\right)
\end{equation*}
on the boundary $\Gamma$.
Since $v_0$ satisfies 1st-order compatibility condition:
\begin{equation*}
	\Pi_0\left(\pa_2v_{0i}-\dfrac{\pa_t\a_{i2}(0)}{\abs{\pa_1\eta_0}^2}\right)=0,
\end{equation*}
we have $v_0$ that satisfies
\begin{equation*}
	\Pi_0 2S_{\eta_0}(v_0)_{ij}\a_{j2}(0)=0.
\end{equation*}
Thus, we can define $r_0$ by solving 
\begin{equation*}
	\begin{cases}
		-\Delta_{\eta_0}r_0=0, &\text{in } \Omega,\\
		r_0= 2S_{\eta_0}(v_0)_{ij}\dfrac{\a_{j2}(0)\a_{i2}(0)}{\abs{\pa_1\eta_0}^2}, &\text{on } \Gamma.
\end{cases} 
\end{equation*}

Now, we smooth $v_0$ by solving the following Stokes-type equation:
\begin{equation*}
	\begin{cases}
	-\Delta_{\eta^\kappa_0} v_{0}^\kappa+\nabla_{\eta_0^\kappa}r_0^\kappa=\rho_\kappa*\mathcal E_\Omega(-\Delta_{\eta_0} v_0+\nabla_{\eta_0}r_0), &\text{in } \Omega,\\
	\Div_{\eta_0^\kappa} v_0^\kappa=0, &\text{in } \Omega,\\ \left(-2S_{\eta_0^\kappa}v_0^\kappa+r_0^\kappa I\right)N^\kappa_0=0,&\text{on } \Gamma.
	\end{cases}
\end{equation*}
It is clear that above, $v_0^\kappa$ satisfies  
\begin{equation*}
\Pi_0 2S_{\eta^\kappa_0}v_0^\kappa N_0^\kappa=0.
\end{equation*}
Thus, since $J_0^\kappa=1$ on $\Gamma$, the smoothed $v_0^\kappa$ also satisfies the 1st compatibility condition.
\subsection*{Smoothing $\pa_tv(0)$ and introducing $\phi^\kappa$}
For the last step, we smooth $w_1^\kappa$ by solving the Stokes equation
\begin{equation*}
	\begin{cases}
		-\Delta_{\eta^\kappa_0} w_1^\kappa+\nabla_{\eta_0^\kappa}r_1^\kappa=\rho_\kappa*\mathcal E_\Omega(-\Delta_{\eta_0} \pa_tv(0)+\nabla_{\eta_0}r_1), &\text{in } \Omega,\\
\Div_{\eta_0^\kappa} w_1^\kappa=-\partial_{\eta^\kappa_{0i}} v^\kappa_{0j}\partial_{\eta^\kappa_{0j}}{v^\kappa_0}_i, &\text{in } \Omega,\\ \left(-2S_{\eta_0^\kappa}w_1^\kappa+r_1^\kappa I\right)N^\kappa_0=g^\kappa&\text{on } \Gamma.
	\end{cases}
\end{equation*}
where $r_1$ is defined by the following elliptic equation:
\begin{equation*}
	\begin{cases}
		-\Delta_{\eta_0}r_1=0, &\text{in } \Omega,\\
		r_1= \left(2S_{\eta_0}\left(\pa_tv(0)\right)N_0+g\right)\cdot N_0, &\text{on } \Gamma.
\end{cases} 
\end{equation*}
\begin{equation*}
	g_i:=\dfrac{1}{\abs{\pa_1\eta_0}}\left(\abs{\pa_1\eta}^2r_0\pa_t\a_{i2}(0)+2S_{\eta_0}(v_0)_{ij}\pa_t\a_{j2}(0)+\pa_t\a_{i\alpha}(0)\pa_{\alpha}v_{0j}\a_{j2}(0)+\pa_t\a_{j\alpha}(0)\pa_\alpha v_{0i}\a_{j2}(0)\right)
\end{equation*}
\begin{equation*}
g_i^\kappa:=\dfrac{1}{\abs{\pa_1\eta^\kappa_0}}\left(\abs{\pa_1\eta_0^\kappa}^2r_0^\kappa\pa_t\a_{i2}^\kappa(0)+2S_{\eta^\kappa_{0}}(v^\kappa_{0})_{ij}\pa_t\a^\kappa_{j2}(0)+\pa_t\a^\kappa_{i\alpha}(0)\pa_{\alpha}v^\kappa_{0j}\a^\kappa_{j2}(0)+\pa_t\a^\kappa_{j\alpha}(0)\pa_\alpha v^\kappa_{0i}\a^\kappa_{j2}(0)\right)
\end{equation*}
The modification term $\phi^\kappa$ is defined by
\begin{equation*}
	\phi^\kappa=\Delta \eta_0^\kappa-J_0^\kappa\nabla_{\eta^\kappa_0}q_0^\kappa-\omega_1^\kappa.
\end{equation*}
Adding $\phi^\kappa$ to equation \eqref{eq:elastic}$_2$, we can guarantee that the $\pa_tv(0)$ calculated by the equation is equal to $\omega_1^\kappa$, which satisfies the 2nd compatibility condition \eqref{2comp}. 

Last, we take $\kappa=\dfrac{1}{\abs{\ln\epsilon}}$ and re-denote $(\eta_0^\kappa,v_0^\kappa)$ as $(\eta_0^\epsilon, v^\epsilon_0)$ and $\phi^\kappa$ as $\phi^\epsilon$. In this way, by the property of the mollifier
$$\norm{\partial^\ell f^\kappa}_0\ls \dfrac{1}{\kappa^s}\norm{f}_0,$$
we see that $\norm{\sqrt\epsilon f^\kappa}_{s}^2\ls \norm{f}_0$ holds for any $s$ and $\epsilon <\hal$ uniformly.

\section*{Conflict of Interest:}
The authors declare that they have no conflict of interest.


\begin{thebibliography}{99}
\bibitem{AD}
T. Alazard, J. M. Delort.
Global solutions and asymptotic behavior for two dimensional gravity water waves.
\emph{Ann. Sci. \'Ec. Norm. Sup\'er.} \textbf{48}  (2015), no. 5, 1149--1238.

\bibitem{Alinhac}
S. Alinhac.
Existence d'ondes de rar\'efaction pour des syst\`emes quasi-lin\'eaires hyperboliques multidimensionnels.(French. English summary) [Existence of
rarefaction waves for multidimensional hyperbolic quasilinear systems]
\emph{Comm. Partial Differential Equations} \textbf{14} (1989), no. 2, 173--230.
\bibitem{Cai}
Y. Cai, Z. Lei, F. Lin, N. Masmoudi.
Vanishing Viscosity Limit for Incompressible Viscoelasticity in Two Dimensions.
\emph{Comm. Pure Appl. Math.} (2019) https://doi.org/10.1002/cpa.21853

\bibitem{CSW_20}
G. Chen, P. Secchi, T. Wang.
Stability of Multidimensional ThermoelasticContact Discontinuities.
Preprint (2019), arXiv:1912.13343.
\bibitem{Chen_17}
R. Chen, J. Hu, D. Wang.
Linear stability of compressible vortex sheets in two-dimensional elastodynamics.
\emph{Adv. Math.} \textbf{311}(2017), 18--60.

\bibitem{Chen_19}
R. Chen, J. Hu, D. Wang.
Linear stability of compressible vortex sheets in 2D elastodynamics: variable coefficients.
\emph{Math. Ann. } (2019). https://doi.org/10.1007/s00208-018-01798-w
\bibitem{Chen_06}
Y. Chen, P. Zhang.
The global existence of small solutions to the incompressible viscoelastic fluid system in 2 and 3 space dimensions.
\emph{Comm. Partial Differential Equations} \textbf{31} (2006), no.12, 1793--1810.

\bibitem{Cheng_Shkoller_10}
C. Cheng, S. Shkoller.
The Interaction of the 3D Navier–Stokes Equations with a Moving Nonlinear Koiter Elastic Shell.
\emph{SIAM J. Math. Anal.} \textbf{42} (2010), no.3, 1094--1155.


\bibitem{CL_00}
D. Christodoulou, H. Lindblad.
On the motion of the free surface of a liquid.
\emph{Comm. Pure Appl. Math.} \textbf{53} (2000), no. 12, 1536--1602.


\bibitem{CS06}
D. Coutand, S. Shkoller,
The Interaction between Quasilinear Elastodynamics and the Navier-Stokes Equations.
\emph{Arch. Ration. Mech. Anal.} \textbf{179} (2006), no.3, 303--352.
\bibitem{CS07}
D. Coutand, S. Shkoller.
Well-posedness of the free-surface incompressible Euler equations with or without surface tension.
\emph{J. Amer. Math. Soc.}  \textbf{20} (2007), no. 3, 829--930.

\bibitem{DS_10}
D. Coutand, S. Shkoller.
A simple proof of well-posedness for the free-surface incompressible Euler equations.
\emph{Discrete Contin. Dyn. Syst. Ser. S.} \textbf{3} (2010), no. 3, 429--449.

\bibitem{Fil2008}
M. Lopes Filho, A. Mazzucato, H. Nussenzveig Lopes. Vanishing viscosity limit for incompressible flow inside a rotating circle. \emph{Physica D Nonlin. Phenom.} \textbf{237} (2008), 1324--1333.

\bibitem{Fil2008ii}M. Lopes Filho, A. Mazzucato, H. Nussenzveig Lopes, M. Taylor. Vanishing viscosity limits and boundary layers for circularly symmetric 2D flows. \emph{Bull. Brazilian Math. Soc.} \textbf{39} (2008), 471--513.


\bibitem{GMS1}
P. Germain, N. Masmoudi, J. Shatah.
Global solutions for the gravity water waves equation in dimension 3.
\emph{Ann. of Math. (2)} \textbf{175} (2012), no. 2, 691--754.
\bibitem{GMS2}
P. Germain, N. Masmoudi, J. Shatah.
Global solutions for capillary waves equation.
\emph{Comm. Pure Appl. Math.} \textbf{68} (2015), no. 4, 625--687.
\bibitem{GuW_2016}
X. Gu, Y. Wang.
On the construction of solutions to the free-surface incompressible ideal magnetohydrodynamic equations.
\emph{J. Math. Pures Appl.} \textbf{128} (2019),1--41.


\bibitem{GW_18}
X. Gu, F. Wang.
Well-posedness of the free boundary problem in incompressible elastodynamics under the mixed type stability condition. \emph{J. Math. Anal. Appl.} \textbf{482} (2020), 123529.


\bibitem{Hao_16}
C. Hao, D. Wang.
A priori estimates for the free boundary problem of incompressible neo-Hookean elastodynamics.
\emph{J. Differential Equations.} \textbf{261} (2016), no. 1, 712--737.

\bibitem{Hu_19}
X. Hu, Y. Huang.
Well-posedness of the free boundary problem for incompressible elastodynamics
\emph{J. Differential Equations.} \textbf{266} (2019), no. 12, 7844--7889.

\bibitem{IP}
A. Ionescu, F. Pusateri.
Global solutions for the gravity water waves system in 2D.
\emph{Invent. Math.} \textbf{199} (2015), no. 3, 653--804.

\bibitem{IP2}
A. Ionescu, F. Pusateri.
Global regularity for 2D water waves with surface tension.
\emph{Mem. Amer. Math. Soc.}, to appear.

\bibitem{kato72}
Kato, T. 
Nonstationary flows of viscous and ideal fluids in $\bR^3$. 
\emph{J. Functional Analysis.} \textbf{9} (1972), 296--305.


\bibitem{Lannes}
D. Lannes.
Well-posedness of the water-waves equations.
\emph{J. Amer. Math. Soc.} \textbf{18} (2005), no. 3, 605--654.

\bibitem{Le}
Le Meur, H.VJ.
Well-posedness of surface wave equations above a viscoelastic fluid.
\emph{J. Math. Fluid Mech.}\textbf{13} (2011), no. 4, 481--514.

\bibitem{Lei_16}
Z. Lei. 
Global well-posedness of incompressible elastodynamics in two dimensions. 
\emph{Comm. Pure Appl. Math.} \textbf{69} (2016), no. 11, 2072--2106.

\bibitem{Lei_08}
Z. Lei, C. Liu, Y. Zhou, 
Global solutions for incompressible viscoelastic fluids.
\emph{Arch. Ration. Mech. Anal.} \textbf{188} (2008) no.3, 371--398.

\bibitem{LZ_05}
Z. Lei, Y. Zhou. Global existence of classical solutions for the two-dimensional Oldroyd model via the incompressible limit. 
\emph{SIAM J. Math. Anal.} \textbf{37} (2005), no. 3, 797--814.

\bibitem{LWZ_2018}
H. Li, W. Wang, Z. Zhang.
Well-posedness of the free boundary problem in incompressible elastodynamics.
\emph{J. Differential Equations.} (2019), https://doi.org/10.1016/j.jde.2019.07.001.
\bibitem{L_12}
F. Lin.
Some analytical issues for elastic complex fluids.
\emph{Comm. Pure Appl. Math.} \textbf{65} (2012), no. 7, 893--919.

\bibitem{LLZ_05}
F. Lin, C.Liu, P. Zhang. 
On hydrodynamics of viscoelastic fluids. 
\emph{Comm. Pure Appl. Math.} \textbf{58} (2005), no. 11, 1437--1471.

\bibitem{LZ_08}
F. Lin, P. Zhang, 
On the initial–boundary value problem of the incompressible viscoelastic fluid system, 
\emph{Comm. Pure Appl. Math.} \textbf{61} (2008), no.4 539--558.

\bibitem{Lindblad05}
H. Lindblad.
Well-posedness for the motion of an incompressible liquid with free surface boundary.
\emph{Ann. of Math. (2)} \textbf{162} (2005), no. 1, 109--194.

\bibitem{mas_07}
N. Masmoudi. 
Remarks about the inviscid limit of the Navier-Stokes system. 
\emph{Comm. Math. Phys.} \textbf{270} (2007), no. 3, 777--788.

\bibitem{MasRou}
N. Masmoudi, F. Rousset.
Uniform regularity and vanishing viscosity limit for the free surface Navier-Stokes equations.
\emph{Arch. Ration. Mech. Anal.} \textbf{223} (2017), no.1, 301--417.

\bibitem{MWX}
Y. Mei, Y. Wang, Z. Xin.
Uniform regularity for the free surface compressible Navier–Stokes equations with or without surface tension.
\emph{Math. Models Methods Appl. Sci.} \textbf{28} (2018), no.2, 259--336.


\bibitem{TaylorI}
A. Mazzucato, M. Taylor.
Vanishing Viscosity Limits for a Class of Circular Pipe Flows.
\emph{Comm. Part. Diff. Eqs} \textbf{36}, (2011), no.2, 328--361.

 \bibitem{N}
V. I. Nalimov.
The Cauchy-Poisson problem. (Russian)
\emph{Dinamika Splo$\breve{s}$n. Sredy Vyp. 18 Dinamika $\breve{Z}$idkost. so Svobod. Granicami.} \textbf{254} (1974), 104--210.



\bibitem{SZ}
J. Shatah, C. Zeng.
Geometry and a priori estimates for free boundary problems of the Euler equation.
\emph{Comm. Pure Appl. Math.} \textbf{61} (2008), no. 5, 698--744.

\bibitem{ST_05}
T. Sideris, B. Thomases.
Global existence for three-dimensional incompressible isotropic elastodynamics via the incompressible limit. 
\emph{Comm. Pure Appl. Math.} \textbf{58} (2005), no. 6, 750--788.

\bibitem{ST_06}
T. Sideris, B. Thomases.
Global existence for three-dimensional incompressible isotropic elastodynamics. 
\emph{Comm. Pure Appl. Math.} \textbf{60} (2007), no. 12, 1707--1730. 




\bibitem{swann}
H.S.G Swann.
The convergence with vanishing viscosity of nonstationary Navier-Stokes flow to ideal flow in $\bR^3$ . 
\emph{Trans. Amer. Math. Soc.} \textbf{157} (1971), 373--397.


\bibitem{WangTan}
Z. Tan, Y. Wang. Zero Surface Tension Limit of Viscous Surface Waves.
\emph{Comm. Math. Phys.} \textbf{328} (2014), no. 2, 733--807.

\bibitem{Trak_18}
Y. Trakhinin.
Well-posedness of the free boundary problem in compressible elastodynamics.
\emph{J. Differential Equations.} \textbf{264} (2018), 1661--1715.

\bibitem{Wang_15}
Y. Wang, Z. Xin.
Vanishing viscosity and surface tension limits of incompressible viscous surface waves.
Preprint (2015), arXiv: 1504.00152.

\bibitem{Wu1}
S. Wu.
Well-posedness in Sobolev spaces of the full water wave problem in 2-D.
\emph{Invent. Math.} \textbf{130} (1997), no. 1, 39--72.


\bibitem{Wu2}
S. Wu.
Well-posedness in Sobolev spaces of the full water wave problem in 3-D.
\emph{J. Amer. Math. Soc.} \textbf{12} (1999), no. 2, 445--495.


\bibitem{Wu3}
S. Wu.
Almost global wellposedness of the 2-D full water wave problem.
\emph{Invent. Math.} \textbf{177} (2009), no. 1, 45--135.


\bibitem{Wu4}
S. Wu.
Global wellposedness of the 3-D full water wave problem.
\emph{Invent. Math.} \textbf{184} (2011), no. 1, 125--220.

\bibitem{XZZ_13}
L. Xu, P. Zhang, Z. Zhang.
Global solvability of a free boundary three-dimensional incompressible viscoelastic fluid system with surface tension,
\emph{Arch. Ration. Mech. Anal.} \textbf{208} (2013), no.3, 753--803.

\bibitem{ZZ}
P. Zhang, Z. Zhang.
On the free boundary problem of three-dimensional incompressible Euler equations.
\emph{Comm. Pure Appl. Math.} \textbf{61} (2008), no. 7, 877--940.

\end{thebibliography}
\end{document}